\documentclass[a4paper,11pt,twoside]{amsart}
\pdfoutput=1

\usepackage[utf8]{inputenc}

\usepackage{amsmath,amssymb,amsfonts}
\usepackage[retainorgcmds]{IEEEtrantools}

\usepackage[hyphens]{url} 
\usepackage[backend=biber,style=alphabetic,sorting=nyt,backref=true,hyperref=true,maxbibnames=10,maxcitenames=5]{biblatex}
\DeclareFieldFormat[article,book,incollection]{title}{\textit{#1}} 

\DeclareFieldFormat[article]{journaltitle}{\textit{#1}} 

\usepackage{epigraph}
\usepackage{graphicx,caption,subcaption}
\usepackage{tikz}

\usepackage[all,cmtip]{xy}

\usepackage{mathtools}

\newcommand{\term}[1]{\textit{\textbf{{#1}}}}
\usepackage{hyperref} 
\usepackage[normalem]{ulem} 

\usepackage{amsthm}
\theoremstyle{plain}
\newtheorem{theorem}{Theorem}[section]
\newtheorem{maintheorem}{Main Theorem}[section]
\newtheorem{corollary}[{theorem}]{Corollary}
\newtheorem{conjecture}[{theorem}]{Conjecture}
\newtheorem{lemma}[{theorem}]{Lemma}
\theoremstyle{definition}
\newtheorem{exer}[{theorem}]{Exercise}
\theoremstyle{remark}
\theoremstyle{remark}
\newtheorem{remark}[{theorem}]{Remark}
\newtheorem{example}[{theorem}]{Example}
\newtheorem{comment}[{theorem}]{Comment}

\makeatletter
\newcommand{\proofstep}[1]{%
  \par
  \addvspace{\medskipamount}
  \textit{#1\@addpunct{.}}\enspace\ignorespaces
}
\makeatother

\def\NN{\mathbb N} 
\def\ZZ{\mathbb Z} 
\def\QQ{\mathbb Q} 
\def\RR{\mathbb R} 
\def\CC{\mathbb C} 
\DeclareMathOperator{\Image}{Img} 
\def\from{\leftarrow}
\DeclareMathOperator{\Fill}{Fill}
\def\SS{\mathbb S}
\def\DD{\mathbb D}
\def\TT{\mathbb T}
\DeclareMathOperator{\Len}{Len}
\DeclareMathOperator{\Vol}{Vol}
\DeclareMathOperator{\Area}{Area}
\def\uHT{\mathrm{uHT}}
\newcommand{\diff}{\mathop{}\!\mathrm d} 
\def\Ste{\mathrm R} 
\def\Sho{\mathrm R^{\textrm{sh}}} 

\def\diskzerogon{\vcenter{\hbox{\includegraphics[height=5ex]{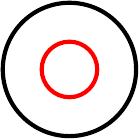}}}}
\def\diskempty{\vcenter{\hbox{\includegraphics[height=5ex]{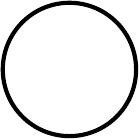}}}}
\def\diskmonogon{\vcenter{\hbox{\includegraphics[height=5ex]{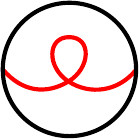}}}}
\def\diskline{\vcenter{\hbox{\includegraphics[height=5ex]{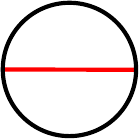}}}}
\def\diskbigon{\vcenter{\hbox{\rotatebox{90}{\includegraphics[height=5ex]{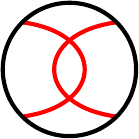}}}}}
\def\diskcrossing{\vcenter{\hbox{\includegraphics[height=5ex]{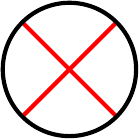}}}}
\def\diskuncrossedv{\vcenter{\hbox{\includegraphics[height=5ex]{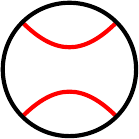}}}}
\def\diskuncrossedh{\vcenter{\hbox{\rotatebox{90}{\includegraphics[height=5ex]{disk_4_uncrossed}}}}}
\def\disktrigonup{\vcenter{\hbox{\includegraphics[height=5ex]{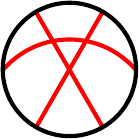}}}}
\def\disktrigondown{\vcenter{\hbox{\scalebox{1}[-1]{\includegraphics[height=5ex]{disk_triangle}}}}}

\DeclareMathOperator{\Conv}{Conv}
\DeclareMathOperator{\Star}{St}
\DeclareMathOperator{\mincr}{mincr}
\DeclareMathOperator{\ncr}{cr}
\DeclareMathOperator{\minlen}{minlen}
\newcommand{\commout}[1]{} 

\title[Discrete surfaces with length and area]{Discrete~surfaces~with~length~and~area and~minimal~fillings~of~the~circle}

\author{Marcos Cossarini}
\thanks{Supported for this work by a PhD scholarship from the CNPq (Brazil), then at the LAMA (Univ. Paris-Est, France) partially supported by the B\'ezout Labex (ANR-10-LABX-58) and the ANR project Min-Max (ANR-19-CE40-0014), now at the EPFL (Switzerland).}
\address{
Section de Math\'ematiques, Ecole Polytechnique F\'ed\'erale de Lausanne, station 8, 1015 Lausanne, Switzerland.} 
\email{marcos.cossarini@epfl.ch}

\begin{document}

\begin{abstract} 
We propose to imagine that every Riemannian metric on a surface is discrete at the small scale, made of curves called walls. The length of a curve is its number of wall crossings, and the area of the surface is the number of crossings of the walls themselves. We show how to approximate a Riemannian (or self-reverse Finsler) metric by a wallsystem. 

  This work is motivated by Gromov's filling area conjecture (FAC) that the hemisphere minimizes area among orientable Riemannian surfaces that fill a circle isometrically. We introduce a discrete FAC: every square-celled surface that fills isometrically a $2n$-cycle graph has at least $n(n-1)/2$ squares. We prove that our discrete FAC is equivalent to the FAC for surfaces with self-reverse metric.

  If the surface is a disk, the discrete FAC follows from Steinitz's algorithm for transforming curves into pseudolines. This gives a new proof of the FAC for disks with self-reverse metric. We also imitate Ivanov's proof of the same fact, using discrete differential forms. And we prove that the FAC holds for M\"obius bands with self-reverse metric. For this we use a combinatorial curve shortening flow developed by de Graaf--Schrijver and Hass--Scott. With the same method we prove the systolic inequality for Klein bottles with self-reverse metric, conjectured by Sabourau--Yassine.

  Self-reverse metrics can be discretized using walls because every normed plane satisfies Crofton's formula: the length of every segment equals the symplectic measure of the set of lines that it crosses. Directed 2-dimensional metrics have no Crofton formula, but can be discretized as well. Their discretization is a triangulation where the length of each edge is 1 in one way and 0 in the other, and the area of the surface is the number of triangles. This structure is a simplicial set, dual to a plabic graph. The role of the walls is played by Postnikov's strands.
\end{abstract}

\maketitle


\clearpage
\thispagestyle{empty}
\null
\newpage
\thispagestyle{empty}
\section*{About this text}

This text is similar to the document that I presented as PhD thesis at IMPA, Brazil, first in September 2018 and then revised in January 2019. For this version I have added the present page and made a few corrections (in particular, in Figure~\ref{fig:fine_hemisphere}).

\section*{Thanks and dedication}

I express my gratitude to those who helped doing this work. To the staff at IMPA who maintained the working conditions. To my fellow students who enriched life there in many ways. To my supervisor Misha Belolipetsky who taught me, listened and gave me his advise. To the jury members who read and made suggestions for improving the text. To Eric Biagoli who helped me make some computer programs. To Stéphane Sabourau, Alfredo Hubard, Eric Colin de Verdière and Arnaud de Mesmay for conversations and encouragement. To Sergei Ivanov, Francisco Santos, Stefan Felsner, Richard Kenyon, Guyslain Naves, Dylan Thurston, Pierre Dehornoy and Hsien-Chih Lin for answering questions or discussing via email.

I also thank warmly those that accompanied me personally during this journey, including my parents and family, and Lucas, Dani, Vero, Gaby, Marian, Sole, Pili, Julián, Eric, Yaya, Miguel, Cata, Gabi, Diego, Silvia, Leo, Tom, Gonzalo and Marzia.

Finally, I thank all the teachers that I had up to now. This text is dedicated to them.

\newpage

\tableofcontents

\newpage

\section{Overview}

In this section we describe the content, main theorems and structure of this thesis. The motivation and context of this work are presented later, in Sections~\ref{sec:from_scattering_to_FAC} and~\ref{sec:integral_geometry}. All definitions and theorems presented here will be repeated later, with more detail. 

\subsection{The filling area of the circle}

Let $C$ be a Riemannian circle, that is, a closed curve of a certain length. A \term{filling} of $C$ is a compact surface $M$ with boundary $\partial M=C$. A \term{Riemannian filling} of $C$ is a filling with a Riemannian metric. It is called an \term{isometric filling} if $d_M(x,y)=d_C(x,y)$ for every $x,y\in C$, where the \term{distance} $d_X(x,y)$ on any space $X$ is the infimum length of paths from $x$ to $y$ along that space $X$. For example, a Euclidean flat disk fills its boundary non-isometrically, but the Euclidean hemisphere is an isometric filling.

The \term{Riemannian filling area} of the circle is the infimum area of all Riemannian isometric fillings. Gromov \cite{gromov1983filling} posed the problem of computing the Riemannian filling area of the circle, proved the minimality of the hemisphere among Riemannian isometric fillings homeomorphic to a disk, and conjectured its minimality among all orientable Riemannian isometric fillings; this is the (Riemannian) \term{filling area conjecture}, or \term{FAC}. For orientable Riemannian isometric fillings of genus $\leq 1$ the Riemannian FAC is known to hold also; the hemisphere is minimal in this class as well \cite{bangert2005filling}.

Computing this single number, the Riemannian filling area of the circle, may seem a very specialized task, but it is related to the problem of determining a whole region of space based on how geodesic trajectories are affected when they cross the region, a problem known as ``inverse geodesic scattering''. In the introductory Section~\ref{sec:from_scattering_to_FAC} we discuss this connection and what is known about the FAC.

\subsection{Fillings with Finsler metric} Ivanov and Burago extended the filling area problem by admitting as fillings surfaces with Finsler metrics, and proved the minimality of the hemisphere among Finsler disks \cite{ivanov2001two,ivanov2011filling,burago2002asymptotic}. Roughly speaking, a \term{Finsler surface} is a smooth surface $M$ with a \term{Finsler metric} $F$, that continuously assigns to each tangent vector $v$ a length $F(v)\geq 0$, also denoted $\|v\|$. The full definition is given in Section~\ref{sec:Finsler}. A Finsler metric $F$ enables one to define the length of any differentiable curve $\gamma$ in $M$ by the usual integration $\Len_F(\gamma):=\int\|\gamma'(t)\|\,\diff t$. Finsler metrics are more general than Riemannian metrics because they need not satisfy the Pythagorean equation at the infinitesimal scale. More precisely, the Finsler metric restricted to the tangent plane at each point is a norm that need not be related to an inner product. In this thesis, norms, Finsler metrics and distance functions are in principle \term{directed}, not \term{self-reverse}; this means that we may have $\|-v\|\neq\|v\|$ and $d(x,y)\neq d(y,x)$.\footnote{Self-reverse Finsler metrics are called ``reversible'' or ``symmetric'' by other authors.}

For the filling area problem, smooth Finsler surfaces are equivalent to \term{piecewise-Finsler surfaces} (triangulated surfaces with a Finsler metric on each face) and are also equivalent to \term{polyhedral-Finsler surfaces} (surfaces made of triangles cut from normed planes), in the sense that when we fill isometrically the circle, the infimum area that can be attained with isometric fillings of any of the three kinds is the same; this is proved in Theorem~\ref{thm:equivalence_smooth_polyhedral}.

The definition of area of a Finsler surface that we use here is due to Holmes--Thompson. It involves the standard symplectic form on the cotangent bundle $T^*M$, which is related to the Hamiltonian approach to the study of geodesic trajectories. However, it is easy to define the \term{Holmes--Thompson area} when the Finsler surface is polyhedral: the area of a triangle cut from some normed plane is its usual (Cartesian) area in any system of linear coordinates $x_i$ on the plane, multiplied by the area of the dual unit ball in the dual coordinates $p_i$. Note that we use a non-standard normalization: the Holmes--Thompson area of a Euclidean triangle (or any Riemannian surface) is $\pi$ times greater than the usual value.

\begin{conjecture}[continuous FAC for surfaces with self-reverse Finsler metric] A surface with self-reverse Finsler metric that fills isometrically a circle of length $2L$ cannot have smaller Holmes--Thompson area than a Euclidean hemisphere of perimeter $2L$.
\end{conjecture}

I do not actually conjecture that this proposition is true, however, I state it as a conjecture rather than as a question (what is the infimum area of a Finsler isometric filling of the circle?) to make it analogous to Gromov's filling area conjecture, which is restricted to Riemannian surfaces and therefore has greater chances of being true.

\subsection{Discretization of self-reverse metrics on surfaces} The discretization that we propose for self-reverse 2-dimensional metrics is based on some facts of integral geometry that we review in Section~\ref{sec:integral_geometry}, namely, the formulas of Barbier and Crofton \cite{barbier1860note,crofton1868theory}. According to these formulas, when random lines are drawn on the Euclidean plane, the expected number of crossings of the lines with a given curve is proportional to the length of the curve, and the expected number of crossings between the lines in any given region is proportional to the area of the region. We will replace the random lines by a definite set of curves called walls, as follows.

We define a \term{wallsystem} $W$ on a compact surface $M$ as a 1-dimensional submanifold made of finitely many smooth compact curves called \term{walls}, that are \term{relatively closed} (either closed or with their endpoints on $\partial M$) and in general position (this means that the wallsystem has no tangencies with itself nor with the boundary $\partial M$, and its self-crossings are simple and in the interior of $M$). The pair $(M,W)$ is called a \term{walled surface}. A wallsystem $W$ defines a discrete metric, where the length of a smooth curve $\gamma$ in general position (that is, transverse to $W$, avoiding the self-crossings of $W$ and with no endpoints on $W$) is its number of crossings with $W$. The area $\Area(M,W)$ of the walled surface is the number of self-crossings of $W$. We also define the Holmes--Thomson area of the surface as the same number, multiplied by 4.


Let $(M,W)$ be a walled surface whose boundary $\partial M$ is a closed curve. The set $\partial W$ of endpoints of the walls can be used to define the length of a curve contained on the boundary $\partial M$ that is either closed or an arc with its endpoins not on $\partial W$. The length of such curve is the number of times that it crosses $\partial W$. The distance $d_{(\partial M,\partial W)}(x,y)$ between points $x,y\in\partial M\setminus\partial W$ is the minimum length of a curve from $x$ to $y$ along $\partial M$. The walled surface $(M,W)$ is called an \term{isometric filling} of its boundary if $d_{(M,W)}(x,y)=d_{(\partial M,\partial W)}(x,y)$ for every pair of points $x,y\in\partial M\setminus W$. It is easy to construct a walled surface $(M,W)$ that fills isometrically its boundary of length $2n$ by letting $M$ be the Euclidean hemisphere and letting the wallsystem $W$ consist of $n$ geodesics in general position. The area of this wallsystem is $\frac{n(n-1)}2$.

\begin{conjecture}[discrete FAC for walled surfaces]
Every surface with wallsystem $(M,W)$ that fills isometrically its boundary of length $2n$ has $\Area(M,W)\geq\frac{n(n-1)}2$.
\end{conjecture}

In Section~\ref{sec:equiv_discrete_continuous_FAC} we prove our main theorem:\footnote{Translations between discrete and continuous systolic inequalities have already been given in \cite{verdiere2015systolic,kowalick2013discrete,kowalick2015combinatorial}, however those translations are lossy: one may start with an optimal discrete inequality and its translation will be a non-optimal continuous inequality, and when one translates back to the discrete setting, the inequality obtained is weaker than the original.}

\begin{maintheorem} The discrete FAC for walled surfaces is equivalent to the continuous FAC for surfaces with self-reverse Finsler metric. Moreover, the equivalence holds separately for each topological class of surfaces.
\end{maintheorem}

Therefore, a proof of the FAC for walled surfaces of certain topology yields a proof of the FAC for continuous surfaces of the same topology with self-reverse Finsler metric.

In Section~\ref{sec:wallsystems_disk} we use classical combinatorics of curves on a disk to prove that the FAC for walled surfaces holds when the surface is topologically a disk, thus re-proving Gromov's theorem in a discrete way. The proof is based on Steinitz's algorithm, which employs certain elementary operations (similar to the Reidemester moves) that reduce the area without affecting boundary distances, until the wallsystem becomes a \term{pseudoline arrangement}, which means that walls are simple arcs with their endpoints on the boundary that cross each other at most once. We also discuss an alternative method due to Lins \cite{lins1981minimax}, who proved a discrete theorem that implies the Pu-Ivanov inequality, that is, the optimal systolic inequality for self-reverse Finsler metrics on the real projective plane \cite{ivanov2011filling}, originally proved in the case of Riemannian metrics by Pu \cite{pu1952some}.

In Section~\ref{sec:mobius_klein} we prove that the FAC for walled surfaces holds also in the simplest non-orientable case, when the surface is topologically a Möbius band. This implies that the continuous FAC holds for Riemannian or self-reverse Finsler metrics on the Möbius band, a fact which was not known before. (Gromov's conjecture was only stated for orientable Riemannian surfaces.)

\begin{maintheorem} The minimum Holmes--Thompson area of a Möbius band with self-reverse Finsler metric that fills isometrically its boundary of length $2L$ equals the area of the hemisphere of perimeter $2L$.
\end{maintheorem}

The proof is based on the work by Schrijver and de Graaf on the problem of routing wires of certain homotopy classes as edge-disjoint paths on a given Eulerian graph on a surface \cite{graaf1997making,graaf1997decomposition}. We proceed in three steps: first we close the Möbius band to obtain a Klein bottle, then we simplify the wallsystem using de Graaf-Schrijver's method and some further steps, and finally we solve a simple quadratic programming problem in four variables related to the lengths of the four homotopy classes of simple curves on the Klein bottle.

With the same method we also prove the optimal systolic inequality for Klein bottles with self-reverse Finsler metric, which was conjectured by Sabourau-Yassine in \cite{sabourau2016optimal}.

\begin{maintheorem} The minimum Holmes--Thompson area of a Klein bottle with self-reverse Finsler metric where every noncontractible curve has length $\geq L$ equals the area of the hemisphere of perimeter $2L$.
\end{maintheorem}

\subsection{Discrete FAC for square-celled surfaces} 
On a surface $M$, every wallsystem $W$ can be regarded as a graph where every crossing is a vertex of degree 4. It is convenient to restrict our study to wallsystems that are \term{cellular}; this means that every wall has crossings (with itself or other walls), and that $W$ divides the surface $M$ into regions that are homeomorphic to either the plane or the closed half-plane. If $W$ is cellular, then the dual graph decomposes the surface into square cells. Therefore a surface with a cellular wallsystem is equivalent to a combinatorial \term{square-celled surface}, that is, a surface made of squares that are glued side-to-side. The discrete FAC for surfaces with wallsystems can then be restated in terms of square-celled surfaces.

Let $C=C_{2n}$ be the cycle graph of length $2n$. A square-celled surface $M$ fills $C_{2n}$ isometrically if $\partial M=C$ and $d_M(x,y)=d_C(x,y)$ for every two vertices $x,y\in C$, where distances $d_M(x,y)$ between vertices $x,y$ of the square-celled surface $M$ are measured along its skeleton graph. The following conjecture is equivalent to the FAC for walled surfaces.

\begin{conjecture}[discrete FAC for square-celled surfaces]
Every square-celled surface $M$ that fills isometrically a cycle graph of length $2n$ has at least $\frac{n(n-1)}2$ square cells.
\end{conjecture}

We give an independent proof of the FAC for square-celled disks using discrete differential forms, imitating the proof by Ivanov \cite{ivanov2011filling} where differential forms are employed in the continuous setting.

\subsection{Directed metrics on surfaces and their discretization} Directed metrics on surfaces in general do not satisfy the Crofton formula, so they cannot be discretized using wallsystems, not even with co-oriented walls (that contribute to the length of a curve only when crossed in one direction). However, there is a way to discretize them using simplicial sets with their natural ``fine metric''. We introduce it in Section~\ref{sec:fine}. This section is only a sketch and the theorems are stated without full proofs.

A \term{fine triangle} or \term{ordered triangle} is a triangle with vertices $a_0$, $a_1$, $a_2$ whose side $[a_i,a_j]$ is directed in the direction $a_i\to a_j$ if $i<j$. A regular \term{fine surface} is a combinatorial surface made by gluing fine triangles side to side, matching the orientations of the sides upon gluing. Note that on a regular fine surface there is a directed graph formed by the sides of the triangles as edges. Therefore, a regular fine surface can also be described as a surface with a directed graph embedded in it, such that it divides the surface into ordered triangles. Regular fine surfaces are a special class of simplicial sets.

The \term{fine metric} on a fine surface assigns lengths $f(v)=1$ and $f(-v)=0$ to each directed edge $v$ and its reverse $-v$, respectively. The area of the fine surface is defined as the number of fine triangles, and the Holmes--Thompson area is the same number divided by 4.

Our main result regarding directed metrics (whose proof is not given here) is that every directed Finsler metric on a surface can be replaced by a fine structure whose discrete distances, lengths and Holmes--Thompson area, when scaled down by an appropriate factor, approximate as precisely as desired the corresponding values determined by the original continuous Finsler metric.

Fine structures on a disk are dual to trivalent perfectly oriented plabic graphs, that are associated to cells of a certain decomposition of the totally nonnegative Grassmanian. Postnikov \cite{postnikov2006total} introduced plabic graphs and showed how to reduce them, just as square-celled disks are reduced by Steinitz's algorithm. With this tool we show that to fill isometrically a fine cycle $C=C_{a,b}$ (that has length $a$ in one direction and $b$ in the other) we need exactly $2ab-a-b$ triangles, if we restrict to fillings homeomorphic to a disk. The \term{fine filling area} is the minimum area over all fine isometric fillings, not necessarily homeomorphic to a disk. The proposition that the fine filling area of $C_{a,b}$ is $2ab-a-b$ is the \term{discrete FAC for directed metrics}, or \term{fine FAC}. The fine FAC is equivalent to the continuous FAC for directed Finsler metrics.

With this discrete theory we can also prove the optimal systolic inequality for projective planes: a Finsler projective plane of systole $L$ cannot have less area than a hemisphere of perimeter $2L$, and the equality is attained by the standard Riemannian metric of constant curvature. This inequality was discovered and proven by Pu \cite{pu1952some} for Riemannian metrics, and was extended by Ivanov to self-reverse Finsler metrics \cite{ivanov2011filling}. Our contribution is to remove the restriction that the Finsler metric be self-reverse.


\subsection{Work to do next}
In Section~\ref{sec:todo} we discuss some possible next steps of this investigation.

\subsubsection{Computing filling areas by integer linear programming} Every even square-celled isometric filling of $C_{2n}$ can be mapped into the injective hull of $C_{2n}$, which is a finite graph. In consequence, every filling becomes a linear combination of 4-cycles of this graph, which allows us to express the filling area problem as an integer linear programming problem. We show that the problems of finding minimal oriented fillings and minimal unoriented fillings are higher-dimensional versions of two classical optimization problems: optimal transportation and optimal matching.

\subsubsection{Discretization of 3-dimensional metrics}
We propose to try using fine tetrahedra to discretize 3-dimensional Finsler metrics. The simplest case are integral (semi)norms on $\RR^3$, and we conjecture that they can be replaced by $\ZZ^3$-periodic fine structures. The dual unit ball of an integral seminorm is an integral polyhedron $K$. The problem of discretization is open already in the case when $K$ is the cube $[0,1]^3$: the best discretization found so far has 40 tetrahedra per period while the conjectured minimum number is 36.

\subsubsection{Poset of minlength functions and tightening algorithms} On a surface $M$ that has a length metric (discrete or continuous), let $\gamma$ be a compact curve that is relatively closed (either closed or with its endpoints on the boundary $\partial M$). The \term{minlength} of $\gamma$ is the infimum length of all curves that are homotopic to $\gamma$.

On a fixed topological surface, we can order the even wallsystems by their \term{minlength function}, also called marked length spectrum. We list the cases where the poset is known. For example, on the disk of perimeter $2n$, the poset contains as an interval the Bruhat poset of permutations of $n$ ordered items. On the torus, the poset elements are equivalent to symmetric integral polygons, ordered by containment.

We also discuss the computational problems of tightening a wallsystem and comparing two wallsystems. Finally, we mention the connection of these posets with the cut-flow duality for the problem of routing wires (finding edge-disjoint paths of certain homotopy classes) along an Eulerian graph on a surface.

\subsubsection{Random discrete surfaces}
We propose to study random discrete disks with given boundary distances, generalizing the work on random lozenge tillings. We conjecture that a simple Riemannian disk does not change significantly when it is discretized and randomized, and that two simple Finsler disks that have the same boundary distances become similar after they are discretized and randomized.

\subsection{Possibly unconventional terminology}
\begin{itemize}
\item Zero is considered a natural number. To number $n$ objects we will generally use the indices $0,\dots,n-1$.
\item Distances and norms in general are not self-reverse; they do not satisfy $d(x,y)=d(y,x)$ and $\|-v\|=\|v\|$. A norm or metric may be sometimes called a \term{directed norm} or \term{directed metric} (as in the term ``directed graph'') to emphasize that it need not be self-reverse.
\item Areas denoted $\Area_\uHT$ are $\pi$ times greater than the usual value.
\item A \term{graph} may have loop edges and multiple edges connecting the same pair of vertices.
\end{itemize}

\newpage

\section{Finsler manifolds and their Holmes--Thompson volume}\label{sec:Finsler}

Before defining Finsler metrics, we set some terminology and notation for manifolds and norms.

In this thesis, every \term{manifold} $M$ is topologically Hausdorff, with countable basis, and of class $C^k$ for some $k=0,1,\dots,\infty$. Its boundary, possibly empty, is denoted $\partial M$. If $k=0$, then $M$ is just a \term{topological manifold}. A manifold or map will be called \term{differentiable} if it is of class $C^k$ for some $k\geq 1$, and will be called \term{smooth} if it is of class $C^\infty$. If $M$ is a differentiable manifold, then the tangent space at a point $x\in M$ is denoted $T_xM$ and the tangent bundle $TM=\bigcup_{x\in M}T_xM$ is a $C^{k-1}$ manifold. Examples of manifolds are the \term{closed unit ball} $\DD^n=\{x\in\RR^n:\sum_ix_i^2\leq 1\}$, the $n$-\term{sphere} $\SS^n=\partial\DD^{n+1}=\{x\in\RR^{n+1}:\sum_ix_i^2=1\}$, the \term{real projective space} $\RR P^n=\SS^n/(x\sim -x)$, and the $n$-\term{torus} $\TT^n=\RR^n/\ZZ^n$.

A \term{surface} is a 2-manifold, and a \term{curve} is a 1-manifold that has \emph{exactly one connected component}. A compact curve is either a \term{path} (or \term{arc}), homeomorphic to the closed interval $[0,1]$, or a \term{closed curve}, homeomorphic to the 1-torus or the circle $\SS^1$. A $C^k$ \term{curve in} a topological space $M$ is a map from an interval $I\subseteq\RR$ or the 1-torus to $M$.

The \term{homotopy class} of a curve $\gamma$ is denoted $[\gamma]$, and we write $\gamma'\simeq\gamma$ if and only if the curves $\gamma$ and $\gamma'$ are homotopic. Unless we indicate otherwise, homotopies of curves \emph{must leave the endpoints fixed}.




A \term{norm} (or \term{directed norm}) on a real vector space $V$ is a function $\|-\|:V\to[0,+\infty)$ that is
\begin{itemize}
\item \term{subadditive}: $\|v+w\|\leq\|v\|+\|w\|$ for every $v,w\in V$
\item \term{scale covariant}: 
$\|\lambda v\|=\lambda\|v\|$ if $\lambda\geq 0$, $v\in V$, and
\item \term{positive definite}: $\|v\|>0$ if $v\neq 0$.
\end{itemize} A \term{seminorm} is a function $V\to[0,+\infty)$ that is subadditive and scale covariant; therefore a norm is a seminorm that is positive definite. A seminorm $\|-\|$ is \term{self-reverse} if $\|-v\|=\|v\|$ for every $v$. A seminorm in general need not be self-reverse, and may be called \term{directed} (as in ``directed graph'') just to emphasize this fact.

A \term{Finsler metric} on a differentiable manifold $M$ is a continuous function $F:TM\to[0,+\infty)$ 
whose restriction to the tangent space $T_xM$ at each point $x\in M$ is a directed norm $F_x:T_xM\to[0,+\infty)$. The pair $(M,F)$ is called a \term{Finsler manifold}. A Finsler metric $F$ (or the Finsler manifold $(M,F)$) is \term{self-reverse} if each norm $F_x$ is self-reverse.\footnote{Self-reverse Finsler metrics are called ``reversible'' or ``symmetric'' by other authors.}

The \term{length} of a vector $v\in T_xM$ tangent to a Finsler manifold $(M,F)$ is the number $F(v)=F_x(v)\geq 0$, that may also be denoted $\|v\|$ or $\|v\|_x$ if it is clear which metric $F$ should be employed. The \term{length} of a piecewise-$C^1$ curve $\gamma$ in a Finsler manifold $(M,F)$ is \[\Len_F(\gamma):=\int\|\gamma'(t)\|\,\diff t,\] and the \term{minlength} of any curve $\gamma$ is \[\minlen_F(\gamma):=\inf_{\gamma'\simeq\gamma}\Len_F(\gamma),\] where the infimum is taken over all piecewise-differentiable curves $\gamma'$ that are homotopic to $\gamma$. The \term{systole} of $(M,F)$ is the infimum of the lengths of noncontractible curves in $M$.

The \term{distance} $d_{(M,F)}(x,y)$ from a point $x\in M$ to a point $y\in M$ of a Finsler manifold $(M,F)$ is the infimum of the lengths of piecewise-$C^1$ curves that go from $x$ to $y$ along $M$. A curve from a point $x$ to a point $y$ \term{minimizes length} if its length equals the distance $d(x,y)$. We sometimes omit $F$ or $M$ and write $d_M$, $d_F$ or simply $d$ instead of $d_{(M,F)}$, and $\Len(\gamma)$ instead of $\Len_F(\gamma)$. Note that $d(x,y)\neq d(y,x)$ in general.

A Finsler metric $F$ on a manifold $M$ is \term{smooth} if $F(v)$ depends smoothly on $v$ as long as $v\neq 0$, and is \term{quadratically convex} if for each fixed $x\in M$ and for every two linearly independent vectors $v,w,\in T_xM$ the function $t\in\RR\mapsto F_x(v+tw)$ has strictly positive second derivative. If the metric is smooth and quadratically convex (or \term{smoothly convex}, for brevity), then every length-minimizing path $\gamma$ with constant speed $\|\gamma'\|$ along the interior of $M$ is a \term{geodesic}, that is, an extremal path $\gamma$ of the action functional $\gamma\mapsto\int\|\gamma'(t)\|^2\diff t$. Geodesics are necessarily smooth and satisfy the Euler-Lagrange equation, a differential equation of second order. 
Geodesics can also be described by Hamiltonian first-order differential equations on the $2n$-dimensional manifold $T^*M$, the cotangent bundle of $M$. The geodesic flow satisfies the Liouville theorem: it preserves the volume of sets in the cotangent bundle, as measured by the volume form $\left|\frac 1{n!}\omega^n\right|$, where $\omega^n=\omega\wedge\dots\wedge\omega$ is the $n$-th exterior power of the \term{standard symplectic form} $\omega$ on $T^*M$, a closed differential $2$-form that can be written in local coordinates as $\omega=\sum_{0\leq i<n}\diff x_i\wedge \diff p_i$. The coordinates $x_i$ may be any system of smooth coordinates $x_i$ on a region of $M$ and the coordinates $p_i$ on $T^*M$ are dual to the coordinates $\frac\partial{\partial x_i}$ on $TM$. We will mention Hamiltonian dynamics and the symplectic form in certain discussions, but our main theorems and proofs do not involve those structures. 

The area of a Finsler surface, and more generally, the volume of a Finsler $n$-manifold $(M,F)$, do not have an obvious definition; see \cite{paiva2004volumes} for an introduction to the subject. The problem is already clear when the Finsler $n$-manifold $(M,F)$ is a piece of $n$-dimensional normed space $(V,\|-\|)$, that is, when $M\subseteq V$ and $F_x=\|-\|$ for every $x\in M$. A definition of Finsler $n$-volume should allow us to measure subsets of different $n$-dimensional normed spaces using the same unit of volume.

A \term{Finsler $n$-volume function} is a function $\Vol_n$ that assigns a volume $\Vol_n(A,F)\geq 0$ to each Borel subset $A\subseteq M$ of every $n$-dimensional Finsler manifold 
$(M,F)$, and has the following properties:
\begin{itemize}
\item Restricted to the Borel sets of each Finsler manifold $(M,F)$, the function $\Vol_n(-,F)$ is a locally finite, regular Borel measure.
\item $\Vol_n$ is monotonic in the sense that if $f:(M,F)\to(M',F')$ is a length-decreasing differentiable function, then $\Vol_n(f(A))\leq\Vol_n(A)$ for every compact set $A\subseteq M$.
\end{itemize}
It can be proved (using Lemma~\ref{thm:polyhed_approx}) that every Finsler $n$-volume function $\Vol_n$ is determined by its behaivor on each $n$-dimensional normed space $(\RR^n,\|-\|)$, where it is constant positive multiple of the Lebesgue measure. Also, one can prove that every Finsler $n$-volume, satisfies the homogeneity property $\Vol_n(-,\lambda F)=\lambda^n\Vol_n(M,F)$ for $\lambda\geq 0$. For more information on Finsler volumes see \cite{paiva2004volumes}.\footnote{Our definition of Finsler volume function is equivalent to the one in \cite{paiva2004volumes} except for the fact that we do not require the normalization condition that $Vol_n$ coincides with the usual Lebesgue measure on the Euclidean plane. In consequence, our volume functions are all the positive multiples of the volume functions of \cite{paiva2004volumes}.}

To define a Finsler volume function, it is then sufficient to decide the volume of one non-trivial subset (say, compact and with non-empty interior) of each $n$-dimensional normed space. The approach initially favored by Busemann \cite{busemann1950geometry} was to declare that the unit ball $B_{\|-\|}=\{x\in V:\|x\|\leq 1\}$ of any $n$-dimensional normed spaces $(V,\|-\|)$ should have the same volume. This leads to \term{Busemann--Hausdorff area}, called that way because it equals the $n$-dimensional Hausdorff measure of the Finsler surface, considered as a metric space. Busemann-Hausdorff area is not convenient for the filling area problem (see \cite{paiva2004volumes,paiva2006wrong}) and has been displaced by Holmes--Thompson area \cite{holmes1979ndimensional} (see also \cite{paiva2004volumes}), which is more useful for the filling area problem and systolic inequalities \cite{ivanov2001two,burago2002asymptotic,paiva2016isosystolic,sabourau2016optimal}.

The definition of Holmes--Thompson area is based on the dual unit ball
, that we shall discuss first.
Note that a norm in general cannot be expressed by finitely many coefficients, unlike a Euclidean metric, which can be given by the $n^2$ coefficients $g_{i,j}$. However, a norm $N$ on a space $V$ can be specified by its \term{dual unit ball} $B_N^*$, which is a compact convex subset of the dual space $V^*$ that defines the norm $N$ by the formula \[N(v)=\max_{\varphi\in B^*_N}\varphi(v).\] The dual unit ball $B^*_N$ is in turn determined by the norm $N$ according to the formula $B^*_N=\{\varphi\in V^*:\varphi\leq N\}$, where $\varphi\leq N$ means that $\varphi(v)\leq N(v)$ for every $v\in V$. The last two formulas give a bijective correspondence between norms $N$ and compact convex sets $B^*\subseteq V^*$ that contain the origin in their interior. This correspondence between norms and their dual unit balls is monotonic: if $N$ and $M$ are two norms, then $N\leq M$ if and only if $B^*_N\subseteq B^*_M$. Also, the sum $N+M$ is a norm whose dual unit ball is the Minkowski sum $B^*_N\oplus B^*_M$ of $B^*_N$ and $B^*_M$. Finally, if a collection of norms $N_i$ is bounded above (by certain norm $M$), then their supremum $\vee_i N_i:v\mapsto \sup_iN_i(v)$ is a norm whose dual unit ball is the convex hull of the union of the balls $B^*_{N_i}$. On a Finsler manifold $(M,F)$ it will often be convenient to specify each norm $F_x$ by its dual unit ball $B_x^*$ in the cotantent space $T_x^*M$, the dual of the respective tangent space $T_xM$.


The \term{(un-normalized) Holmes--Thompson volume} (or \term{uHT} volume, for short) $\Vol_\uHT(M,F)$ of a Finsler $n$-manifold $(M,F)$ is defined as the symplectic volume (that is, the volume according to the symplectic volume form $\left|\frac 1{n!}\omega^n\right|$) of the bundle of dual unit balls \[B^*M=\bigcup_{x\in M}B_x^*\subseteq T^*M.\] More concretely, if $M$ is a region of $\RR^n$, then \[\Area_\uHT(M,F)=\int_M\left|B_x^*\right|\,\diff\Vol(x),\] where $\left|B_x^*\right|$ is the usual Cartesian (or Lebesgue) volume of the dual unit ball $B_x^*\subseteq(\RR^n)^*=\RR^n$, and the volume differential $\diff\Vol(x)$ is also defined according to the Cartesian volume on $\RR^n$. This formula holds because the symplectic volume form is $\left|\frac 1n\omega^n\right|=|\diff x_0\wedge\dots\wedge\diff x_{n-1}\wedge\diff p_0\wedge\dots\wedge\diff p_{n-1}|$. In particular, if $M$ is a piece of normed plane, then the Holmes--Thompson area of $M$ is the Cartesian area of $M$ (in any system of linear coordinates $x_i$) multiplied by the Cartesian area of the dual unit ball (in the dual coordinates $p_i$). Note that the Holmes--Thompson area of any piece of Euclidean plane, or any Riemannian surface, is $\pi$ times greater than the usual value, because the dual unit ball is a round disk whose Cartesian area is $\pi$.

\begin{example}\label{ex:uHT_areas} Holmes--Thompson area of some surfaces:
\begin{itemize}
\item In the plane $\RR^2$ with $\ell_1$ metric, the uHT area of the unit square $Q=[0,1]^2$ is $\Area_\uHT(Q)=|Q|\times|B^*_{\ell_1}|=1\times 4=4$, because $B^*_{\ell_1}=B_{\ell_\infty}=[-1,1]^2$.
\item The uHT area of the Euclidean disk of perimeter $2L$ (and radius $r=\frac L\pi$) is $\pi r^2\times\pi=L^2$.
\item The uHT of the Euclidean hemisphere with the same perimeter is $2L^2$, twice the area of the disk.\footnote{The area of the Euclidean hemisphere was computed by Archimedes, in the paper ``On the Sphere and the Cylinder'' \cite{archimedes-225sphere}.}
\end{itemize}
\end{example}

The next example is a Möbius band with self-reverse Finsler metric that fills isometrically the circle. We will later show that it has minimum area among surface that have these properties.

\begin{example}[Self-reverse Finsler Möbius band that has minimum area among isometric fillings of the circle] From the $xy$ plane with $\ell_\infty$ metric, take a square $[0,L]\times[0,L]$ and glue each point $(0,y)$ of the left side to the point $(L,L-y)$ of the right side. We obtain a Möbius band that fills isometrically its boundary, of length $2L$. (Proof: Antipodal boundary points are of the form $p=(x,0)$ and $q=(x,L)$. Going from $p$ to $q$ costs at least $L$, either if we go directly along the square, or if we use transportation from one vertical side to the other, whose distances to the points $p$ and $q$ have sum equal to $L$.) The uHT area of this surface is $L^2\times|B^*_{\ell_\infty}|=2L^2$ because $|B^*_{\ell_\infty}|=|B_{\ell_1}|=2$.
\end{example}

Finally, we note that the volume of a 1-dimensional Finsler manifold has a reasonable value.

\begin{example} If $C$ is an unoriented curve with a Finsler metric $G$, then $\Vol_\uHT(C,G)=\Len_G(C^+)+\Len_G(C^-)$ where $C^+$ and $C^-$ are the two oriented versions of the curve $C$.
\end{example}

\newpage
\section{From geodesic scattering to the filling area conjecture}\label{sec:from_scattering_to_FAC}
\epigraph{Mathematics is the part of physics\\ were experiments are cheap.}{Vladimir I. Arnold \cite{arnold1998teaching}}

Can a region of a Riemannian space be known completely without entering it, just by recording the point and velocity of entry and exit of particles that travel through the region along geodesic trajectories? This task is called \term{inverse geodesic scattering}. In a similar vein, the problem of \term{travel-time tomography} asks to describe the region based on knowledge of the lengths of all geodesics that join each pair of boundary points. (See the surveys \cite{croke1991rigidity,ivanov2010volume,uhlmann2016journey}.)

More precisely, let the region $M$ be compact and smoothly bounded, which makes it a Riemannian manifold in its own right. The \term{scattering relation} of $M$ connects each entry-point-and-velocity $(x,v)$ to the exit-point-and-velocity $(y,w)$, provided that the geodesic that enters $M$ at the boundary point $x\in\partial M$ with initial velocity $v$ actually exits $M$ at some boundary point $y\in\partial M$
, and does not wander forever inside $M$. The scattering relation maps bijectively the set of inwards-pointing vectors to the set of outwards-pointing vectors if $M$ is a \term{simple Riemannian manifold}, that is, if \begin{itemize}
\item its boundary $\partial M$ is curved strictly inwards,\footnote{We mean that the second fundamental form of the boundary is strictly positive. In other words, whenever a particle moves along the boundary with non-zero speed, its acceleration vector points strictly inside the manifold $M$. The second fundamental form of the boundary gives the normal component of the acceleration as a quadratic function of the velocity, which is any vector tangent to the boundary.}
\item every two points $x,y\in M$ are connected by a unique geodesic trajectory $\gamma_{x,y}:[0,1]\to M$, whose initial velocity $v_{x,y}$ at $x$ and whose final velocity $w_{x,y}$ at $y$ both depend smoothly on $x$ and $y$,\footnote{The smooth dependence of $v_{x,y}$ with respect to $y$ is equivalent to the non-existence of conjugate points along geodesics, which is the hypothesis which appears instead in Burago--Ivanov's definition of simple manifolds \cite{burago2010boundary}.} and
\item $M$ has one connected component, necessarily diffeomorphic to a ball.
\end{itemize} It follows that the distance $d_M(x,y)$ between any two points $x,y\in M$, defined as the infimum length of the curves that go from $x$ to $y$ along $M$, equals the length of the geodesic $\gamma_{x,y}$, and therefore depends smoothly on $x$ and $y$, as long as $x\neq y$. On a simple Riemannian manifold, the distance between all pairs of boundary points contains the same information as the scattering map,
and it is believed sufficient to determine the manifold completely:

\begin{conjecture}[Michel \cite{michel1981rigidite}] Each simple Riemannian manifold $M_0$ is \term{boundary distance rigid}: any other Riemannian manifold $M$ that has the same boundary and the same distance between each pair of boundary points must be equivalent to $M_0$ by a length-preserving diffeomorphism that fixes the boundary points.
\end{conjecture} 

Simple manifolds are general in a certain sense: any point of a Riemannian manifold has a simple neighborhood. Thus Michel's conjecture would imply that one cannot modify the manifold in that region without affecting the scattering of geodesics that go through it. If one restricts to a smaller neighborhood that is almost flat, then Michel's conjecture is true, as was proved by Burago--Ivanov \cite{burago2010boundary}.
Michel's conjecture has also been proved in the cases when the simple manifold $M_0$ is two-dimensional \cite{pestov2005two}, has constant curvature or is locally symmetric; see \cite{burago2010boundary} for references.

An advantage of measuring boundary distances rather than lengths or velocities of all geodesics is that distances are not significantly affected by small deformations of $M$, and can still be defined if $M$ is another kind of space (non-Riemannian, possibly non-smooth) with a length metric. The downside of this stability is that small, practically imperceptible variations on the boundary distances may be the only indication of large changes in the manifold, that are often impossible to trace back. Indeed, one can construct a very different manifold $M$ whose boundary distances nearly coincide with those of a simple manifold $M_0$ as follows: start with a manifold $M$ that is very different from $M_0$ and has \emph{larger} boundary distances, and then bring the distances down to nearly the same values of $M_0$ by connecting distant points using narrow tubes.\footnote{If $\dim M=2$, a tube is a cylinder $[0,L]\times\SS^1$ that is attached to $M$ after cutting away two disks from $M$. If $\dim M\geq 3$, then instead of tubes one can insert in $M$, without changing the topology, ``wires'' that are curves along which the metric is low, isolated from $M$ by a region where the metric is high.} Also, one can add a big chamber connected to $M$ through a tube, thus increasing the volume without significantly affecting boundary distances. In view of these possibilities, what can we say for certain about $M$, from the purely metric (and possibly imprecise) data of boundary distances?

In the paper \textit{Filling Riemannian manifolds} \cite{gromov1983filling}, Gromov proposed to require only \emph{lower bounds} for the boundary distances, and to attempt to deduce from these a lower bound for the volume. Applying this idea to simple manifolds and combining with Michel's question, Burago and Ivanov conjectured:

\begin{conjecture}[\cite{burago2013area,ivanov2010volume}] Each simple Riemannian manifold $M_0$ is a \term{minimal filling}: every other Riemannian manifold $M$ with the same boundary as $M_0$ and non-smaller boundary distances has $\Vol(M)\geq\Vol(M_0)$. Moreover, if $\Vol(M)=\Vol(M_0)$, then $M$ is equivalent to $M_0$ by a length-preserving diffeomorphism that fixes the boundary (therefore $M_0$ is called a unique or \term{strict minimal filling}).
\end{conjecture}

More generally, let $C$ be a compact closed, possibly oriented manifold with an arbitrary distance function $d_C:C\times C\to[0,+\infty)$. A \term{filling without shortcuts} or \term{nonshortcutting filling} of $C$ is a Riemannian manifold $M$ that is oriented if $C$ is oriented,\footnote{Note that if $C$ is orientable, we may still decide not to give it an orientation, and therefore we allow non-orientable fillings.} and also compact (or at least complete\footnote{Not every closed manifold $C$ is boundary of a compact manifold, so Gromov admitted as filling a cone based on $M$. To get a smooth filling, the apex can be sent to infinity by increasing the metric near it; then one gets a complete metric on $M\times[0,+\infty)$.}), with boundary $\partial M=C$, and such that for each $x,y\in C$, the path-length distance $d_M(x,y)$ along $M$ is not smaller than $d_C(x,y)$. If the equality $d_M(x,y)=d_C(x,y)$ holds, then the surface $M$ is called an \term{isometric filling} of $C$. The infimum volume that a nonshortcutting filling $M$ can have is called the \term{filling volume} of $C$.\footnote{Compare with the definitions in \cite[\S 2]{gromov1983filling} and \cite[\S 2.2]{gromov1983filling}). In \S 2, Gromov allowed $C$ to be any $n$-dimensional pseudo-manifold (a chain of $n$-simplices whose $(n-1)$-dimensional faces cancel out); then a \term{filling} $M$ of $C$ is an $(n+1)$-chain whose boundary (the $n$-faces that that remain of the boundaries of the $(n+1)$-cells, after internal cancellations) is $C$. The ring of coefficients of the chain may be $\ZZ$ to represent an oriented manifold, or $\ZZ_2$ to represent a non-oriented manifold. In \S 2.2, Gromov considers the case in which $C$ is a closed manifold. Then the two definitions of filling (a chain or a complete Riemannian manifold) are equivalent for the purpose of defining filling volume. However, Gromov's definition does not allow non-orientable fillings of orientable manifolds, and at this point our definition differs.}

If $C$ is the boundary of a region in Euclidean space, with the Euclidean distance, then the region itself is a minimal filling, and this was proved by Gromov \cite{gromov1983filling} using Besicovitch's inequality \cite{gromov1983filling,besicovitch1952two}. Minimality can also be proved for regions in hyperbolic space, and Riemannian simple manifolds that are nearly flat or hyperbolic (see \cite{burago2010boundary,burago2013area} and the survey \cite{ivanov2010volume}). But if $C$ is itself a connected Riemannian manifold with its path-length distance, then no nonshortcutting filling has been found and proven minimal. If $\dim C=1$, then $C$ can only be a closed curve, characterized completely by its length. The Euclidean hemisphere is a nonshortcutting filling of its boundary circle, but nobody knows whether it is minimal.

\begin{conjecture}[Filling area conjecture (FAC), \cite{gromov1983filling}] An orientable Riemannian surface $M$ that fills without shortcuts a Riemannian circle of length $2L$ cannot have less area than a Euclidean hemisphere of perimeter $2L$.
\end{conjecture}

Although the hemisphere is not simple, it is the limit for $t\to 0$ of circular spherical caps $S^2_t=\{(x,y,z)\in\RR^3:x^2+y^2+z^2=1\text{ and }z>t\}$, that are simple for $t\in(0,1)$, so the filling area conjecture would follow from the (conjectured) filling minimality of simple manifolds. Note also that small spherical caps, with $t$ sufficiently near 1, are nearly flat and therefore minimal fillings according to \cite{burago2010boundary}. This suggests that the filling area conjecture may be true, otherwise there would exist a non-trivial number $t_0\in(0,1)$ such that the spherical cap $S_t$ is a minimal filling if and only if $t\geq t_0$.

Apart from proposing the conjecture,\footnote{In fact, Gromov conjectured that for any number $d\geq 1$ of dimensions, the $(d+1)$-dimensional hemisphere has minimum volume among the orientable Riemannian nonshortcutting fillings of the $d$-dimensional sphere, but so far no nonshortcutting filling of a closed Riemannian manifold has been proven minimal.} Gromov proved it under the restriction that $M$ is homeomorphic to a disk. The proof begins by gluing each boundary point of $M$ to its antipodal, thus obtaining a closed surface homeomorphic to the projective plane, whose \term{systole} (the infimum of the lengths of noncontratible curves) is $\geq L$. Reciprocally, any projective plane of systole $L$ can be cut open along a shortest noncontractible loop yielding a disk that fills isometrically its boundary of length $2L$. Then the FAC for Riemannian disks is equivalent to Pu's systolic inequality \cite{pu1952some}, which asserts precisely that a Riemannian real projective plane of systole $\geq L$ attains its minimum area if and only if it is round (obtained from a Euclidean sphere by identifying antipodal points). This shows that the hemisphere is the unique area minimizer among Riemannian disks that fill a circle of given length wihout shortcuts.\footnote{Although I find it convincing enough, this proof is not rigourous because when antipodal points of the circle are glued together, the resulting closed surface may not be smooth, so Pu's inequality cannot be applied directly. To prove the minimality of the hemisphere, it is sufficient to show that the metric can be smoothed in a way that changes the area and systole as little as desired; this has been stated in \cite{bangert2005filling} and will be done here in Appendix~\ref{sec:approx_smooth_polyhed}. The uniqueness claim may need more work to establish rigorously.} 
The proof of Pu's inequality relies, in turn, on the uniformization theorem for Riemann surfaces, which asserts that each Riemannian metric on a closed surface is conformally equivalent to a metric of constant curvature. 

About two decades later, Gromov's work on fillings of the circle was generalized in two ways.

In one direction, Bangert, Croke, Ivanov and Katz \cite{bangert2005filling} proved the FAC for Riemannian fillings of genus 1, homeomorphic to a torus with an open disk removed.\footnote{To the extent that I studied the paper \cite{bangert2005filling}, it seems to me that the minimum area is not attained: the only way in which the area of a Riemannian nonshortcutting filling of genus 1 can approach the area of the hemisphere is by topologically degenerating to a disk.} In this case, the gluing of antipodal boundary points yields a nonorientable closed surface whose oriented double cover has genus 2. The proof uses, again, the uniformization theorem, and also the fact that every Riemannian surface of genus 2 is hyperelliptic, that is, conformally equivalent to a double cover of a sphere ramified at some points. Hyperelliptic surfaces can be described as complex algebraic curves defined by an equation $y^2=h(x)$, where $h\in\CC[X]$ is a polynomial whose roots are the ramification points of the cover. This approach will not be continued here, but we mention these features to emphasize the complex (in the sense of complex numbers) geometry involved in it.

\subsection{Finsler filling area problem}
In another direction, Ivanov \cite{ivanov2001two,ivanov2011filling} found a new proof of the FAC for disks. His argument did not employ the uniformization theorem, and was instead based on the simple topological fact that any two curves on a disk must necessarily cross if their four endpoints are on the boundary and interlaced. Additionally, when Ivanov's proof is cast in the Hamiltonian setting (which means that geodesic flow takes place in the cotangent bundle, rather than in the tangent bundle), it becomes clear that the proof applies not only to Riemannian disks, but also to disks with Finsler metrics. Moreover, Ivanov found that if the area of a Finsler surface is defined in the appropriate way, using the symplectic definition due to Holmes--Thompson, then there are many Finsler disks that fill the circle without shortcuts and have the same area as the hemisphere \cite{ivanov2001two}, \cite[Rmk. 1.5]{burago2010boundary}; we will call them \term{Finsler hemispheres}.

\begin{theorem}[Ivanov \cite{ivanov2001two,ivanov2011filling}] Let $M_0$ be a disk with smoothly convex Finsler metric such that each geodesic segment in the interior of $M_0$ minimizes length. Then $M_0$ has minimum Holmes-Thompson area among Finsler disks $M$ that have the same boundary and the same or greater boundary distances as $M_0$.
\end{theorem}

The hypothesis of the theorem is satisfied by the Euclidean hemisphere, and by any simple Riemannian surface. The hypothesis is also satisfied by any smoothly-bounded simply-connected region of a surface that satisfies the hypothesis. Finally, the hypothesis is satisfied by any simple Finsler surface. A \term{simple Finsler manifold} is a smoothly convex Finsler manifold diffeomorphic to a closed ball such that:
\begin{itemize}
\item from any point $x\in M$ to any point $y\in M$ there is a unique, length minimizing geodesic, and
\item the last property still holds if the metric undergoes any sufficiently $C^\infty$-small perturbation.
\end{itemize}

\begin{remark} The hypothesis that every geodesic segment in the interior of $M_0$ minimizes length is necessary for the theorem to hold, because if the Finsler surface $M_0$ has a non-minimizing geodesic, then one can reduce the area of the surface without reducing boundary distances by diminishing the metric on a small open neighborhood of a vector tangent to that non-minimizing geodesic. This argument does not apply to Riemannian disks (for example, a neighborhood in the Euclidean sphere of half a great circle), because the unit balls of a Riemannian metric are required to be ellipses, and one cannot modify the value of a Riemannian metric only on a small neighborhood of a vector without breaking this condition.
\end{remark}

\begin{remark} The hypothesis that every geodesic segment in the interior of $M_0$ minimizes length implies that from each point $x\in (M_0)^\circ$ to each point $y\in (M_0)^\circ$, there cannot be more than one geodesic. Proof: Assume there are two different geodesics $\gamma$, $\widetilde\gamma$ from $x$ to $y$. The two geodesics arrive at $y$ with different velocity vectors. Extend $\gamma$ smoothly with a small geodesic segment $\delta$ from $y$ to a nearby point $z$. The resulting geodesic $\gamma*\delta$ must be a shortest path from $x$ to $z$, by the hypothesis. However, the path $\widetilde\gamma*\delta$ has the same length, and it cannot minimize length because it is broken at $y$.
\end{remark}

\begin{remark}\label{thm:boundary_dist_determines_area} Note that according to this theorem, if two Finsler disks $M_0$, $M_0'$ satisfy the hypothesis and have the same boundary and boundary distances, then their Holmes-Thompson areas must also coincide. This does not happen if we use Hausdorff area. In fact, any simple Riemanniann disk $M_0$ can be replaced by a non-Riemannian Finsler disk $M_0'$ that has the same boundary distances and satisfies the hypothesis, and this disk will have greater Hausdorff area than the Riemannian disk. Proof: The uHT areas of the two disks coincide by Ivanov's theorem. If the Hausdorff and uHT areas are normalized so that they coincide on Euclidean or Riemannian surfaces, then in general the Hausdorff area is greater than the uHT area, by the Blaschke-Santaló inequality, and in fact it is strictly greater for non-Riemannian surfaces (see \cite[Thm. 3.3]{paiva2004volumes}).
\end{remark}

\subsection{Rational counterexamples to the Finsler FAC}
Another discovery about Finsler fillings (by Burago--Ivanov \cite{burago2002asymptotic}) is that if the Holmes-Thompson area is used, then the filling area conjecture admits ``rational'' counterexamples: the flat Euclidean disk $M_0$ (or the hemisphere) can be replaced by a Finsler disk $M$ with non-smaller boundary distances, that has boundary $\partial M_0=k\,\partial M$ for some number $k\in\NN$ (meaning that the curve $\partial M_0$ wraps $k$ times around $\partial M_0$), but has $\Area_\uHT(M)<k\Area_\uHT(M_0)$. The area saving of Burago--Ivanov's counterexample is of about 5 parts in 10.000, obtained with values of $k$ around 5 to 10. (This magnitude is not shown in their paper, but I calculated it with a computer.) The counterexample is constructed inside a normed $\RR^4$ and builds on previous work by Busemann-Ewald-Shephard \cite{busemann1963convex}, based on the non-convexity of the Grassmanian cone generated by pure 2-vectors in $\RR^4$. The counterexample implies that if the FAC is true for Finsler surfaces, then it cannot be proved using calibrations, which was the method employed for proving minimality of nearly-flat Riemannian simple manifolds among \emph{Riemannian} fillings by Burago--Ivanov \cite{burago2010boundary,burago2013area}.\footnote{The concept of calibration is part of the classical theory of minimal surfaces in $\RR^3$; see \cite{colding2011course,morgan1988area}.} The example also shows that either the hemisphere is not a minimal Finsler filling, or the problem of minimal nonshortcutting Finsler fillings has an ``integrality gap''.


\newpage
\section{Length and area according to integral geometry}\label{sec:integral_geometry}

Integral geometry was created in response to Buffon's needle experiment (see the original articles \cite{barbier1860note,crofton1868theory,sylvester1890funicular} and the books \cite{santalo1976integral,blaschke1949vorlesungen}). Two basic formulas of the theory, due to Barbier \cite{barbier1860note} and Crofton \cite{crofton1868theory}, describe the length of a curve $\gamma$ and the area of a set $S$ in the Euclidean plane:
\begin{equation}\label{eq:crofton_length}\Len(\gamma)=\frac 12\int_{w\in W}\#(w\cap\gamma)\,\diff\mu(w),\end{equation}
\begin{align}\pi\Area(S)
&=\int_{w\in W}\Len(w\cap S)\,\diff\mu(w)\label{eq:barbier_area},
\end{align}
where $W$ is the topological space of lines in the Euclidean plane, and $\mu$ is the unique Borel measure on $W$ that is invariant by translations and rotations, normalized so that the set of lines that cross a unit segment has measure 1. More explicitly, if each line $w\in W$ is given by the equation $x\cos\theta+y\sin\theta=s$ (for some $\theta\in[0,\pi)$ and $s\in\RR$), then the measure $\mu$ is given by $\diff\mu=\diff\theta\,\diff s$. To check that this normalization is correct, apply the formulas when $\gamma$ is a circle and $S$ is a disk.

The Crofton formula for length~\eqref{eq:crofton_length} was extended by Blaschke \cite{blaschke1935integralgeometrie} to simple Finsler surfaces $(M,F)$ where the metric is smoothly convex and, most importantly, self-reverse. Instead of lines, the space $W$ consists of geodesics, measured using the symplectic form. In more detail, to define the $\mu$-measure of a set of geodesics, one must smoothly parametrize the set by giving each geodesic $w$ in the form $w=w_{(x,p)}$, where $x$ is a point on $w$ and $p$ is its momentum at that point (a unit covector $p\in T^*_xM$). Then the measure of the set is the integral of the form $|\omega|$ (the standard symplectic form taken in absolute value) over the 2-manifold of pairs $(x,p)$. The fact that this measure does not depend on the choice of $x$ along each geodesic $w$ follows from the fact that $\omega$ is invariant along the geodesic flow. Since Blaschke's paper is written in German, we give a proof of the Crofton formula~\eqref{eq:crofton_length}. For simplicity we assume that $\gamma$ is a geodesic segment $[a,b]$. The case in which $\gamma$ is piecewise-geodesic follows immediately, and the general case can be obtained by approximation.

\begin{theorem}[\cite{blaschke1935integralgeometrie}] If $a,b$ are points on a simple Finsler disk with self-reverse metric, then \[d(a,b)=\textstyle{\frac 12}\,\mu(W_{[a,b]}),\] where $W_{[a,b]}$ is the set of geodesics that cross the geodesic segment $[a,b]$.
\end{theorem}

\begin{proof} To compute $\mu(W_{[a,b]})$ we parametrize the set of geodesics $W_{[a,b]}$ by the point and momentum $(x,p)$ of crossing of each geodesic $w\in W_{[a,b]}$ with the segment $[a,b]$. We employ exponential coordinates $x=(x_0,x_1)$ such that the segment $[a,b]$ becomes part of the geodesic ray $x_1=0$, $x_0\geq 0$. Then at each point $x\in[a,b]$, the standard unit vector $e_0$ and covector $e^0$ have norm 1, and this implies that for each $p_0\in(-1,1)$, there are exactly two unit covectors of the form $p=(p_0,p_1)$. We discard the one that has $p_1<0$ to avoid counting geodesics twice, once with each orientation. To obtain the measure $\mu(W_{[a,b]})$ we must integrate the area form $|\omega|=|\diff x_0\wedge\diff p_0+\diff x_1\wedge\diff p_1|$ over the set of pairs $(x,p)$ with $x\in [a,b]$, $\|p\|=1$ and $p_1\geq 1$. The second term $\diff x_1\wedge\diff p_1$ of $\omega$ vanishes because $x_1$ is constantly $0$. To compute the first term we note that for each value of $x_0$, the number $p_0$ ranges in the interval $[-1,1]$. We conclude that $\mu(W_{[a,b]})=\int_{x\in[a,b]}\int_{p_0\in[-1,1]}|\diff x_0||\diff p_0|=2\,d(a,b)$.
\end{proof} 

The area formula~\eqref{eq:barbier_area} was extended as well, to simple Riemannian manifolds of all dimensions, by Santaló \cite{santalo1976integral,santalo1952measure,croke2011santalo}. According to Santaló's formula the total volume of an $n$-dimensional simple Riemannian manifold $M$ equals (in suitable units) the integral of the function $\gamma\mapsto\Len(\gamma)$ defined on the $(2n-2)$-dimensional space of complete geodesics $\gamma:I\to M$, measured using the $(n-1)$-th exterior power of the symplectic form. (This formula should also apply to simple Finsler manifolds but I am not aware of a published proof; see \cite{ivanov2013local}.) 
Santaló's formula has been employed in the problems of inverse scattering and minimal nonshortcutting fillings since the beginning, for example, in \cite{michel1978quelques,michel1981rigidite,gromov1983filling,croke1984sharp,croke2008synthetic,ivanov2013local} and many other works.

The Crofton formula was also studied in connection to Hilbert's fourth problem \cite{busemann1976problem,pogorelov1979hilbert,busemann1981review,paiva2003hilbert,szabo1986hilbert,alexander1978planes,alexander1988zonoid,papadopoulos2013hilbert}. As interpreted by Busemann, the problem asks to determine and study all \term{projective metrics}, that are those metrics defined on a convex subset of affine space $\RR^n$ (or defined on the projective space $\RR P^n$) that produce the usual topology and such that every straight segment is a shortest curve between its endpoints \cite{busemann1976problem,busemann1981review,alexander1988zonoid}. Busemann noted that projective metrics can be obtained from very general measures $\mu$ on the set $W$ of hyperplanes by the Crofton formula~\eqref{eq:crofton_length}. In the case of self-reverse metrics on 2-dimensional spaces, the reciprocal was proven by Pogorelov, Ambartzumian and Alexander (\cite{pogorelov1979hilbert,alexander1978planes,ambartzumian1976note}; see also \cite{alexander1988zonoid}): every projective metric can be obtained from a measure $\mu$ on the set of straight lines.

In fact, in Alexander's paper \cite{alexander1978planes} the shortest paths are not required to be really straight, but only pseudostraight. A \term{pseudoplane disk} 
is a pair $(M,W)$, where $M$ is a topological disk and $W$ is a family of curves $w\subseteq M$ called \term{pseudolines} with the following properties: each pseudoline is a simple path
 with its endpoints (and only its endpoints) on the boundary circle, and every two different points $x,y$ are joined by a unique pseudoline. Alexander proved that any ``pseudoprojective'' self-reverse distance function on $M$ (that produces the topology of $M$, and such that any pseudoline segment $[x,y]$ has length $d(x,y)$) must be given by the Crofton formula for some measure $\mu$ on $W$.

Alexander's theorem can be applied to simple Finsler disks with self-reverse metric, whose geodesics are a system of pseudolines. In this case we know by the theorem of Blaschke \cite{blaschke1935integralgeometrie} mentioned above that the measure $\mu$ is in fact given by the symplectic form. However, Alexander (see also \cite{ambartzumian1976note}) 
gives an alternative elementary method to determine the measure from the boundary distance function via the \term{funicular formula} that follows.\footnote{The formula seems to appear explicitly for the first time in the paper by Sylvester\cite{sylvester1890funicular} and we borrow from there the term ``funicular'', although it is just one of many interesting formulas in that paper that could deserve that name, and also, it is in fact a specialization of an earlier formula by Crofton for the measure of sets of lines that simultaneously intersect two convex figures on the plane. The term ``product of intervals'' is also taken from Sylvester's paper.} Let $[x,x']$ and $[y,y']$ be two \emph{non-overlapping} counterclockwise segments of the boundary circle $C=\partial M$. Then the $\mu$-measure of the set of pseudolines that go from $[x,x']$ to $[y,y']$ is\footnote{Assuming that the metric is obtained from a measure $\mu$ by the Crofton formula, the funicular formula is easy to prove by examining the contribution of each wall $w\in W$ to the right hand side $[x,x']\times_d[y,y']$, depending on whether it crosses zero, one or two of the intervals $[x,x']$,~$[y,y']$. The contribution is nonzero only in the last case.} \begin{equation}\label{eq:funicular}\mu\left(W_{[x,x']}\cap W_{[y,y']}\right)=d(x,y)+d(x',y')-d(x,y')-d(x',y),
\end{equation} where $W_A$ denotes the set of geodesics $w\in W$ that intersect some set $A$. Once we fix the boundary distance function $d$, note that the $d$-\term{product of intervals}\[([x,x'],[y,y'])\mapsto[x,x']\times_d[y,y']:=d(x,y)+d(x',y')-d(x,y')-d(x',y),\] defined for non-overlapping counterclockwise segments $[x,x']$, $[y,y']$ of the boundary circle, is always non-negative. (A distance function $d$ defined on a circle $C$ will be called \term{disk-like} if it has this property.) Note also that the product is additive with respect to each variable: if we break an interval $[y,y']=[y,y'']\cup[y'',y']$, then \[[x,x']\times_d[y,y']=[x,x']\times_d[y,y'']+[x,x']\times_d[y'',y'],\] and the same happens with the first variable. This enables proving that the funicular formula really defines a measure on the set $W$ of pseudolines, that produces the same boundary distances as $d$. (The trickier part of Alexander's proof is proving that, if the measure is constructed in this way, then the Crofton formula holds also for pseudoline segments that do not reach the boundary. But here we only care about boundary distances.)

The funicular formula was later used by Arcostanzo \cite{arcostanzo1992metriques,arcostanzo1994metriques} to construct a Finsler metric on a pseudoplane disk given the geodesics (a system of smooth pseudolines whose crossings are transverse) and the boundary distance (disk-like, with some differentiability and convexity). This implies that if $F$ and $F'$ are two self-reverse simple Finsler metrics on the disk, then one can create a hybrid Finsler metric that has the same geodesics as $F$ and the same boundary distances as $F'$. A similar construction was given for higher dimensional simple manifolds \cite{ivanov2013local} and also for closed surfaces \cite{bonahon1993surfaces}. The funicular formula was also applied by Otal \cite{otal1990longueurs,otal1990spectre} to prove the rigidity of negatively curved Riemannian disks (or closed surfaces) given their boundary distances (resp. marked length spectrum).

Additionally, in the same paper \cite{alexander1978planes}, Alexander defined the area of a set $S\subseteq D$ by the formula~\eqref{eq:barbier_area}, where $D$ is a simple Finsler disk with self-reverse metric. This area is the same as the Holmes--Thompson area, published about the same time by Holmes and Thompson \cite{holmes1979ndimensional}, but this fact was only noted almost 20 years later by Schneider-Wieacker \cite{schneider1997integral} (in the case of normed planes) and more generally by Álvarez-Paiva--Fernandes \cite{paiva1998crofton} (for projective metrics on the plane). 

\newpage
\section{Discrete self-reverse metrics on surfaces and the discrete filling area problem}\label{sec:discrete_FAC}


In this section we discretize the notion of self-reverse metric on a surface. The discretization is based on the Barbier--Crofton formulas that we discussed in the previous section. 

Imagine that you zoom into a surface $M$ with Riemannian or self-reverse Finsler metric $F$, and discover that the metric is determined by a system $W$ of curves called walls. The length of any curve in $M$ is its number of crossings with $W$, and the un-normalized Holmes--Thompson area of $M$ is four times the number of self-crossings of $W$. Therefore lengths and areas are not continuous-valued magnitudes but integers.

The factor 4 can be derived from the following example.


\begin{example}[Discretized Euclidean hemisphere]\label{ex:discretized_euclidean_hemisphere} If $M$ looks macroscopically like a Euclidean hemisphere, the walls could be a large number $n$ of geodesics, each obtained by intersecting with $M$ a random plane through the center of the sphere, as shown in Fig.~\ref{fig:hemisphere_walls}. If the planes are chosen independently and with uniform distribution (which means that each plane is given by a normal vector that is drawn randomly from the uniform distribution over the unit sphere), then the number of wall crossings of each smooth curve should be approximately proportional to its length (according to the Crofton formula
), so indeed we may expect to get an approximately Euclidean hemisphere in this way. The length of the boundary will be $2n$, and the number of crossings between walls will be $\frac{n(n-1)}2=\frac12n^2-\frac 12n$, since each pair of planes through the center crosses once in the hemisphere $M$, almost surely. Remembering that the Holmes--Thompson area of a Euclidean hemisphere with perimeter $2L$ is $2L^2$, and ignoring the smaller term $\frac 12n$, we see that the un-normalized Holmes--Thompson area of each crossing between walls should be 4.
\end{example}

\begin{figure}[h]
  \centering
  \includegraphics{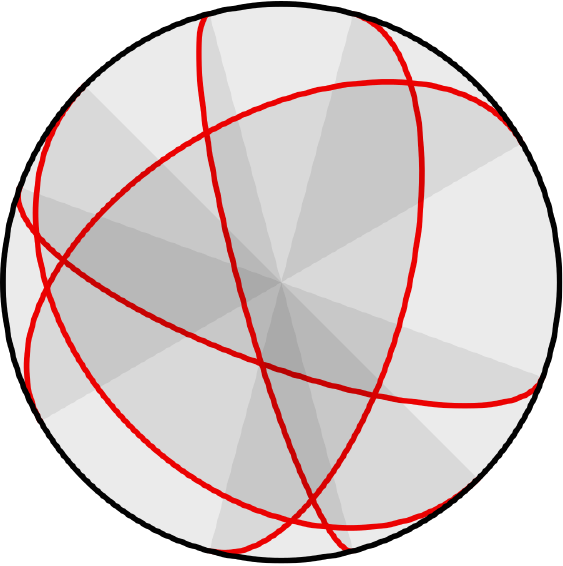}
  \caption{The Euclidean upper hemisphere $M$ seen from above, with some random geodesic red walls, each obtained by intersecting $M$ with a plane through the center.}\label{fig:hemisphere_walls}
\end{figure}

In this section we will define wallsystems and state the discrete FAC for surfaces with wallsystem. We will also define square-celled surfaces, and state a version of the discrete FAC for such surfaces. Finally, we will prove that these two versions of the discrete FAC are equivalent, by establishing a duality between wallsystems and square-celled decompositions of a surface. We leave for a later section 
the proof of the equivalence between these discrete FACs and the continous FAC for surfaces with self-reverse Finsler metrics.

\subsection{Wallsystems}

A \term{wallsystem} $W$ on a differentiable manifold $M$ is a differentiable submanifold of codimension 1 that is relatively closed ($\partial W\subseteq\partial M$), proper, immersed, and in general position. This means, if $M$ is a compact surface, that $W$ consists of finitely many compact curves (called \term{walls}) that intersect each other or self-intersect at finitely many points; each intersection is a simple, transverse crossing in the interior of $M$; and each wall is either a closed curve that avoids the boundary of $M$, or a path that avoids the boundary except at its endpoints, where it meets the boundary transversely. Note that the set $\partial W=W\cap\partial M$ of endpoints of the walls is a wallsystem on the boundary $\partial M$; it consists of an even number of points. A surface with a wallsystem is shown on Fig.~\ref{fig:wallsystem}

\begin{figure}
  \centering
  \includegraphics{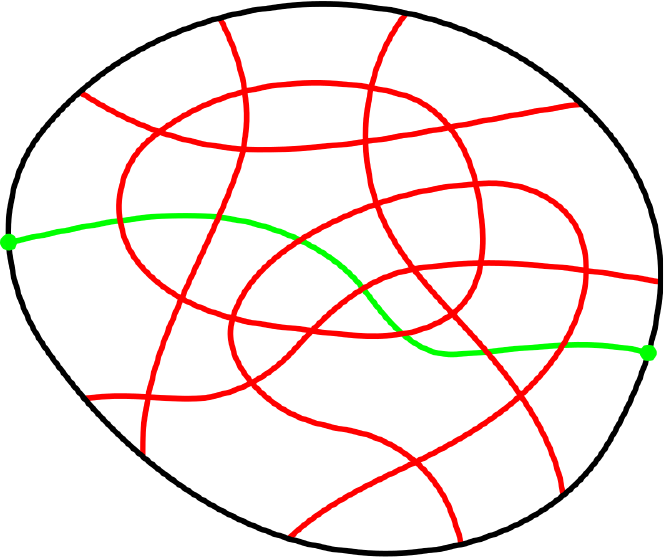}
  \caption{A surface $M$ with a red wallsystem $W$ with $\Area(M,W)=19$ and a green generic piecewise-differentiable curve $\gamma$ of length $\Len_W(\gamma)=7$.}\label{fig:wallsystem}
\end{figure}

A piecewise-differentiable compact curve in $M$ is \term{generic} or \term{in general position} with respect to the wallsystem $W$ if and only if it can be decomposed as a concatenation of finitely many differentiable paths that are transverse to $W$, avoid the self-crossings of $W$ and have no endpoints on $W$. 

The \term{length} of a generic piecewise-differentiable compact curve $\gamma:I\to M$ is defined as \[\Len_W(\gamma):=\#\{(t,s)\in I\times W:\gamma(t)=\iota(s)\},\] where $\iota$ is the immersion map of $W$ into $M$ and $\#A=|A|$ denotes the number of elements of a set $A$. The \term{minlength} of $\gamma$ is \[\minlen_W(\gamma):=\min_{\gamma'\simeq\gamma}\Len_W(\gamma'),\] where the minimum is taken over all the generic piecewise-differentiable curves $\gamma'$ that are homotopic to $\gamma$, with the endpoints fixed if $\gamma$ has endpoints. We also define the \term{distance} \[d_{(M,W)}(x,y):=\min_{\substack{\gamma\text{ curve in }M\\\text{ from }x\text{ to }y}}\Len_W(\gamma)\] for each pair of points $x,y\in M\setminus W$. Note that we may have $d(x,y)=0$ for $x\neq y$. Finally, we define the \term{area} of $(M,W)$ as \[\Area(M,W):=\,\#\big\{\{s,s'\}\subseteq W:\iota(s)=\iota(s')\text{ and }s\neq s'\big\}.\] (We use a set $\{s,s'\}$ rather than an ordered pair $(s,s')$ to avoid counting twice the same crossing.) We also define the \term{un-normalized Holmes--Thompson area} $\Area_\uHT(M,W)=4\Area(M,W)$, which will be useful later for passing to continuous surfaces. The pair $(M,W)$ is called a \term{walled surface}.

A walled surface $(M,W)$ \term{fills isometrically} its boundary if for every two points $x,y\in \partial M\setminus W$, we have \[d_{(M,W)}(x,y)=d_{(\partial M,\partial W)}(x,y).\] Note that the boundary $\partial M$ of an isometric filling must be a single curve (or empty!), and must have even length. Fixed the boundary length $2n$ we ask: how small can be the area of an isometric filling? An isometric filling whose area is $\frac{n(n-1)}2$ was given in Example~\ref{ex:discretized_euclidean_hemisphere} above, where $M$ is the Euclidean hemisphere and each wall is a geodesic that joins opposite points of the boundary. Is that filling minimal?

\begin{conjecture}[Discrete FAC for walled surfaces] Every walled surface $(M,W)$ that fills isometrically its boundary of length $2n$ has $\Area(M,W)\geq\frac{n(n-1)}2$.
\end{conjecture}

\subsection{Square-celled surfaces}

Even though walled surfaces have integer lengths and areas, they are not entirely discrete objects. However, they can be transformed to combinatorial square-celled surfaces which are defined as follows. (The transformation is described later.)

\subsubsection{Definition}
A compact \term{square-celled surface}\footnote{A square-celled surface is a kind of cube complex, that is, a space obtained from a set of cubes (of possibly different dimensions) by gluing faces using affine bijections. For more information about cube complexes see \cite{bridson1999metric,sageev2012cat,wise2012riches}.}
is a surface $M$ obtained from a finite set of squares by choosing some disjoint pairs of sides $\{e,e'\}$ (that may belong to the same square), and then gluing each such pair by an affine bijection $e\to e'$. Note that every space constructed in this way is locally homeomorphic to the closed half-plane, even at the vertices.

Let $M^0$ be the set of \term{0-cells} or \term{vertices} of $M$, that is, the vertices of the squares used to form $M$. Two vertices are consider equal if they are glued together in $M$. Let $M^1$ be the set of \term{1-cells} or \term{edges} of $M$, that is, the edges of the squares used to form $M$. Two edges are considered equal if they have the same image in $M$. Let $M^2$ be the set of squares used to form $M$, called \term{2-cells} or \term{square cells} of $M$. 

\begin{remark}\label{rmk:cells_not_subsets} Each $k$-cell of $M$ (for $k=0,1,2$) is a $k$-dimensional cube $Q\simeq[0,1]^k$ endowed with an \term{insertion map} $\iota_Q:Q\to M$ that may not be injective. Most of the time, however, we will regard the cell $Q$ as a subset of $M$ and avoid mentioning the map $\iota_Q$.
\end{remark}

The data $(M^0,M^1,M^2)$ is called a \term{square-celled decomposition} of $M$. Its \term{1-skeleton} is the graph $M^{\leq 1}=(M^0,M^1)$. (Recall that a graph may have loop edges and multiple edges joining the same pair of vertices.) The boundary of $M$ is a subgraph $C\subseteq M^{\leq 1}$, formed by the non-paired sides of the square cells. If the boundary is connected, then it is a cycle graph $C_{2n}$ of even length $2n$.

\begin{remark} The square-celled surface could be defined more formally as the 4-tuple $(M,M^0,M^1,M^2)$, but instead we will just refer to the topological space $M$ and regard the decomposition $(M^0,M^1,M^2)$ as implicit when we say that $M$ is a square-celled surface.
\end{remark}

\subsubsection{Discrete curves in a square-celled surface}
Each edge of a square-celled surface $M$ can be oriented in two ways, denoted $e$ and $-e$. The set of oriented edges, also called \term{directed edges}, will be denoted $\overrightarrow{M^1}$. Each directed edge $e\in\overrightarrow{M^1}$ has a startpoint $\partial_0(e)\in M^0$ and an endpoint $\partial_1(e)\in M^0$.

A \term{discrete curve} in a square-celled surface $M$ is a curve in the graph $M^{\leq 1}$ that is expressed as a concatenation $\gamma=e_0*\ldots*e_{n-1}$ of oriented edges $e_i\in\overrightarrow{M^1}$ such that $\partial_1{e_i}=\partial_0(e_{i+1})$ for every $i<n-1$. If the startpoint and endpoint of $\gamma$ coincide, then $\gamma$ may be considered as a \term{closed discrete curve}, and in this case the cyclically shifted expression $e_i*\ldots*e_{n-1}*e_0*\ldots*e_{i-1}$ represents the same curve as $\gamma$.

Two discrete curves on a square-celled surface $M$ are homotopic as continuous curves if and only if they are connected by a \term{discrete homotopy}, that is, a sequence of \term{elementary homotopies}, that are operations of one of the following kinds or their inverses:
\begin{itemize}
\item The $2\to 0$ homotopy: If $\gamma$ is of the form $\gamma=\gamma_0*e*(-e)*\gamma_1$, where $\gamma_0$ and $\gamma_1$ are discrete paths, we replace $\gamma$ by $\gamma'=\gamma_0*\gamma_1$.
\item The $2\to 2$ homotopy: If $\gamma=\gamma_0*v*w*\gamma_1$ where $\gamma_0$ and $\gamma_1$ are discrete paths and $Q=\vcenter{\xymatrix@C-1.3pc@R-1.3pc{
& \ar[dr]^{w} &\\
\ar[dr]_{w'} \ar[ur]^{v} &&\\
&\ar[ru]_{v'} & \\
}}\in M^2$, we replace $\gamma$ by $\gamma'=\gamma_0*w'*v'*\gamma_1$.
\end{itemize}
Note that in the case of a closed curve $\gamma$, we may apply these operations to the original expression $\gamma=e_0*\ldots*e_{n-1}$ or to any of its cyclic shifts $e_i*\ldots*e_{n-1}*e_0*\ldots*e_{i-1}$.

\subsubsection{Length and area on square-celled surfaces and the filling area problem}
If $\gamma=e_0*\ldots*e_{n-1}$ is a discrete curve in $M$, we define its \term{length} $\Len_{M^{\leq 1}}(\gamma)=n$ and its \term{minlength} \[\minlen_{M^{\leq 1}}(\gamma):=\min_{\gamma'\simeq\gamma}\Len_{M^{\leq 1}}(\gamma'),\] where the minimum is taken over discrete curves $\gamma'$ that are homotopic to $\gamma$. The \term{distance} $d_{M^{\leq 1}}(x,y)$ between two vertices $x,y\in M^0$ is the minimum length of a path from $x$ to $y$ along the graph $M^{\leq 1}$. The \term{area} of $M$ is $\Area(M):=|M^2|$, that is, the number of square cells of $M$, and we also define the \term{un-normalized Holmes--Thompson area} $\Area_\uHT(M):=4\,|M^2|$.

An \term{isometric filling} of the cycle graph $C=C_{2n}$ of length $2n$ is a compact square-celled surface $M$ with boundary $\partial M=C$ such that $d_{M^{\leq 1}}(x,y)=d_C(x,y)$ for every two boundary vertices $x,y\in M^0\cap\partial M$.

\begin{conjecture}[Discrete FAC for square-celled surfaces, or \term{square-celled FAC}] Every square-celled surface $M$ that fills isometrically a cycle graph of length $2n$ has at least $\frac{n(n-1)}2$ square cells.
\end{conjecture}


In the next two subsections we will prove the equivalence between the two versions of the discrete FAC, for walled surfaces and for square-celled surfaces.

\subsection{Duality between cellular wallsystems and square-celled decompositions of a surface}

Every square-celled surface $M$ has a \term{dual wallsystem} $W$ defined as follows. For each square cell $Q\simeq[0,1]^2$ of $M$, let \[W\cap Q=\left\{(x,y)\in Q\simeq [0,1]^2:x=\frac 12\text{ or }y=\frac 12\right\}.\] In principle $M$ is only a topological surface, not a smooth surface. However, we can make it smooth as follows. Let each square cell $Q\in M^2$ be a copy of the unit square $[0,1]^2\subseteq\RR^2$ with the standard Euclidean metric. Then the surface $M$ becomes locally isometric to the Euclidean plane, except at the vertices $x\in M^0$, where it has cone singularities. The cone singularities can be smoothed away. Then $M$ becomes a smooth Riemannian surface with a geodesic wallsystem $W$ whose crossings are orthogonal.

The wallsystem $W\subseteq M$ produced in this way is a \term{cellular wallsystem}, that is, a wallsystem such that:
\begin{itemize}
\item every wall of $W$ has at least one crossing (with itself or other walls), and
\item every connected component of $M\setminus W$ 
is homeomorphic to the plane or to a closed half-plane.
\end{itemize}

Reciprocally, for any cellular wallsytem $W$ on a smooth surface $M$ we can construct a dual square-celled decomposition as follows. Consider $W$ as an embedded graph $(W^0,W^1)$, whose vertex set $W^0$ consists of the self-crossings and endpoints of $W$, and whose edge set $W^1$ consists of the pieces into which $W$ is broken by $W^0$. We may construct the dual graph $W^*=(M^0,M^1)$, a graph smoothly embedded in $M$ characterized as follows:
\begin{itemize}
\item Each cell $U$ of $M\setminus W$ contains exactly one vertex $x\in M^0$, and if $U$ is a boundary cell, then $x$ is on the boundary $\partial M$.
\item Each edge $e\in W^1$ is intersected, transversely and exactly once, by exactly one edge $e^*\in M^1$, and if $e$ ends at the boundary $\partial M$, then $e^*$ is a piece of the boundary.
\end{itemize}
The dual graph $W^*$ is unique up to isotopies of $M$ that leave $W$ fixed. Each self-crossing of $W$ is enclosed by a 4-cycle of $M^1$, which bounds a square cell. These cells form a set $M^2$ that completes a square-celled decomposition $(M^0,M^1,M^2)$ of $M$, called a \term{dual square-celled decomposition} of the cellular wallsystem $W$. 

\begin{figure}
  \centering
  \includegraphics{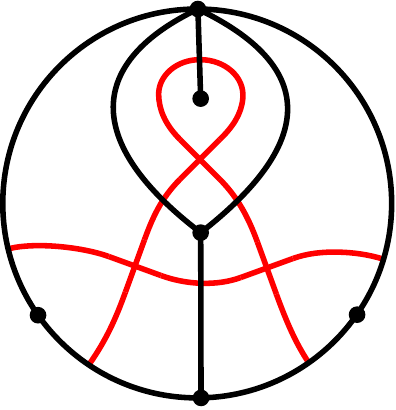}
  \caption{A disk with a red cellular wallsystem and its dual square-celled decomposition.}\label{fig:wallsystem_dual}
\end{figure}

On each compact surface $M$, the duality described above is a bijection between square-celled decompositions and cellular wallsystems, both considered up to isotopy. A square-celled decomposition and its dual cellular wallsystem have the same lengths and areas:

\begin{lemma} Let $M$ be a square-celled surface and let $W$ be the dual cellular wallsystem. Then
\begin{itemize}
\item $|M^2|=\Area(M,W)$.
\item For any two vertices $x,y\in M^0$, we have $d_{M^{\leq 1}}(x,y)=d_W(x,y)$,
\item and if $x,y\in M^0\cap\partial M$, then $d_C(x,y)=d_{(\partial M,\partial W)}(x,y)$, where $C\subseteq M^{\leq 1}$ is the boundary subgraph.
\item Every compact curve in $M$ that is closed or has its endpoints in $M^0$ is homotopic to a cycle or path $\gamma$ in the graph $M^{\leq 1}$, and $\minlen_{M^{\leq 1}}(\gamma)=\minlen_W(\gamma)$.
\end{itemize}
\end{lemma}

\begin{proof} The equality between $|M^2|$ and $\Area(M,W)$ is clear from the construction of the square-celled decomposition.

Regarding lengths, each path or cycle $\gamma$ in the graph $M^{\leq 1}$ can be considered as a piecewise-differentiable curve of the same length in the walled surface $(M,W)$. Reciprocally, each genereic piecewise-differentiable compact curve $\gamma$ in $M$ that is closed or has its endpoints in $M^0$ is homotopic to a cycle or path $\overline \gamma$ in the graph $M^{\leq 1}$, that visits the same cells of $M\setminus W$ and crosses the same edges of $W^1$ in the same order as $\gamma$, and therefore has the same length: $\Len_{M^{\leq 1}}(\overline\gamma)=\Len_W(\gamma)$. This proves the last assertion, from which the second one follows.

The third assertion has a similar proof. Along the boundary $\partial M$, for every two vertices $x,y\in M^0\cap\partial M$ we have $d_C(x,y)=d_{\partial W}(x,y)$ because every path from $x$ to $y$ along the boundary graph $C$ can be considered as a path of the same length in the walled curve $(\partial M,\partial W)$, and every generic piecewise-differentiable curve $\gamma$ from $x$ to $y$ along $\partial M$ is homotopic to a path $\overline\gamma$ on the graph $C$ that crosses the same points of $\partial W$ as $\gamma$, in the same order and direction.
\end{proof}

From the lemma we conclude the following.

\begin{theorem}\label{thm:equiv_wallFAC_squareFAC} The discrete FAC for walled surfaces, restricted to \emph{cellular} wallsystems, is equivalent to the discrete FAC for square-celled surfaces. More precisely:
\begin{itemize}
\item If a compact square-celled surface $M$ fills isometrically a $2n$-cycle, and $W$ is the dual wallsystem, then the cellular walled surface $(M,W)$ fills isometrically its boundary of length $2n$ and has the same area as the square-celled surface $M$.
\item Reciprocally, if a cellular walled surface $(M,W)$ fills isometrically its boundary of length $2n$, and we construct a square-celled decomposition that is dual to $W$, then we obtain a square-celled surface $M$ that fills isometrically a $2n$-cycle and has the same area as $(M,W)$.
\end{itemize}
\end{theorem}


A non-cellular walled surface $(M,W)$ that fills isometrically its boundary can be transformed, without increasing its area, into a cellular walled surface that fills isometrically a curve of the same length. However its topology may change.

\subsection{Topology of surfaces that fill a circle, and how to make a wallsystem cellular} In this subsection, the word ``surface'' means a compact surface that has exactly one connected component.

Recall the topological classification of surfaces: every surface is homeomorphic to some surface $M_{b,g,p}=bD\#gT\#pP$, the connected sum of $b$ copies of the disk $D=\DD^2$, $g$ copies of the torus $T=\TT^2$ and $p$ copies of the real projective plane $P=\RR P^2$. The set of possible topologies of a surface is the commutative semigroup generated by the disk $D$, $T$, $P$, modulo the relation $P\#2P\simeq P\# T$. A surface $M'$ is \term{topologically simpler} than $M$ if we can write $M\simeq M_{b,g,p}$ and $M'\simeq M_{b',g',p'}$ with $b'\leq b$, $g'\leq g$ and $p'\leq p$. 

If we consider only surfaces that fill a circle, their possible topologies form the following poset:
\[\xymatrix@C-.7pc@R-.5pc{
D\#P \ar[d] & D\#2P \ar[l] & D\# 3P \ar[l] \ar[d] & D\#4P \ar[l] & D\# 5P\ar[l] \ar[d] & \cdots \ar[l]\\
D           &              & D\# T  \ar[ll]        &              & D\# 2T \ar[ll]& \cdots\ar[l]
}\]
where an arrow $M\to M'$ denotes that $M'$ is simpler than $M$. The vertical arrows are of the form $D\#(2k+1)P\simeq D\#k T\#P \to D\# k T$. Note that the Möbius band $D\# P$ and the genus-1 orientable filling $D\# T$ are incomparable.

\begin{theorem}\label{thm:making_cellular} Let $(M,W)$ be a walled surface that fills isometrically its boundary of length $2n\geq 4$. Then there is a \emph{cellular} walled surface $(M',W')$ of simpler topology than $M$, that fills isometrically its boundary of length $2n$, and has $\Area(M',W')\leq\Area(M,W)$.
\end{theorem}

\begin{proof} The surface $M'$ is constructed from $M$ as follows. Let $U$ be a connected component of $M\setminus(W\cup\partial M)$. Decompose $U$ as connected sum of open disks, tori and real projective planes. Discard the projective planes and tori, and do not reconnect the disks. Perform this decomposition on each component of $M\setminus(W\cup\partial M)$. By this process the wallsystem $W$ and the boundary $\partial M$ remain intact, but the surface $M$ may be divided into several surfaces. We discard all of them except the one that has the boundary. In the end we obtain a surface $M'$ that is topologically simpler than $M$ and has the same boundary $\partial M'=\partial M$. Let $W'=W\cap M'$.

The walled surface $(M',W')$ is an isometric filling because to obtain $(M',W')$ from $(M,W)$ we only made changes away from the walls and the boundary, that consisted of undoing connected sums. These changes can only increase distances between boundary points measured along the walled surface, and do not modify the distances measured along the boundary.

To prove that $(M',W')$ is a cellular wallsystem, we will show that: 
\begin{itemize}
\item Each wall of $W'$ has crossings.
\item For each connected component $U$ of $M'\setminus W'$, the intersection $\partial M'\cap U$ is connected.
\item Each connected component of $M'\setminus(W'\cup\partial M')$ is an open disk.
\end{itemize}
Note that the second and third properties together imply that each connected component of $M'\setminus W'$ is homeomorphic to the plane or a closed half-plane. The third property follows immediately from the construction of $M'$. The second property is already present in $(M,W)$ because it is consequence of being an isometric filling. To finish the proof we must show that each wall of $W'$ has crossings.

Suppose that a wall $w$ of $W'$ has no crossings. If $w$ is a closed, orientation preserving curve, then it is the common boundary of two open disks, therefore it is contained in a surface (homeomorphic to a sphere) that has empty boundary and has been discarded. Similarly, if $w$ is a closed, orientation reversing curve, then it is contained in a surface homeomorphic to the projective plane, that has been discarded. If $w$ is non-closed, then it has its endpoints $a$, $b$ on the boundary $\partial M'$. Consider a closed neighborhood $B$ of $W$ that is a narrow band bounded by two curves $w_0$, $w_1$ that run from boundary to boundary parallel to $w$, one on each side of $w$. The band must contain no piece of $W$ except $w$. Denote $a_i,b_i$ the endpoints of each curve $w_i$, so that $B\cap\partial M'=I_a\cup I_b$, where $I_a=[a_0,a_1]$ is a counterclockwise segment of the boundary $\partial M'$ that contains the point $a$, and $I_b=[b_0,b_1]\subseteq\partial M'$ contains the point $b$. The interval $I_b=[b_0,b_1]$ may be a clockwise or counterclockwise segment. In both cases we have $d_{(M',W')}(a_i,b_i)=0$. If $[b_0,b_1]$ is a counterclockwise segment (or, equivalently, if $a_0,a_1,b_0,b_1$ are in cyclic order along the boundary $\partial M'$), then we obtain a contradiction with the isometric filling condition, namely, we see that $d_{(\partial M',\partial W')}(a_0,b_0)\geq 1$ because along the boundary $\partial M'$ the points $a_0,b_0$ are separated by the endpoints $a,b$ of $w$. If $[b_0,b_1]$ is a clockwise segment (or, equivalently, if $a_0,a_1,b_1,b_0$ are in cyclic order along $\partial M'$), then from the isometric filling condition we have $d_{(\partial M',\partial W')}(a_i,b_i)=0$ and we also have $d_{(\partial M',\partial W')}(a_0,a_1)=1$ and $d_{(\partial M',\partial W')}(b_0,b_1)=1$. We conclude that the length of the boundary is $2$, in contradiction with the hypothesis that $2n\geq 4$. This finishes the proof that all walls have crossings, and therefore the walled surface $(M',W')$ is cellular.
\end{proof}

\subsection{Even wallsystems and even square-celled surfaces}

A wallsystem $W$ on a surface $M$ is called \term{even} if the length $\Len_W(\gamma)$ of every \emph{closed} generic piecewise-differentiable curve $\gamma$ is an even number. To verify that a wallsystem is even, it is enough to test one such curve $\gamma$ of each homotopy class, because if $\gamma,\gamma'$ are two such curves homotopic to each other, then their lengths have the same parity. In particular, on a disk, where all closed curves are contractible, every wallsystem is even.

A square-celled surface $M$ is called \term{even} if its dual wallsystem is even, or, equivalently, if its 1-skeleton graph $M^{\leq 1}$ is bipartite.


\begin{theorem}\label{thm:equiv_even_noneven_FAC} The FAC for square-celled surfaces is equivalent to the FAC for even square-celled surfaces. More precisely, every square-celled surface $M$ that fills isometrically its boundary of length $2n$ and has $\Area(M)=m$ can be transformed into a homeomorphic \emph{even} square-celled surface $N$ that fills isometrically its boundary of length $2n'$ and has $\Area(N)\leq m'$, where $n'=2n$ and $m'=4m+n$.

The same proposition holds for cellular walled surfaces or for walled surfaces instead of square-celled surfaces.
\end{theorem} 

The proof uses a construction called ``eyes'', that will be useful again later, when proving the equivalence between discrete and continuous FACs.

\begin{proof}[Proof of Thm.~\ref{thm:equiv_even_noneven_FAC} for square-celled surfaces] Note that the first claim of the theorem (that is, the equivalence between FAC for even and non-even square-celled surfaces) follows from the second claim, because if the original, not-necessarily-even square-celled surface $M$ is a counterexample to the discrete FAC, that is, if $m<\frac{n(n-1)}2$, then the even square-celled surface $N$ is also a counterexample since \[m'=4m+n<4\frac{n(n-1)}2+n=2n^2-n=n(2n-1)=\frac{n'(n'-1)}2.\] Therefore we just need to prove the claim about the transformation of surface $M$ to surface $N$ with the prescribed properties.

Let $M$ be a square-celled surface that fills isometrically a cycle graph $C_{2n}$ and has $\Area(M)=m$. Color the 1-skeleton graph $M^{\leq 1}$ in blue and draw the dual wallsystem $W$ in red. The wallsystem $W$ breaks each square cell of $M$ into four smaller squares. In this way we obtain a new square-celled surface $N$ that is the same surface $M$ but divided into $4m$ square cells, as shown in Fig.~\ref{fig:hexagon_doubled}. The 1-skeleton graph is made of blue edges and red edges. It is a bipartite graph, because we can assign one color to the vertices where only blue or only red edges meet, and another color to the vertices where blue and red edges meet. Therefore $N$ is an even square-celled surface that fills a cycle graph $C_{4n}$. Does it fill it isometrically? Almost\dots

\begin{figure}
  \centering
  \begin{subfigure}[b]{0.45\textwidth}
    \centering
    \includegraphics[width=\linewidth]{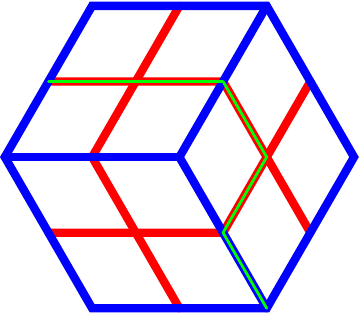}
    \caption{A 3-celled square-celled surface $M$ that fills isometrically $C_6$ is subdivided by its walls, obtaining a 12-celled square-celled surface $N$. A green curve $\gamma$ is drawn in the graph $N^{\leq 1}$.}\label{fig:hexagon_doubled}
  \end{subfigure}
  \hspace{.05\textwidth}
  \begin{subfigure}[b]{0.45\textwidth}
    \centering
    \includegraphics[width=\linewidth]{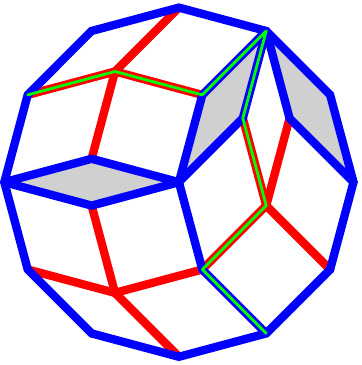}
    \caption{After we insert eyes (drawn in grey) in the square-celled surface $N$, the green curve $\gamma$ must deviate.}\label{fig:hexagon_doubled_eyes}
  \end{subfigure}
  \caption{Transforming a square-celled isometric filling into an even square-celled isometric filling.}
\end{figure}

Let $\gamma$ be a discrete path on $N$ that starts and ends on the boundary $\partial N$. If $\gamma$ is made of blue edges, then it is contained in the original graph $M^{\leq 1}$ and it cannot be a shortcut, since $M$ is an isometric filling of its boundary.

The following proposition is almost true: Every discrete path $\gamma$ in $N$ is not a shortcut, in fact, it can be transformed into a path made of blue edges by a sequence of elementary homotopies that do not increase its length.

\begin{proof}[Attempt of proof] We apply to $\gamma$ repeatedly the following process, until no longer possible:
\begin{enumerate}
\item\label{stp:return} IF it is possible to apply an elementary homotopy of type $2\to 0$ to $\gamma$, do it.
\item\label{stp:turn} ELSE IF $\gamma$ turns at a self-crossing of $W$, let $\delta$ be the piece of $\gamma$ that consists of the two red edges immediately preceeding and succeeding the turning point, let $Q\in M^2$ be the square that contains $\delta$, let $\delta'$ be the shortest path along the boundary of $Q$ that has the same endpoints as $\delta$; it is made of two blue edges. The paths $\delta$ and $\delta'$ form the boundary of a square $Q'\in N^2$. 
Apply to $\gamma$ the elementary homotopy of kind $2\to 2$ that sweeps across the square $Q'$ and replaces $\delta$ by $\delta'$, as shown in Fig.~\ref{fig:push_to_boundary_0}.
\item\label{stp:bluered} ELSE IF $\gamma$ contains a blue edge followed by a red edge, note that $\gamma$ continues straight with another red edge, since the possibility of turning or returning is eliminated by how we handled the previous two cases. These three edges (blue, red, red) form a piece $\delta$ of $\gamma$ that is contained in a square $Q\in M^2$. Let $\delta'$ be the shortest path along the boundary of $Q$ that joins the same endpoints as $\delta$. It consists of three blue edges. The paths $\delta$ and $\delta'$ together enclose two square cells of $N$. By applying to $\gamma$ two consectuive elementary homotopies of type $2\to 2$, we replace in $\gamma$ the piece $\delta$ by the piece $\delta'$, as shown in Fig.~\ref{fig:push_to_boundary_1}.
\item\label{stp:redblue} ELSE IF $\gamma$ contains a blue edge \emph{preceeded} by a red edge, proceed in analogy to the previous case, where the blue edge was \emph{followed} by a red edge.
\end{enumerate}

\begin{figure}
  \centering
  \begin{subfigure}[b]{0.45\textwidth}
    \centering
    \includegraphics[width=.6\linewidth]{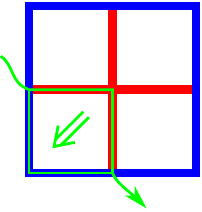}
    \caption{If $\gamma$ turns at a self-crossing of $W$, contained in a square $Q\in M^2$, we push $\gamma$ to the boundary of the square $Q$ in step \ref{stp:turn}.}\label{fig:push_to_boundary_0}
  \end{subfigure}
  \hspace{.05\textwidth}
  \begin{subfigure}[b]{0.45\textwidth}
    \centering
    \includegraphics[width=.6\linewidth]{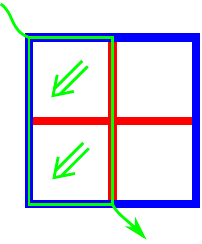}
    \caption{If $\gamma$ has a blue edge followed by two red edges, contained in a square $Q\in M^2$, we push $\gamma$ to the boundary of the square $Q$ in step \ref{stp:bluered}.}\label{fig:push_to_boundary_1}
  \end{subfigure}
  \caption{Pushing a discrete curve $\gamma$ in $N$ to the graph $M^{\leq 1}$.}
\end{figure}

Note that this process reduces the length of $\gamma$ or keeps the length unchanged but reduces the number of red edges. This implies that the process terminates. Once the process finishes, we know that every sequence of consecutive red edges of $\gamma$ is a piece of wall, since the possibilities of returning along a wall or turning at a wall crossing are eliminated by how we handled the first two cases. Additionally, $\gamma$ contains no blue-red or red-blue transitions, since these have been eliminated by the handling of the last two cases. Therefore, in the end $\gamma$ is either all blue or a red wall. This finishes the attempt of proving the proposition. The attempt has failed because we may obtain a red wall rather than a path made of blue edges.
\end{proof}

To obtain a complete proof we need to eliminate the possibility that $\gamma$ is a red wall at the end of the process. Therefore, we will modify the square-celled surface $N$ in order to interrupt the walls of $W$, as follows. Consider a red wall that is not closed. Choose an edge $e\in M^1$ that is intersected by that wall. Note that $e$ consists of two edges of $N$. Slit the surface $N$ along $e$, producing on $N$ a quadrilateral hole. Fill this hole with a new square cell, called an \term{eye}. After this interruption is done on each of the $n$ red walls that are not closed, the modified surface, still denoted $N$, is made of $4m+n$ square cells, as shown in Fig.~\ref{fig:hexagon_doubled_eyes}. On this surface $N$ the proposition stated above is true, in fact, the attempt of proof serves as proof. The possibility that in the end of the process the path $\gamma$ is a red wall is eliminated because every red wall is interrupted by an eye.
\end{proof}


We have now proved that Thm.~\ref{thm:equiv_even_noneven_FAC} holds for square-celled surfaces: the FAC for even square-celled surfaces is equivalent to the FAC for not-necessarily-even square-celled surfaces. The versions of Thm.~\ref{thm:equiv_even_noneven_FAC} for cellular walled surfaces and for walled surfaces follow easily from the version for square-celled surfaces:

\begin{proof}[Proof of Thm.~\ref{thm:equiv_even_noneven_FAC} for cellular walled surfaces] Let $(M,W)$ be a cellular walled surface that fills isometrically its boundary of length $2n$ and has area $m$. The dual square-celled decomposition makes $M$ as square-celled surface that fills isometrically a $2n$-cycle and has area $m$, according to the second part of Thm.~\ref{thm:equiv_wallFAC_squareFAC}. Then, as explained in the last proof (of Thm.~\ref{thm:equiv_even_noneven_FAC} for square-celled surfaces), we can transform this square-celled surface $M$ into a homeomorphic even square-celled surface $M'$ that fills isometrically a $4n$-cycle and has area $4m+n$. The dual wallsystem $W'$ makes $(M',W')$ an even cellular walled surface that fills isometrically its boundary of length $4n$ and has area $4m+n$; according to the first part of Thm.~\ref{thm:equiv_wallFAC_squareFAC}. Note that the topology of the surfaces remains unchanged by each of the transformations employed in this proof.
\end{proof}

\begin{proof}[Proof of Thm.~\ref{thm:equiv_even_noneven_FAC} for walled surfaces] Let $(M,W)$ be a not-necessarily-cellular walled surface that fills isometrically its boundary of length $2n$ and has area $m$. We may transform $(M,W)$ into a cellular walled surface $(M',W')$ of simpler topology that fills isometrically its boundary of length $2n$ and has area $\leq m$, according to theorem~\ref{thm:making_cellular}. Then we apply the Thm.~\ref{thm:equiv_even_noneven_FAC} for cellular walled surfaces to obtain an even cellular walled surface $(M'',W'')$ that fills isometrically its boundary of length $4n$ and has area $\leq 4m+n$. Then we may apply to the surface $M''$ connected sum with copies of a torus or projective plane, as necessary, to obtain a surface homeomorphic to the original surface $M$. If we do not change the surface $M''$ in a neighborhood of the wallsystem $W''$, then $(M'',W'')$ remains an isometric filling after the connected sums. 
\end{proof}

\newpage
\section{Polyhedral approximation of Finsler surfaces}\label{sec:approx_smooth_polyhed}

In this section we prove some technical lemmas that will be used in the proof of the equivalence between discrete and continuous FACs. We consider some variants of the notion of Finsler surface (piecewise-Finsler, polyhedral-Finsler) and show that they are equivalent to Finsler surfaces for the filling area problem, in the sense that the infimum area of an isometric filling of the circle of any of these kinds is the same number. We also give two examples of polyhedral-Finsler hemispheres, that is, polyhedral-Finsler surfaces that isometrically fill the circle and have the same area as the Euclidean hemisphere.


\subsubsection{Simplicial complexes} 

A \term{simplex} of dimension $k$ (or \term{$k$-simplex}) in $\RR^n$ is  the convex hull $T=\Conv(\sigma)\subseteq\RR^n$ of a set $\sigma\subseteq\RR^n$ 
of $k+1$ points, called the vertices of $T$, that are affinely independent (that is, not contained in any $(k-1)$-dimensional plane). A \term{face} of a simplex $T=\Conv(\sigma)$ is a simplex $T'=\Conv(\sigma')$ such that $\sigma'\subseteq\sigma$. We write $T'\leq T$ if $T'$ is a face of $T$.

Low-dimensional simplices have special names. A 2-simplex is called a \term{triangle}. A 1-simplex $[x,y]=\Conv\{x,y\}$ is called an \term{edge} or \term{straight segment}. A 0-simplex is called a \term{vertex} because it may be safely confused with the unique point that it contains, its vertex. The $(-1)$-simplex $\Conv\{\emptyset\}=\emptyset$ is admitted.

A \term{finite simplicial complex} in $\RR^n$ is a finite set $\Sigma$ of simplices $T\subseteq\RR^n$ such that
\begin{itemize}
\item If $T\in\Sigma$ and $T'\leq T$, then $T'\in\Sigma$.
\item If $T,T'\in\Sigma$, then $T\cap T'\leq T.$
\end{itemize}
The set $\Sigma^k_M\subseteq\Sigma_M$ contains the simplices $T\in\Sigma_M$ that have dimension $k$, which are also called the \term{$k$-faces} of the complex $\Sigma$. The \term{dimension} of $\Sigma$ is the maximum dimension of its faces. A \term{simplicial subcomplex} of $\Sigma$ is any simplicial complex $\Sigma'\subseteq\Sigma$. For example, the \term{$d$-skeleton} $\Sigma^{\leq d}\subseteq\Sigma$ is the simplicial subcomplex whose simplices are the simplices of $\Sigma$ that have dimension $k$ or smaller.

The \term{polyhedral space} $|\Sigma|$ of a finite simplicial complex $\Sigma$ is the topological subspace of $\RR^n$ obtained as union $|\Sigma|=\bigcup_{T\in\Sigma}T$ of all the faces of $\Sigma$. The simplicial complex $\Sigma$ is called a \term{simplicial decomposition} of the space $|\Sigma|$.

\subsubsection{Polyhedral manifolds} A compact \term{polyhedral $n$-manifold} is a topological $n$-manifold $M\subseteq\RR^n$ that is the polyhedral space of a finite simplicial complex $\Sigma_M$. Strictly speaking, we should define a polyhedral manifold as a pair $(M,\Sigma_M)$ where $M=|\Sigma_M|$, however, we will refer only to the manifold $M$ and regard the simplicial decomposition $\Sigma_M$ as implicit when we say that $M$ is a polyhedral manifold.

In particular, a compact \term{polyhedral surface} is a topological surface $M\subseteq\RR^n$ that is the union of a finite set $\Sigma_M^2$ of triangles, such that the intersection of any two different triangles $T,T'\in\Sigma_M^2$ is either empty or a common vertex or a common side of the two triangles. 

A compact \term{piecewise-$C^k$ curve} in a polyhedral $n$-manifold $M$ is a curve expressed as a concatenation of finitely many $C^k$ paths, called its \term{pieces}, each contained in an $n$-dimensional face of $M$. Such a curve is called a \term{polygonal curve} if its pieces are straight segments.

\subsubsection{Piecewise-Finsler and polyhedral-Finsler manifolds}

A \term{piecewise-Finsler $n$-manifold} is a pair $(M,F)$, where $M$ is a polyhedral $n$-manifold and $F$ is a \term{piecewise-Finsler metric} on $M$, that is, a family $(F_T)_{T\in\Sigma_M^n}$ of Finsler metrics, one metric $F_T$ on each $n$-dimensional face $T\in\Sigma_M^n$. Each face $T\in\Sigma_M^n$ is considered as a differentiable manifold-with-corners, where the notion of Finsler metric is defined in the same way as on a differentiable manifold. The metrics on two simplices $T,T'\in\Sigma_M^n$ need not agree on vectors that are tangent to the face $T\cap T'$ shared by $T$ and $T'$.

Lengths and $n$-volume on a piecewise-Finsler $n$-manifold are defined as follows. The \term{length} of a compact piecewise-differentiable curve on $M$ is defined as the sum of the lengths of its pieces, and the \term{volume} of the manifold is defined as the sum of the volumes of the $n$-dimensional faces of $M$. This applies to Holmes--Thompson volume or any other Finsler volume function.

A \term{polyhedral-Finsler surface} is a piecewise-Finsler surface $(M,F)$ where the metric $F_T$ is \term{constant} on each triangular face $T$ of $M$, in the sense that the norms $F_{T,x}$ at all points $x\in T$ are equal, therefore the manifold $(T,F_T)$ is a triangular piece of normed plane. In this situation, each of the norms $F_{T,x}$ may be sometimes denoted $F_T$, since it does not depend on the point $x\in T$. 
Saying that a Finsler metric is constant on a manifold makes sense whenever there is a way to identify the tangent spaces at different points of the manifold. This happens, for instance, when the manifold is an $n$-dimensional submanifold of $\RR^n$, or the torus $T^n=\RR^n/\ZZ^n$, or the band $[0,1]\times T^1$.

\subsubsection{Integral norms and metrics}

An \term{integral seminorm} on $\ZZ^d$ is a function $\|-\|:\ZZ^d\to\NN$ that is
\begin{itemize}
\item subadditive: $\|v+w\|\leq\|v\|+\|w\|$ for every $v,w\in\ZZ^d$, and
\item scale covariant: $\|n\,v\|=n\,\|v\|$ for every $v\in\ZZ^d$ and $n\in\NN$.
\end{itemize}

An \term{integral convex polytope} is a set $K\subseteq\RR^n$ obtained as the convex hull of a finite set $S\subseteq\ZZ^n$ of integral points. Note that if $K\subseteq\RR^n$ is an integral convex polytope that contains the origin, then the function \[\|-\|_K:v\in\ZZ^n\mapsto\|v\|_K:=\sup_{p\in K}\langle p,v\rangle\] is an integral seminorm. Reciprocally, Thurston \cite{thurston1986norm} showed\footnote{Among geometers, this fact was discovered by Thurston \cite{thurston1986norm} (see proof in \cite[Thm. 14.5]{fried2009fibrations} or \cite{salle2016norms}), but it is similar to an earlier theorem by Edmonds and Giles (\cite{edmonds1977min}, described at the beginning of \cite[Ch. 22]{schrijver1986theory}), which implies Thurston's theorem under the additional assumption that the dual unit ball $K$ is a rational polytope (so what remains to prove is that its vertices are in fact integral). This is also discussed in~\cite[Thm.~5]{schrijver1993graphs}.} that every integral seminorm is of the form $\|-\|_K$ for some integral convex polytope $K\subseteq\RR^n$.

Note that an integral seminorm, necessarily of the form $\|-\|_K:\ZZ^d\to\NN$, is not technically a seminorm because it is defined only on $\ZZ^d$, not on $\RR^d$. However, we can extend $\|-\|_K$ to a seminorm $\RR^d\to[0,+\infty)$, denoted also $\|-\|_K$, using the same formula: if $v\in\RR^d$ then $\|v\|_K:=\sup_{\varphi\in K}\varphi(v)\in[0,+\infty)$. The extended function will also be called an integral seminorm. Therefore a seminorm on $\RR^d$ is \term{integral} if it has integer values on $\ZZ^d$. 

More generally, let $V$ be a $d$-dimensional real vector space, and choose a lattice $\Lambda\subseteq V$, that is a subgroup of $V$ isomorphic to $\ZZ^d$. Let $\delta>0$. A seminorm $\|-\|$ on $V$ is called $\delta$-integral if $\|v\|$ is an integer multiple of $\delta$ for every $v\in\Lambda$. Note that we can introduce coordinates on $V$ so that $V=\RR^d$ and $\Lambda=\ZZ^d$, and then the seminorm $\|-\|$ is $\delta$-integral if and only if its dual unit ball $B^*\subseteq V^*=(\RR^d)^*=\RR^d$ is a polytope with vertices in $(\delta\ZZ)^d$.

Let $T$ be a $d$-simplex, let $V_T$ be the real vector space spanned by the directed edges of $T$, and let $\Lambda_T\subseteq V_T$ be the lattice of integer vectors, that is, vectors that are integer linear combinations of the directed edges of $T$. A seminorm $F_T$ on $V_T$ is called \term{$\delta$-integral} if $F_T(v)$ is an integer multiple of $\delta$ whenever $v$ is an integral vector. Note that we may introduce linear coordinates so that $T$ becomes the standard simplex \[\Conv\{(0,\dots,0,0),(0,\dots,0,1),\dots,(1,\dots,1,1)\}\subseteq\RR^d,\] therefore $V_T=\RR^d$ and $\Lambda_T=\ZZ^d$. Then the seminorm $F_T$ is $\delta$-integral if and only if its dual unit ball $B_{F_T}^*\subseteq V_T^*=(\RR^d)^*=\RR^d$ is a polyhedron with vertices in $(\delta\ZZ^d)$. 

On a polyhedral-Finsler $d$-manifold $(M,F)$, the polyhedral metric $F$ is called a \term{$\delta$-integral} if for every simplex $T\in\Sigma_M^d$, the norm $F_T$ is $\delta$-integral.

\begin{lemma}[Integral approximation of polyhedral-Finsler metrics]\label{thm:polyhed_integral_approx} If $(M,F)$ is a polyhedral-Finsler $d$-manifold and $\mu>1$, then for any sufficiently small $\delta>0$ there exists $\delta$-integral polyhedral-Finsler metrics $\overline F$ and $\underline F$ on $M$ such that $F\leq\overline F\leq\mu F$ and $F\geq\underline F\geq\mu^{-1}F$.
\end{lemma}

\begin{proof} On each face $T\in\Sigma_M^d$, introduce linear coordinates so that $T$ becomes the standard simplex $\Conv\{(0,\dots,0,0),(0,\dots,0,1),\dots,(1,\dots,1,1)\}\subseteq\RR^d$. Then every polyhedron $B_{F_T}^*\subseteq(\RR^d)^*=\RR^d$ with vertices in $(\delta\ZZ^d)$ is the dual unit ball of a $\delta$-integral norm $\overline F_T$ on $T$. This norm satisfies the inequalities $F_T\leq\overline F_T\leq\mu F_T$ if and only if $B_{F_T}^*\subseteq B_{\overline F_T}^*\subseteq\mu B_{F_T}^*$. A polyhedron $B_{\overline F_T}^*$ with this properties exists if $\delta$ is small enough. The number $\delta$ must be sufficiently small so that for every faces $T\in\Sigma_M^d$, the desired polyhedron $B_{\overline F_T}^*$ exists. This finishes the proof of the existence of $\overline F$. The existence of $\underline F$ is proved in a similar way.
\end{proof}

\subsection{Filling a circle with piecewise-Finsler or polyhedral-Finsler surfaces}

Let $C$ be a closed curve with a distance function $d_C$. A compact Finsler (or piecewise-Finsler, or polyhedral-Finsler) surface $(M,F)$ \term{fills $C$ without shortcuts} (or is a \term{nonshortcutting filling} of $C$) if $\partial M=C$ and \begin{equation}\label{eq:nonshortcutting}d_{(M,F)}(x,y)\geq d_C(x,y)\end{equation} for every two points $x,y\in C$. Note that a nonshortcutting filling is an isometric filling if the inequality \eqref{eq:nonshortcutting} turns out to be an equality.

We give two examples of polyhedral-Finsler hemispheres, that is, polyhedral-Finsler surfaces that, like the Euclidean hemisphere, fill isometrically a circle of length $2L$ and have area $\Area_\uHT=2L^2$.

\begin{example}[Some polyhedral Finsler hemispheres]\label{examp:polyhedral_hemispheres}$\phantom{.}$
\begin{itemize}
\item From the $xy$ plane with $\ell_1$ metric, take the triangle $T$ determined by the inequalities $x,y\geq 0$, $x+y\leq L$, and collapse all the diagonal side $x+y=L$ to a single point. We obtain a disk that fills isometrically its boundary, of length $2L$. Proof: Antipodal boundary points are of the form $p=(x,0)$ and $q=(0,y)$ with $x+y=L$. Going from $p$ to $q$ costs $L$, either if we go straight along the triangle, or if we visit the diagonal side, since $L$ is the sum of the distances of $p$ and $q$ to the diagonal side.

This surface is topologically a disk, but not a polyhedral-Finsler surface because a whole side has been collapsed to a point. To fix this, instead of collapsing the diagonal side, we can just fold it at its midpoint, by gluing each point $(x,y)$ such that $x+y=L$ with the point $(y,x)$. Then we need to subdivide the surface into triangles, for example, using the three lines that connect pairs of midpoints of the sides of $T$. Finally, we should embed the surface in some space $\RR^n$.
\item From the $xy$ plane with $\ell_\infty$ metric, take the strip $B=[-L,L]\times\left[0,\frac L2\right]$, glue each point $(-L,y)$ of the left side to the point $(L,y)$ of the right side to obtain a closed band (equal to $\RR/2L\ZZ\times\left[0,\frac L2\right]$), and then collapse the top side $y=\frac L2$ to a single point. The resulting disk fills isometrically its boundary (the bottom side $y=0$), of length $2L$. Proof: Antipodal boundary points are of the form $p=(x,0)$ and $q=(x+L,0)$. Going from $p$ to $q$ costs at least $L$, either if we go along the band $\RR/2L\ZZ\times\left[0,\frac L2\right]$, or if we visit the top side, which is at distance $\frac L2$ to both points.

Again, the surface is not polyhedral, and to fix this, instead of collapsing the top side to a single point, we can just glue together each pair of points $\left(\pm x,\frac L2\right)$. Then we can subdivide the surface into triangles and embed it in $\RR^n$ to obtain a polyhedral-Finsler surface.
\end{itemize}
\end{example}

The main theorem of this section is the following.

\begin{theorem}\label{thm:equivalence_smooth_polyhedral} Consider surfaces $M$ of a certain fixed topological class whose boundary is a closed curve $C$, and let $C^+$ and $C^-$ be the same curve $C$ with the two possible orientations. For each $a,b,A>0$, the following propositions are equivalent: 
\begin{enumerate}
\item\label{smooth_counterex} There exists a \emph{smoothly convex} Finsler surface $(M,F)$ with $\Area_\uHT(M,F)<A$ that fills without shortcuts a smooth Finsler closed curve $(C,G)$ with $\Len_G(C^+)>a$ and $\Len_G(C^-)>b$.
\item\label{piecewise_counterex} There exists a \emph{piecewise-Finsler} surface $(M,F)$ with $\Area_\uHT(M,F)<A$ that fills without shortcuts a piecewise-Finsler closed curve $(C,G)$ with $\Len_G(C^+)>a$ and $\Len_G(C^-)>b$.
\item\label{polyhedral_counterex} There exists a \emph{polyhedral-Finsler} surface $(M,F)$ with $\Area_\uHT(M,F)<A$ that fills without shortcuts a polyhedral-Finsler closed curve $(C,G)$ with $\Len_G(C^+)>a$ and $\Len_G(C^-)>b$. Additionally, the metrics $F$ and $G$ are $\delta$-integral for some number $\delta>0$.
\end{enumerate}
Furthermore, a surface $(M,F)$ that fills without shortcuts a curve $(C,G)$ as stated in \ref{smooth_counterex}, \ref{piecewise_counterex} or \ref{polyhedral_counterex} can be transformed into a surface $(\widehat M,\widehat F)$ that fills \emph{isometrically} the curve $(C,G)$ and satisfies and satisfies the same conditions \ref{smooth_counterex}, \ref{piecewise_counterex} or \ref{polyhedral_counterex}.
\end{theorem}


\begin{remark} The theorem also holds restricted to self-reverse metrics (with $a=b=2L$), and one could use any Finsler area function instead of the Holmes--Thompson area. The theorem also holds for Riemannian metrics, if we exclude the sentence $\delta$-integrality, that is never satisfied by a Riemannian metric. The reader may check that this variants of the theorem can be proved with the same arguments; no significant changes are needed.
\end{remark}

The proof of Theorem~\ref{thm:equivalence_smooth_polyhedral} ocuppies sections \ref{sec:triangulate}--\ref{sec:smoothapp}.

\subsection{Triangulating Finsler surfaces}\label{sec:triangulate}

The fact that $(\ref{smooth_counterex})$ implies $(\ref{piecewise_counterex})$ is a consequence of the fact that every $C^k$ surface has a $C^k$ triangulation. A proof of this fact may be found in \cite{whitney1957geometric}. A $C^k$ \term{triangulation} of a $C^k$ surface $M$ is a homeomorphism $f:M'\to M$, where $M'$ is a polyhedral surface, such that for each triangular face $T\subseteq M$, the restricted map $f|_T:T\to M$ is $C^k$ and has (if $k\geq 1$) injective differential at every point. When we employ a triangulation $f:M'\to M$ of a $C^k$ surface $M$, we may sometimes identify each point $x\in M'$ with the point $f(x)\in M$, therefore regarding the polyhedral surface $M'$ and the $C^k$ surface $M$ as the same object.

Note that if we triangulate a Finsler surface $(M,F)$, then we obtain a piecewise-Finsler surface, where the metric $F_T$ on each face $T$ of $M$ is the restriction of $F$ to $T$.

\begin{lemma} When we triangulate a Finsler surface $(M,F)$, turning it into a piecewise-Finsler surface, the area of sets and the distances between points and the minlengths of curves remained unchanged.
\end{lemma}

\begin{proof} The fact that the areas do not change is clear. Regarding distances and minlengths, an argument is required because the family of piecewise-differentiable curves, whose length is defined, is reduced when we triangulate the surface. Every piecewise-differentiable curve in the triangulated surface is also piecewise-differentiable (and has the same length) in the original differentiable surface. However, a piecewise-differentiable curve $\gamma$ in the original differentiable surface is not necessarily piecewise-differentiable after the triangulation, because the differentiable pieces of $\gamma$ may not be contained each in a face of the triangulation. In fact, a differentiable piece of $\gamma$ may cross from one face to another infinitely many times. To solve this issue, we may slightly perturb the curve $\gamma$, preserving its homotopy class and modifying its length as little as desired, so that it avoids the vertices and becomes transverse to the edges of the triangulation. Then the curve crosses the edges at finitely many points, and can therefore be broken as concatenation of finitely many differentiable pieces, each contained in a face of the triangulation.
\end{proof}

We are now ready for the next step of the proof of Theorem~\ref{thm:equivalence_smooth_polyhedral}.

\begin{proof}[Proof that $(\ref{smooth_counterex})\Rightarrow(\ref{piecewise_counterex})$ in Thm.~\ref{thm:equivalence_smooth_polyhedral}] Let $(M,F)$ be a smoothly convex Finsler surface with $\Area_\uHT(M,F)<A$ that fills without shortcuts a smooth Finsler curve $(C,G)$ with $\Len_G(C^+)>a$ and $\Len_G(C^-)>b$. Triangulate $M$. Then the metric $F$ on $M$ is broken into Finsler metrics $F_T$ on each of the faces $T$ of $M$. The boundary $\partial M=C$ is also divided into segments, and the metric $G$ is broken into Finsler metrics on each of the segments. Neither the uHT area of $(M,F)$ nor the lengths of $C^+$ and $C^-$ change. The no-shortcuts condition $d_{(M,F)}(x,y)\geq d_{(C,G)}(x,y)$ also remains true, by the last lemma.
\end{proof}


\subsection{Polyhedral approximation of piecewise-Finsler manifolds}\label{sec:polyapp}

The proof that $(\ref{piecewise_counterex})\Rightarrow(\ref{polyhedral_counterex})$ in Theorem~\ref{thm:equivalence_smooth_polyhedral} is based on the two lemmas that follow.

Recall that a Finsler metric $F$ on an $n$-manifold (or manifold-with-corners) $M$ is a continuous function $F:TM\to[0,+\infty)$ that yields a norm $F_x:T_xM\to[0,\infty)$ when restricted to the tangent space $T_xM$ at each point $x\in M$. More genererally, we can consider a \term{Finsler semimetric}, that is, a continuous function $F:TM\to[0,+\infty)$ such that the functions $F_x$ are seminorms, rather than norms. A \term{seminorm} on a real vector space $V$ is a function $V\to[0,+\infty)$ that is subadditive and scale covariant, but not necessarily positive definite. If $M\subseteq\RR^n$, then we can identify all tangent spaces $T_xM$ with $\RR^n$, and the continuity of $F$ can be restated in the following way.

\begin{lemma}[Continuity of Finsler metrics]\label{thm:continuity_Finsler} Let $x$ be a point of a Finsler $n$-manifold $(M,F)$ such that $M\subseteq\RR^n$. Then for every $\mu>1$, there is a neighborhood $U$ of $x$ in $M$ such that every two points $y,y'\in U$ satisfy the inequality $F_{y'}\leq\mu\,F_y$.

More generally, if $F$ is a Finsler semimetric, then for every norm $G$ on $\RR^n$ there exists a neighborhood $U$ of $x$ in $M$ such that every two points $y,y'\in U$ satisfy the inequality $F_{y'}\leq F_y+G$.
\end{lemma}

The proof is a standard argument based on the fact that the unit sphere of any norm is compact.

\begin{proof} We first show that the second version of the lemma, for Finsler semimetrics, implies the first one, for Finsler metrics. Indeed, if $F$ is a Finsler metric, then we can choose a norm $G$ such that the inequality $F_{y'}\leq F_y+G$ implies the inequality $F_{y'}\leq\mu\,F_y$, as follows. Note that it is enough to ensure that $G\leq(1-\mu)\,F_y$ for every $y\in U$, because then we have \[F_{y'}=F_y+G\leq F_y+(1-\mu)\,F_y=\mu\,F_y.\] To define the norm $G$ such that $G\leq(1-\mu)\,F_y$ we do as follows. We may assume that $M$ is compact; if it is not, we first restrict to a compact neighborhood of $x$ in $M$. Then choose any norm $H$ on $\RR^n$ and set $G=\varepsilon H$, where \[\varepsilon=\min_{\substack{y\in M\\v\in\RR^n\\H(v)=1}}(1-\mu)\,F_y(v)>0.\] (The number $\varepsilon$ is strictly positive because it is the minimum of the continuous, strictly positive function $(y,v)\to F_y(v)$ on the compact set $N\times S_H$, where $S_H=\{v\in\RR^n:H(v)=1\}$.) It follows that $G(v)=\varepsilon\,H(v)\leq(1-\mu)\,F_y(v)$ for every $y\in M$ and every $v\in\RR^n$ such that $H(v)=1$, but the inequality is also valid when $H(v)\neq 1$ because both $G(v)$ and $F_y(v)$ are scale-covariant with respect to $v$.

We turn to the proof of the second version of the lemma. Let $F$ be a Finsler semimetric and let $G$ be a norm on $\RR^n$. We have to show that here exists a neighborhood $U$ of $x$ such that $F_{y'}(v)\leq F_y(v)+G(v)$ for every $y,y'\in U$ and every $v\in\RR^n$. To prove this inequality we may assume that $H(v)=1$, since the validity of the inequality is not affected by rescaling $v$, and every vector is a multiple of some vector $v$ such that $H(v)=1$. The inequality can then be rewriten as $F_{y'}(v)-F_y(v)\leq 1$, and we will deduce it from two inequalities $|F_{y'}(v)-F_x(w)|\leq\frac 12$ and $|F_y(v)-F_x(w)|\leq\frac 12$, where $w\in\RR^n$ is an auxiliary vector.

The neighborhood $U$ is defined as follows. For each vector $w\in\RR^n$ such that $H(w)=1$ consider an open neighborhood $U_w\times V_w$ of $(x,w)$ such that $|F_y(v)-F_x(w)|\leq\frac 12$ whenever $(y,v)$ is in the neighborhood. It follows that $|F_{y'}(v)-F_y(v)|\leq 1$ when $y,y'\in U_w$ and $v\in W_w$. The unit sphere $S_G=\{v:G(v)=1\}$ is compact, so it can be covered by finitely many sets $W_{w_i}$. We then define the open neighborhood $U=\bigcap_{i}U_{w_i}$ of $x$, so that the inequality $|F_{y'}(v)-F_y(v)|\leq 1$ holds for $y,y'\in U$ and every $v\in S_H$, as we had to prove.
\end{proof}

A \term{subdivision} of a simplicial complex $\Sigma$ is a simplicial complex $\Sigma'$ such that $|\Sigma'|=|\Sigma|$ and each face of $\Sigma'$ is contained in some face of $\Sigma$. If $\dim(\Sigma)\leq 2$, for any integer $m\geq 1$ we define the \term{lattice subdivision of order $m$} of $\Sigma$ as follows: each edge $E$ of $\Sigma$ is divided into $m$ edges that are translate copies of $\frac 1m E$ and each triangular face $T$ of $\Sigma$ is divided into $m^2$ triangles $T'$, each of which is a translate copy of $\frac 1mT$ or $-\frac 1mT$.\footnote{This subdivision can be generalized to a simplicial complex $\Sigma$ of higher dimension, but it requires that for each simplex $T=\Conv\{\sigma\}\in\Sigma$, the set of vertices $\sigma$ is totally ordered, and that for each face $T'=\Conv\{\sigma'\}\leq T$, the order of $\sigma'$ is the same order of $\sigma$ restricted to $\sigma'\subseteq\sigma$. The subdivision $\Sigma'$ is defined as follows. Consider a $d$-simplex $T\in\Sigma$ with ordered vertices $(a_i)_{0\leq i<d+1}$. Place $T$ it in $\RR^d$ so that $a_i=(0,\dots,0,1,\dots,1)\in\RR^d$ ($i$ zeros followed by $d-i$ ones). Let $\Lambda$ be the set of points in $T$ whose coordinates are integer multiples of $\frac 1m$. The simplices of $\Sigma'$ that are contained in $T$ are the simplices $T'$ with ordered vertices $(p_j)_{0\leq j<k'+1}\subseteq\Lambda$ such that if $j\leq j'$, then each coordinate of the vector $p_{j'}-p_j$ is either 0 or $\frac 1m$. This definition is based on \cite[p. 109, example 9]{spanier1966algebraic}. The subdivision is treated in \cite{edelsbrunner2000edgewise}, and called ``edgewise subdivision''. For subdivisions of simplicial sets see \cite{bokstedt1993cyclotomic}.}

\begin{lemma}[Polyhedral-Finsler approximation of piecewise-Finsler manifolds]\label{thm:polyhed_approx} Let $(M,F)$ be a piecewise-Finsler surface, and let $\mu>1$. Then for any sufficiently fine subdivision $M'$ of $M$ there exists a polyhedral-Finsler metric $\overline F$ on $M'$ that approximates $F$ according to the inequalities $F\leq\overline F\leq\mu\,F$. If $F$ is self-reverse, then $\overline F$ is self-reverse as well, and if $F$ is Riemannian, then $\overline F$ is Riemannian.
\end{lemma}

The expression ``sufficiently fine subdivision $M'$ of $M$'' should be understood in the following sense. During the proof we will construct, for each $n$-simplex $T$ of $M$, a cover $\mathcal U_T$ of $T$ by open sets. A subdivision $M'$ of $M$ is considered sufficiently fine if and only if each $n$-simplex $T'$ of $M'$ that is contained in $T$, is also contained in an open set $U\in\mathcal U_T$. For example, if $\dim M\leq 2$, then the lattice subdivision of order $m$ will be sufficiently fine if the integer $m\geq 1$ is large enough.

\begin{proof} Let $T$ be an $n$-simplex of $M$. According to Lemma~\ref{thm:continuity_Finsler}, each point $x\in T$ has an open neighborhood $U_x$ such that $F_{T,y}\leq\sqrt\mu\,F_{T,y'}$ for every $y,y'\in U_x$. Define the cover $\mathcal U_T$ of $T$ as $\mathcal U_T:=\{U_x:x\in T\}$. Proceed in this way for each face $T$ of $M$.

Let $M'$ be a subdivision of $M$ such that each $n$-simplex $T'$ of $M'$ is contained in the set $U_{x_{T'}}\in\mathcal U_T$ for some point $x_{T'}\in T$, where $T$ is the $n$-simplex of $M$ that contains $T'$. Define on $M'$ the polyhedral-Finsler metric $\overline F$ so that $\overline F_{T',y}=\sqrt\mu\,F_{T,x}$ for every $y\in T'$, where $x=x_{T'}$. From the inequalities $F_{T,y}\leq\sqrt\mu\,F_{T,x}$ and $F_{T,x}\leq\sqrt\mu\,F_{T,y}$ we conclude that $F_{T,y}\leq\overline F_{T',y}\leq\mu\,F_{T,y}$ for every $y\in T'$. Therefore $\overline F$ is a polyhedral-Finsler metric on $M'$ that satisfies $F\leq\overline F\leq\mu\,F$.
\end{proof}


We can now do the second step of the proof of Thm.~\ref{thm:equivalence_smooth_polyhedral}.

\begin{proof}[Proof that $(\ref{piecewise_counterex})\Rightarrow(\ref{polyhedral_counterex})$ in Theorem~\ref{thm:equivalence_smooth_polyhedral}] Let $(M,F)$ be a piecewise-Finsler surface with $\Area_\uHT(M,F)<A$ that fills without shortcuts a piecewise-Finsler curve $(C,G)$ such that $\Len_G(C+)>a$ and $\Len_G(C^-)>b$. From this surface we will obtain a polyhedral-Finsler surface $(\overline M,\overline F)$ that fills without shortcuts a polyhedral-Finsler curve $(\overline C,\underline G)$, satisfying the same bounds for area and length: $\Area_\uHT(\overline M,\overline F)<A$, $\Len_{\underline G}(\overline C^+)>a$ and $\Len_{\underline G}(\overline C^-)>b$.

The construction is as follows. The polyhedral surface $\overline M$ will be a subdivision of $M$, therefore the polyhedral boundary curve $\overline C=\partial\overline M$ will be a subdivision of $C=\partial M$. The polyhedral metric $\overline F$ on $\overline M$ will be \emph{larger} than $F$ and the boundary metric $\underline G$ will be \emph{smaller} than $G$ to ensure that $(\overline M,\overline F)$ fills $(\overline C,\underline G)$ without shortcuts. Additionally, the metrics $\overline F$ and $\underline G$ must be near $F$ and $G$ so that the polyhedral-Finsler surface and boundary curve satisfy the same bounds for area and lengths as the original piecewise-Finsler surface and boundary curve.

Let $\mu>1$ be sufficiently close to 1 so that 
\[\mu^2\Area(M,F)<A,\qquad\mu^{-1}\,\Len_G(C^+)>a,\quad\text{and}\quad
\mu^{-1}\,\Len_G(C^-)>b.\] By Lemma~\ref{thm:polyhed_approx}, if the subdivisions $\overline M$ of $M$ and $\overline C$ of $C$ are fine enough, then there exist polyhedral-Finsler metrics $\overline F$ and $\overline G$ on $\overline M$ and $\overline C$, respectively, such that $F\leq\overline F\leq\mu F$ and $G\leq\overline G\leq\mu G$. We define $\underline G=\mu^{-1}\overline G$ to ensure that $\underline G\leq G$. 
Regarding area and lengths, we have
\[\Area_\uHT(\overline M,\overline F)
\leq\Area_\uHT(\overline M,\mu F)
=\mu^2\Area_\uHT(\overline M,F)<A,\]
\[\Len_{\underline G}(C^+)
=\mu^{-1}\Len_{\overline G}(C^+)
\geq\mu^{-1}\Len_G(C^+)
>a\] and
\[\Len_{\underline G}(C^-)
=\mu^{-1}\Len_{\overline G}(C^-)
\geq\mu^{-1}\Len_G(C^-)
>b\]
as we had to prove.

Finally, we must show how to replace our metrics $\overline F$ and $\underline G$ by metrics that satisfy the same conditions and are $\delta$-integral for some $\delta>0$. For this purpose we do an integral approximation with bounds similar to those employed above. By Lemma~\ref{thm:polyhed_integral_approx}, for any $\nu>1$ there exists a small $\delta>0$ and $\delta$-integral polyhedral-Finsler metrics $\overline{\overline F}$ and $\underline{\underline G}$ on $M$ and $G$ respectively, such that $\overline F\leq\overline{\overline F}\leq\nu\overline F$ and $\underline G\leq\underline{\underline G}\leq\nu\underline G$. The number $\nu>1$ should we chosen sufficiently close to 1 so that 
\[\nu^2\Area(M,\overline F)<A,\qquad\nu^{-1}\,\Len_{\underline G}(C^+)>a,\quad\text{and}\quad
\nu^{-1}\,\Len_{\underline G}(C^-)>b.\]
Then one can repeat the computations to prove that $(M,\overline{\overline F})$ fills $(C,\underline{\underline G})$ without shortcuts, and that $\Area_\uHT(M,\overline{\overline F})<A$, $\Len_{\underline{\underline G}}(C^+)>a$ and $\Len_{\underline{\underline G}}(C^-)>b$.
\end{proof} 

\subsection{Smooth approximation of polyhedral-Finsler surfaces}\label{sec:smoothapp}

To prove that (\ref{polyhedral_counterex}) implies (\ref{smooth_counterex}) in Thm.~\ref{thm:equivalence_smooth_polyhedral}, we need to transform a polyhedral-Finsler surface into a smoothly convex Finsler surface. The first step does not involve the Finsler metric.

A \term{compatible smooth structure} on a polyhedral $n$-manifold $M$ is a smooth structure such that for each face $T\in\Sigma_M$, the inclusion map $\iota_T:T\to M$ is smooth and has injective differential at every point $x\in T$. A polyhedral $n$-manifold is \term{smoothable} if there exists a subdivision $M'$ of $M$ that admits a compatible smooth structure.

\begin{remark}\label{rmk:triangle_extension} The fact that the inclusion map $T\to M$ is smooth and has injective differential implies that this map can be extended to a diffeomorphism $V\to U$, where $U$ is a neighborhood of $T$ in $M$ and $V$ is a neighborhood of $T$ in the plane or a closed half-plane that contains $T$. This follows from Whitney's extension theorem \cite{whitney1934analytic} (any $C^k$ function defined on a closed set $T$ can be extended to a $C^k$ function defined on a neighborhood of $T$) and the fact that a function that is injective in a compact set $T$ and injective in a neighborhood of each point of $T$ is also injective in a neighborhood of $T$.
\end{remark}

\begin{lemma}[Smoothing polyhedral surfaces]\label{thm:smoothing_polyhedral_surfaces} Every compact polyhedral surface is smoothable. In fact, a compact polyhedral surface $M$ has a compatible smooth structure if and only if \begin{equation}\label{eq:bound_vert_cond}\text{each vertex }x\in\partial M\text{ is shared by at least two triangles }T,T'\in\Sigma_M^2,\end{equation} and every polyhedral surface can be subdivided so that this condition is met.
\end{lemma}



\begin{proof} Let $M$ be a polyhedral surface.

Note first that the condition~(\ref{eq:bound_vert_cond}) is necessary for $M$ to admit a differentiable structure. Indeed, suppose that $M$ admits a smooth structure. Then it admits a Riemannian metric which can be restricted to each triangular face of $M$. At each boundary vertex $x\in\Sigma_M^0\cap\partial M$, the sum of the corner angles at $x$ of the triangular faces $T\in\Sigma_M^2$ that share the vertex $x$ is $\pi$, and each corner angle is $<\pi$, therefore the vertex $x$ must be shared by at least two triangles.

Note also that if the condition~\ref{eq:bound_vert_cond} is not met initially by the polyhedral surface $M$, then one can subdivide some triangles of $M$ so that the condition is met, as follows. Let $x$ be a boundary vertex of $M$ that is contained in only one triangle $T=\Conv\{x,y,z\}$. Let $c$ be an interior point of $T$. Subdivide $T$ into three triangles $\{c,x,y\}$, $\{c,y,z\}$, $\{c,z,x\}$. Then the vertex $x$ is shared by the two triangles $\{c,x,y\}$ and $\{c,z,x\}$. This process is repeated for every boundary vertex $x$ where the hypothesis is not met.

To complete the proof it remains to show that $M$ has a compatible smooth structure if it does have the property that each boundary vertex is shared by at least two triangular faces. Note that each interior vertex is shared by at least three triangles.

For simplicity, we consider first the case in which $M$ is a closed surface. We will build a $C^\infty$ atlas for $M$ that has three kinds of charts: face charts, vertex charts and edge charts.
\begin{itemize}
\item For each triangular face $T$, let $U_T=T^\circ$ be the interior of $T$, and let the face chart $\varphi_T:U_T\to\RR^2$ be any linear map defined on the plane that contains the triangle $T$, restricted to the interior of $T$.
\item For each vertex $x$, we let $U_x\subseteq M$ be an open, star-shaped neighborhood of $x$ in $M$. The vertex neighborhoods $U_x$ should be disjoint from each other. For each vertex $x$, the vertex chart is any injective map $\varphi_x:U_x\to\RR^2$ that is linear on each of the triangular faces of $M$ that meet at $x$. Note that each corner of a triangular face that meets at $x$ is mapped to an angle of less than $\pi$ radians, and this is possible because there are at least three faces that share the vertex $x$.
\item For each edge $E=[x,y]$, let $E'=(x',y')\subseteq E$ be an open subinterval of $E$ that has its endpoints $x'\in U_x$ and $y'\in U_y$, but different from $x$ and $y$, respectively. We will define an edge chart $\varphi_E$ by giving its inverse map $\varphi_E^{-1}:E'\times(-\varepsilon,\varepsilon)\to U_E$, whose image $U_E$ is contained in the union of the two triangular faces $T$, $T'$ that share the edge $E$, and such that $U_E\cap E=E'$. We construct the inverse chart $\varphi_E^{-1}$ as follows. Along the segment $E'$, define two smooth vector fields $(v_t)_{t\in E^\circ}$ and $(v'_t)_{t\in E'}$. The vectors $v_t$ should point inside the face $T$, and the vectors $v'_t$ should point inside the other face $T'$. For $t\in U_x$, the vectors $v_t$ and $v'_t$ should be constant and opposite when mapped to the plane by the chart $\varphi_x$ (more precisely, by the linear maps $\varphi_x|_T$ and $\varphi_x|_T'$, respectively), and the same should happen for $t\in U_y$. (Note that one of the two vector fields can be chosen constant on $E'$, but not both, in general.) We can now define the inverse chart $\varphi_E^{-1}$ as follows. For every $(t,s)\in E'\times(-\varepsilon,\varepsilon)$, let $\varphi_E^{-1}(t,s)=t+s\,v_t\in T$ if $s\geq 0$ and let $\varphi_E^{-1}(t,s)=t-s\,v'_t\in T'$ if $s\leq 0$. The number $\varepsilon>0$ should be chosen small enough so that $\varphi_E^{-1}$ maps $E'\times[0,\varepsilon)\to T$ and $E'\times(-\varepsilon,0]\to T'$ injectively and with injective differential. Also, the number $\varepsilon$ should be small enough so that the edge chart domains $U_E$ that correspond to different edges $E$ are disjoint from each other.
\end{itemize}
These charts constitute a smooth atlas for $M$. The chart transition maps are of three kinds: vertex-face, vertex-edge and edge-face. The maps of the two first kinds are linear and the third one is smooth with injective differential. Additionally, for each triangular face $T$, the inclusion $T\to M$ is smooth and has injective differential, in fact, one can extend it explicitely to a smooth diffeomorphism $V\to U$, where $U$ is a neighborhood of $T$ in $M$ and $V$ is a neighborhood of $T$ in the plane that contains $T$. We omit the details of this extension. 
This finishes the proof of the lemma in the case when $M$ is a closed surface.

If the surface $M$ is not closed, we construct a smooth atlas in a similar way but with some modifcations. The charts will be of the form $U\to H$, where $U$ is an open subset of $M$ and $H$ is the upper half-plane $\{x\in\RR^2:x_0\geq 0\}$. There will be face charts, vertex charts an edge charts, constructed as follows.
\begin{itemize}
\item For every triangular face $T$, let $U_T$ be the set of points of $T$ that are not contained in any other triangular face $T'$. The face chart $\varphi_T:U_T\to H$ should be a linear map restricted to $U_T$, such that $\varphi_T(y)\in\partial H$ if and only if $y\in\partial M$.
\item For each vertex $x$ of $M$, let $U_x$ be an open, star-shaped neighborhood of $x$ in $M$ such that all the sets $U_x$ are disjoint. Let $\varphi_x:U_x\to H$ be an injective homeomorphism that is linear on each face that meets at $x$, such that $\varphi_x(y)\in\partial H$ if and only if $y\in\partial M$. If $x$ is a boundary vertex, the fact that $\varphi_x$ maps the boundary segment $\partial M\cap U_x$ to a straight line (a piece of $\partial H$) implies that at least two triangular faces of $M$ should share the vertex $x$.
\item The edge charts are constructed as above, but only for the interior edges of $M$. The boundary edges of $M$ are already covered by the face charts and the vertex charts.
\end{itemize}
As in the case of closed surfaces, these charts constitute a smooth atlas for $M$ such that the inclusion map $T\to M$ of each triangular face $T$ is smooth and has injective differential. One can extend the inclusion map $T\to M$ explicitely to a smooth diffeomorphism $V\to U$, where $U$ is a neighborhood of $T$ in $M$ and $V$ is a neighborhood of $T$ in the plane or a closed half-plane that contains $T$. We omit the details of this extension.
\end{proof}

\begin{lemma}[Smoothing polyhedral-Finsler metrics]\label{thm:smoothing_piecewise-Finsler} Let $(M,F)$ be a polyhedral-Finsler manifold with a compatible smooth structure. Then for any $\varepsilon>0$, there exist smoothly convex Finsler metrics $\overline F\geq F$ and $\underline F\leq F$ such that $|\Vol_\uHT(M,\overline F)-\Vol_\uHT(M,F)|<\varepsilon$ and $|\Vol_\uHT(M,\underline F)-\Vol_\uHT(M,F)|<\varepsilon$.
\end{lemma}

\begin{proof} We explain first how to construct the metric $\overline F\geq F$, and at the end we show how to modify the construction to obtain the metric $\underline F\leq F$.

The first step is to modify the metric $F$ in order to ensure convexity and smoothness of the constant metric $F_T$ on each $n$-face $T\in\Sigma_M^n$. Equivalently, we have to ensure that the boundary of the dual unit ball $B_T^*$ is smooth and curved strictly inwards. If this is not the case initially, then we replace the dual unit ball $B_T^*$ by a slightly larger symmetric convex set that does have these properties. This process increases the norm, therefore the volume of the manifold also increases. The increase must be small enough, so that in the end the volume of the Finsler surface is still $<A$. We do not rename the metric $F$, even though it has increased.

The second step is to smooth out the joins between the different metrics $F_T$, to get a smoothly convex Finsler metric on $M$. For this purpose we extend each metric $F_T$, defined on each $n$-face $T\in\Sigma_M^n$, to a slightly larger open set $U_T\subseteq M$. As mentioned in Remark~\ref{rmk:triangle_extension}, for each face $T$, the inclusion map $T\to M$ can be extended to a diffeomorphism $\iota_T:V_T\to U_T$, where $U_T$ is an open neighborhood of $T$ in $M$ and $V_T$ is an open neighborhood of $T$ in the $n$-plane (or closed half-plane) that contains the $n$-simplex $T$. Therefore, to extend the metric $F_T$ to $U_T$, it is enough to extend it to the plane set $V_T$. Since the metric is constant on $T$, it can be extended to $V_T$ keeping its constant value. This finishes the proof that the metric $F_T$ can be extended to $U_T$. The metric $F_T$ extended to the set $U_T$ will be also called $F_T$.

To transform this collection of metrics $F_T$, defined on the sets $U_T$, into a smooth metric $\overline F$ on $M$, we do as follows. For each $T\in\Sigma_M^n$ and each small number $\delta>0$, choose a smooth function $\eta_T^\delta:U_T\to[0,1]$ that is constantly 1 on $T$ and vanishes outside the $\delta$-neighborhood $T^\delta$ of $T$. To define the neighborhood $T^\delta$ we choose any fixed Riemannian metric $G$ on the smooth surface $M$ and let $T^\delta=\{x\in M:d_G(x,T)<\delta\}$. 

The Finsler metric $\overline F$ will be of the form $F^\delta=\sum_{T\in\Sigma_M^2}\eta_T^\delta F_T$, for some small $\delta>0$. This metric is smooth because it is a sum of the smooth semimetrics $\eta_T^\delta F_T$, and it is also quadratically convex because each seminorm $\eta_T^\delta(x)F_{T,x}$ is a quadratically convex norm unless $\eta_T^\delta(x)=0$. The metric $F^\delta=\sum_T\eta_T^\delta F_T$ is also greater than the original metric $F=\sum_T 1_TF_T$ because $\eta^\delta\geq 1_T$, where $1_T:M\to\{0,1\}$ is the indicator function of the set $T\subseteq M$. Finally, we can ensure that $|\Vol_\uHT(M,F^\delta)-\Vol_\uHT(M,F)|<\varepsilon$ by choosing $\delta$ small enough, because $\Vol_\uHT(M,F^\delta)\xrightarrow{\delta\to 0}\Vol_\uHT(M,F)$. To prove this convergence we decompose \[\Vol_\uHT(M,F^\delta)=\Vol_\uHT(M\setminus N^\delta,F^\delta)+\Vol_\uHT(N^\delta,F^\delta),\] where $N^\delta$ is the $\delta$-neighborhood of the $n-1$-skeleton of $M$ according to the metric $G$. Then we note that:
\begin{itemize}
\item The metric $F^\delta$ coincides with $F$ outside $N^\delta$, therefore \begin{align*}\Vol_\uHT(M\setminus N^\delta,F^\delta)
&=\Vol_\uHT(M\setminus N^\delta,F)\\
&=\Vol_\uHT(M,F)-\underbrace{\Vol_\uHT(N^\delta,F^\delta)}_{\to 0\text{ as }\delta\to 0}\xrightarrow{\delta\to 0}\Vol_\uHT(M,F)\end{align*}
\item The metric $F^\delta$ is bounded by some multiple $\lambda G$ of $G$, therefore \[\Vol_\uHT(N^\delta,F^\delta)\leq\Vol_\uHT(N^\delta,\lambda G)\xrightarrow{\delta\to 0}0.\] 
\end{itemize}
Putting these two facts together we have \[\Vol_\uHT(M,F^\delta)=\underbrace{\Vol_\uHT(M\setminus N^\delta,F^\delta)}_{\to\Vol_\uHT(M,F)}+\underbrace{\Vol_\uHT(N^\delta,F^\delta)}_{\to 0}\xrightarrow{\delta\to 0}\Vol_\uHT(M,F),\] as we had to prove. This finishes the proof of the existence of the metric $\overline F\geq F$ with the desired properties.

The construction of the metric $\underline F\leq F$ is similar, but not entirely analogous. When we replace each norm $F_T$ by a smoothly convex norm, we must reduce it rather than increase it, of course. It is not necessary to extend $F_T$ beyond the face $T$. Then, for small numbers $\delta>0$, define $F^\delta=\sum_T\eta_T^\delta F_T$, where $\eta_T^\delta:M\to[0,1]$ is a smooth function with value 0 on $M\setminus T$ and value 1 outside the $\delta$-neighborhood of $M\setminus T$. Note that $\eta_T^\delta\leq 1_T$, therefore \[F^\delta=\sum_T\eta_T^\delta F_T\leq\sum_T1_TF_T=F.\] The proof that there exists a small number $\delta>0$ such that $|\Vol_\uHT(M,F^\delta)-\Vol_\uHT(M,F)|<\varepsilon$ is exactly as above.

The problem is that $F^\delta$ is a semimetric rather than a metric because it vanishes in a neighborhood of the $n-1$-skeleton of $M$. To transform a semimetric on $M$ into a metric, it suffices to add a small multiple $\alpha\,G$ of the Euclidean metric $G$. However, before summing $\alpha\,G$ to the semimetric $F^\delta$, we must first reduce the semimetric so that the final metric does not exceed $F$. Therefore we define $\underline F=\lambda\,F^\delta+\alpha\,G$, where $\lambda<1$ so that $\lambda F^\delta$ is strictly smaller than $F$, but $\lambda$ is very close to 1 so that $\Vol_\uHT(M,\lambda\,F^\delta)>\Vol_\uHT(M,F)-\varepsilon$, and $\alpha>0$ is small enough so that $\alpha\,G$ can be added to $\lambda\,F^\delta$ without exceeding $F$. We conclude that $\underline F$ is a strictly convex Finsler metric such that $\underline F\leq F$ and $|\Vol_\uHT(M,\underline F)>\Vol_\uHT(M,F)-\varepsilon$, as we had to prove.
\end{proof}

We can now finish the proof of Thm.~\ref{thm:equivalence_smooth_polyhedral}

\begin{proof}[Proof that $(\ref{polyhedral_counterex})\Rightarrow(\ref{smooth_counterex})$ in Thm.~\ref{thm:equivalence_smooth_polyhedral}] Let $(M,F)$ be a polyhedral-Finsler surface that has $\Area_\uHT(M,F)<A$ and fills without shortcuts a Finsler curve $(C,G)$ with $\Len_G(C^+)>a$ and $\Len_G(C^-)>b$. We will produce a smoothly convex Finsler surface $(M,\overline F)$ that fills without shortcuts a smooth Finsler curve $(C,\underline G)$ satisfying the same inequalities for area and length: $\Area_\uHT(M,\overline F)<A$, $\Len_{\underline G}(C^+)>a$ and $\Len_{\underline G}(C^-)>b$.

The first step is to subdivide some triangular faces of $M$ so that each boundary vertex is shared by at least two triangular faces of $M$. After this we can introduce on $M$ a compatible smooth structure, as explained in Lemma~\ref{thm:smoothing_polyhedral_surfaces}. Then, according to Lemma~\ref{thm:smoothing_piecewise-Finsler}, there is a smoothly convex Finsler metric $\widetilde F\geq F$ on $M$ such that $\Area(M,\widetilde F)<A$. Also according to Lemma~\ref{thm:smoothing_piecewise-Finsler}, on the boundary curve $C=\partial M$ we may replace the polyhedral-Finsler metric $G$ by a smoothly convex Finsler metric $\overline G\leq G$ such that the reduction in volume $\Vol_\uHT(C,G)-\Vol_\uHT(C,\overline G)$ is less than an arbitrarily small number $\varepsilon>0$. The uHT volume of a closed curve $C$ is the sum of the lengths of $C^+$ and $C^-$, therefore we conclude that the reductions of length $\Len_G(C^+)-\Len_{\underline G}(C^+)$ and $\Len_G(C^-)-\Len_{\underline G}(C^-)$ are also less than $\varepsilon$, therefore we have $\Len_G(C^+)>a$ and $\Len_G(C^-)>b$ if the number $\varepsilon$ is small enough.

The inequalities $\overline F\geq F$ and $\underline G\leq G$ ensure that $(M,\overline F)$ fills $(C,\underline G)$ without shortcuts, as we had to prove.
\end{proof}



\subsection{Polygonal approximation of curves}

In this subsection we prove the following lemma, which will be used in Section~\ref{sec:equiv_discrete_continuous_FAC}.

\begin{lemma}[Polygonal approximation of curves]\label{thm:polyg_approx_curves} Let $(M,F)$ be a compact polyhedral-Finsler manifold. Then for every $\mu>1$ and $\varepsilon>0$, there exist $\kappa,\alpha\geq 0$ such that every piecewise-differentiable path or closed curve $\gamma$ is homotopic to a polygonal curve $\overline\gamma$ made of at most $\kappa\Len_F(\gamma)+\alpha$ segments and of length $\Len_F(\overline\gamma)\leq\mu\Len_F(\gamma)+\varepsilon$.
\end{lemma}

\begin{remark} If $M$ is a simplicial complex in Euclidean space with the Euclidean metric, then for every curve $\gamma$ in $M$ we can obtain a polygonal curve $\overline\gamma$ with length $\Len(\overline\gamma)\leq\Len_F(\gamma)$ and made of at most $k(\frac{\Len_F(\gamma)}\lambda+1)$ segments, where $\lambda$ is half the systole of the complex and $k$ is the number of maximal faces of $M$. This fact will not be used later. The proof is left as exercise for the reader.
\end{remark}

The proof of Lemma~\ref{thm:polyg_approx_curves} is based on the following lemma.

\begin{lemma}\label{thm:polyhed_approx_distance} Let $(M,F)$ be a compact polyhedral-Finsler manifold. Then for every $\varepsilon>0$ there exists a number $k\in\NN$ such that for every two points $x,y\in M$ there exists a polygonal path $\gamma$ from $x$ to $y$ of length $\Len_F(\gamma)\leq d_F(x,y)+\varepsilon$ and made of at most $k$ straight segments.
\end{lemma}

\begin{proof} For every $x,y\in M$, let $\gamma_{x,y}$ be a polygonal curve from $x$ to $y$ of length $\Len_F(\gamma_{x,y})\leq d_F(x,y)+\frac\varepsilon 5$. Let $k_{x,y}$ be the number of straight segments of the curve $\gamma_{x,y}$.

For every point $x\in M$, let $U_{x}$ be an open neighborhood of $x$ such that for every point $y\in U_x$, the straight segment $[x,y]$ is contained in a face $T$ of $M$ such that $\Len_{F_T}[x,y]\leq\frac\varepsilon 5$ and $\Len_{F_T}[y,x]\leq\frac\varepsilon 5$. Let $S\subseteq M$ be a finite set of points such that $M=\bigcup_{x\in S}U_x$. Let $k=\max_{x,y\in S}k_{x,y}+2$.

Let $x,y\in M$. We will construct a polygonal curve $\gamma$ from $x$ to $y$ of length $\Len_F(\gamma)\leq d_F(x,y)+\varepsilon$ and made of at most $k$ straight segments. Let $x',y'\in S$ be such that $x\in U_{x'}$ and $y\in U_{y'}$. We define the curve $\gamma$ as the concatenation $\gamma=[x,x']*\gamma_{x',y'}*[y',y]$. It is a polygonal curve of $k_{x',y'}+2\leq k$ segments and its length is \[\Len_F(\gamma)
\leq\frac\varepsilon 5+\underbrace{\Len_F(\gamma_{x',y'})}_{\leq\underbrace{d_F(x',y')}_{\leq \frac\varepsilon 5+d_F(x,y)+\frac\varepsilon 5}+\frac\varepsilon 5}+\frac\varepsilon 5\leq d_F(x,y)+\varepsilon.\]
\end{proof}

\begin{proof}[Proof of Lemma~\ref{thm:polyg_approx_curves}] For every vertex $x\in M^0$, define its \term{closed star} $\overline\Star(x)$ as the union of the faces $T\in\Sigma_M$ that contain the vertex $x$. Note that the interiors of these sets $\overline\Star(x)$ cover the compact space $M$, therefore there exists a number $\lambda>0$ such that every piecewise-differentiable path of length $\leq\lambda$ is contained in one of the sets $\overline\Star(x)$. 

Note that each closed star $(\overline\Star(x),F)$ is a polyhedral-Finsler manifold, therefore by Lemma~\ref{thm:polyhed_approx_distance} there exists a number $k_x\in\NN$ such that every two points $y,z\in\overline\Star(x)$ are joined by a polygonal curve $\gamma\subseteq\overline\Star(x)$ made of at most $k_x$ straight segments and of length $\Len_F(\gamma)\leq d_{(\overline\Star(x),F)}(y,z)+\varepsilon'$, where $\varepsilon'$ is a fixed small number $\varepsilon'>0$ such that \begin{equation}\label{eq:ksc34g}\varepsilon'\leq\varepsilon\quad\text{ and }\quad 1+\frac{\varepsilon'}\lambda\leq\mu.\end{equation} Let $k=\max_{x\in M^0}k_x$.

Let $\gamma$ be a compact piecewise-differentiable curve in $M$, and let $L:=\Len_F(\gamma)$. To finish the proof, we will show that there exists a polyhedral curve $\overline\gamma$ of length $\Len_F(\overline\gamma)\leq\mu\Len_F(\gamma)+\varepsilon$, homotopic to $\gamma$, and made of at most $\kappa L+\alpha$ straight segments, where $\kappa=\frac k\lambda$ and $\alpha=k$.

To construct the polygonal $\overline\gamma$, decompose the curve $\gamma$ as a concatenation of paths $\gamma_i$, with $0\leq i<m$, each of them of length $\Len_F(\gamma_i)\leq\lambda$, and with $m\leq\frac L\lambda+1$. Each curve $\gamma_i$ is contained in some closed star $\overline\Star(x)$, with $x\in M^0$. Note that $\Len_F(\gamma_i)\geq d_{(\overline\Star(x),F)}(x_i,x_{i+1})$, where $x_i$ and $x_{i+1}$ are the startpoint and endpoint of $\gamma_i$, respectively. Let $\overline\gamma_i$ be a polygonal curve in $\overline\Star(x)$ from $x_i$ to $x_{i+1}$, of length \[\Len_F(\overline\gamma_i)\leq d_{(\overline\Star(x),F)}(x_i,x_{i+1})+\varepsilon'\leq\Len_F(\gamma_i)+\varepsilon',\] and made of at most $k$ straight segments; this curve exists by how the number $k$ was chosen. Note that $\overline\gamma_i$ is homotopic to $\gamma_i$ because both curves are contained in the closed star $\overline\Star(x)$, which is contractible. Let $\overline\gamma$ be the concatenation of the curves $\overline\gamma_i$. It is a polygonal curve homotopic to $\gamma$, of length \begin{align*}\Len_F(\overline\gamma)
&=\sum_{0\leq i<m}\underbrace{\Len_F(\overline\gamma_i)}_{\leq\Len_F(\gamma_i)+\varepsilon'}\leq\sum_{0\leq i<m}\Len_F(\gamma_i)+m\,\varepsilon'\\
&=L+m\,\varepsilon'
\leq L+\left(\frac L\lambda+1\right)\varepsilon'
=\left(1+\frac{\varepsilon'}\lambda\right)\,L+\varepsilon'
\stackrel{\text{by \eqref{eq:ksc34g}}}{\leq}\mu\,L+\varepsilon,
\end{align*}
and its number of straight segments is $\leq k\,m\leq k\left(\frac L\lambda+1\right)=\kappa\,L+\alpha$, as we had to show.
\end{proof}

\subsection{Proof that the discrete FAC implies the continuous FAC}
We can now prove that the discrete FAC, for walled surfaces, implies the continuous FAC, for Finsler surfaces with self-reverse metric.

\begin{theorem}[Walled FAC implies Finsler self-reverse FAC]\label{thm:walledFAC_implies_FinslerFAC} Let $(M,F)$ be a self-reverse Finsler surface that has $\Area_\uHT(M,F)<2L^2$ and fills without shortcuts a self-reverse Finsler closed curve $(C,G)$ of length $2L$. Then there is a walled surface $(\overline M, W)$, homeomorphic to $M$, that fills isometrically its boundary of length $2n$ and has $\Area(\overline M,W)<\frac{n(n-1)}2$.
\end{theorem}


\begin{proof} The discretization will be done in four steps, described roughly as follows. First, we will approximate our Finsler surface by a polyhedral surface, made of triangles cut from different normed planes. Second, we will modify the norms so that they become integral. Third, we will declare that some (finitely many) points of the polyhedral surface are integral, and we will approximate each curve by a polygonal whose breakpoints are integral. Fourth, on each triangular face of the surface we will replace the norm by a wallsystem, without modifying the minlength of any curve that has integral endpoints. We proceed to the details.

\begin{figure}
  \centering
  \begin{subfigure}[t]{.3\textwidth}
    \includegraphics[width=\hsize]{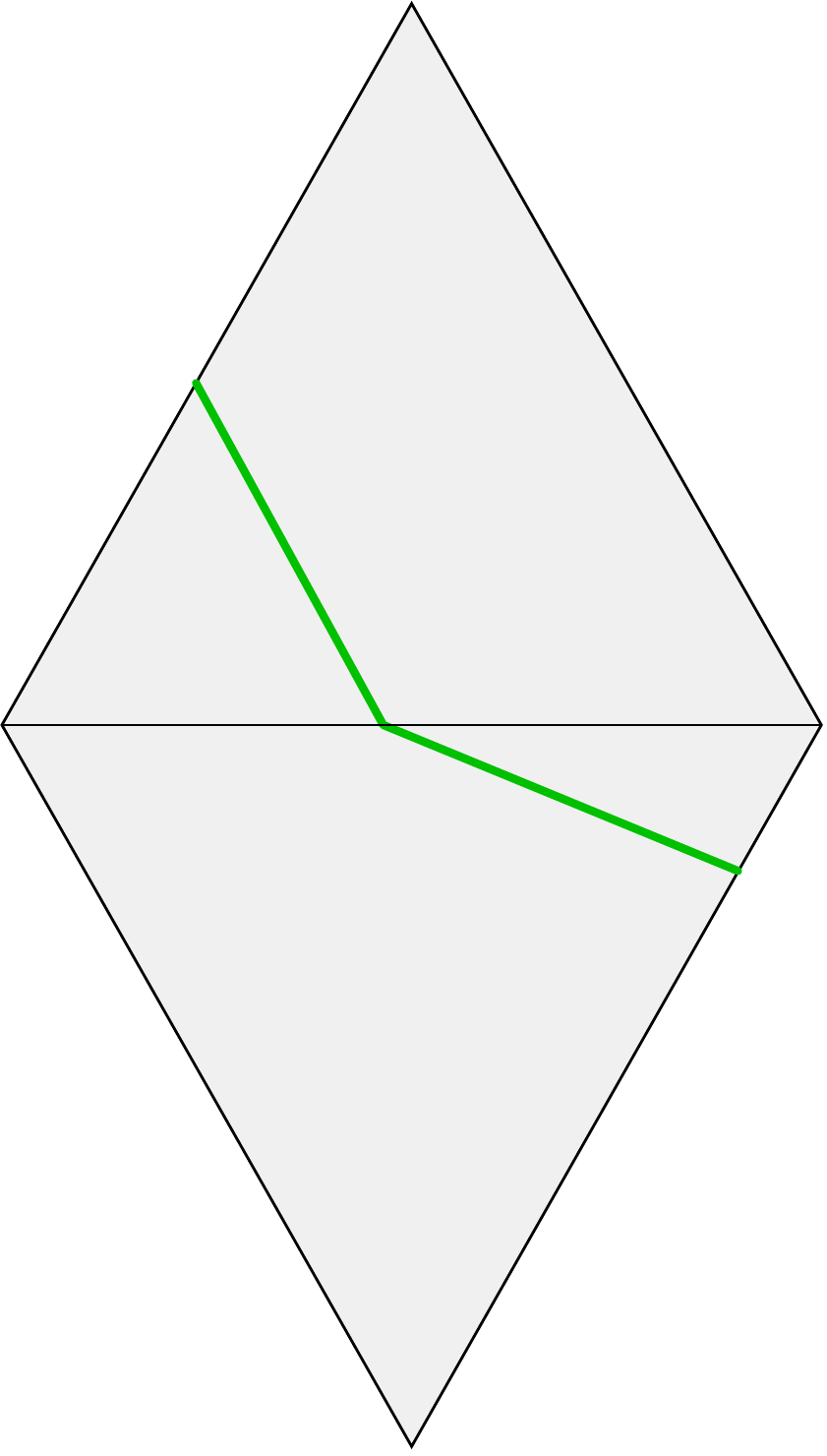}
    \caption{A polyhedral-Finsler surface with a shortest path from one boundary point to another.\label{fig:polyhed_Finsler_surface}}
  \end{subfigure}
  \hspace{.03\textwidth}
  \begin{subfigure}[t]{.3\textwidth}
    \includegraphics[width=\hsize]{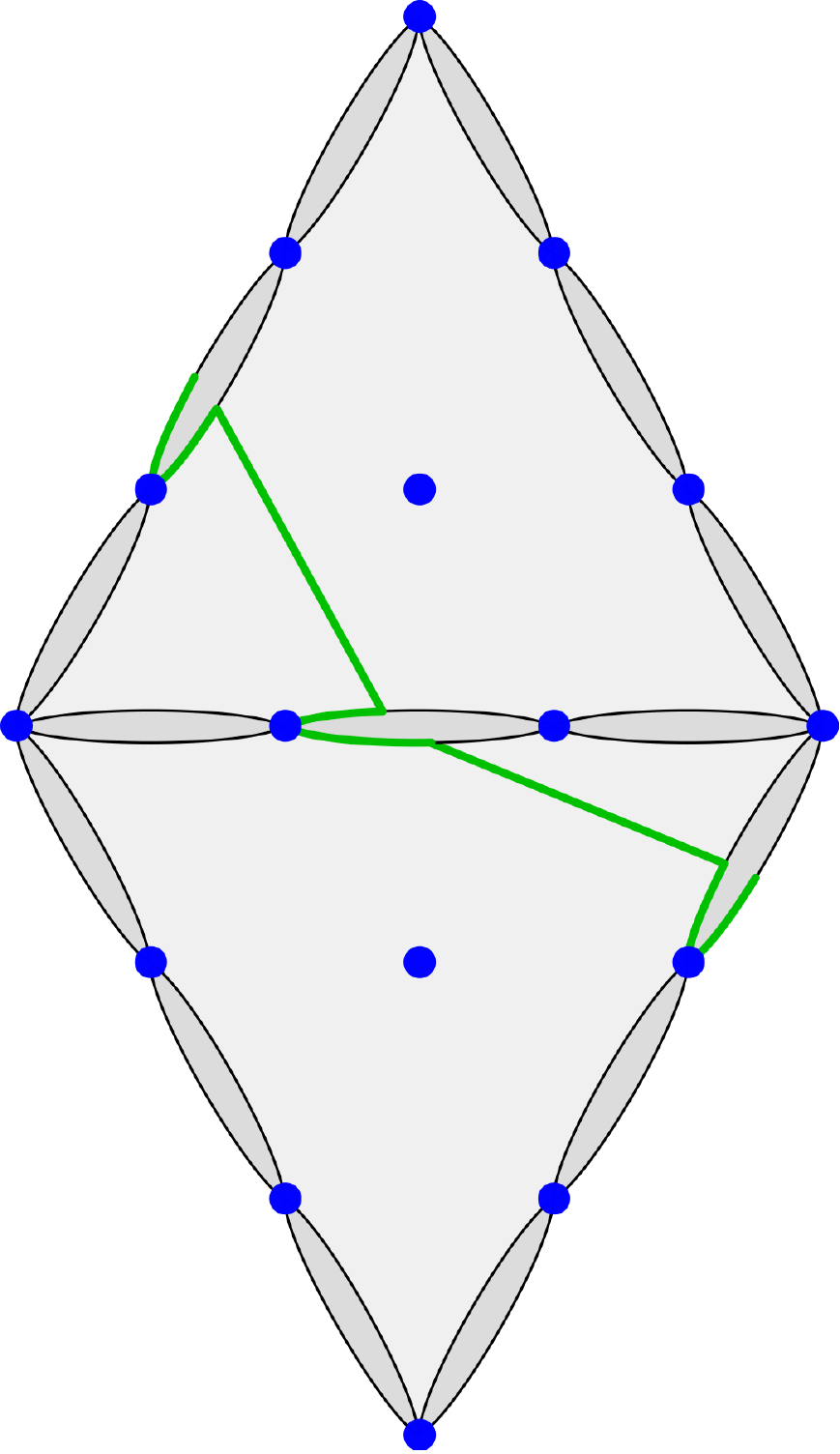}
    \caption{After inserting portals and eyes, the shortest path joining the same two points changes.\label{fig:surface_with_eyes}}
  \end{subfigure}
  \hspace{.03\textwidth}
  \begin{subfigure}[t]{.3\textwidth}
    \includegraphics[width=\hsize]{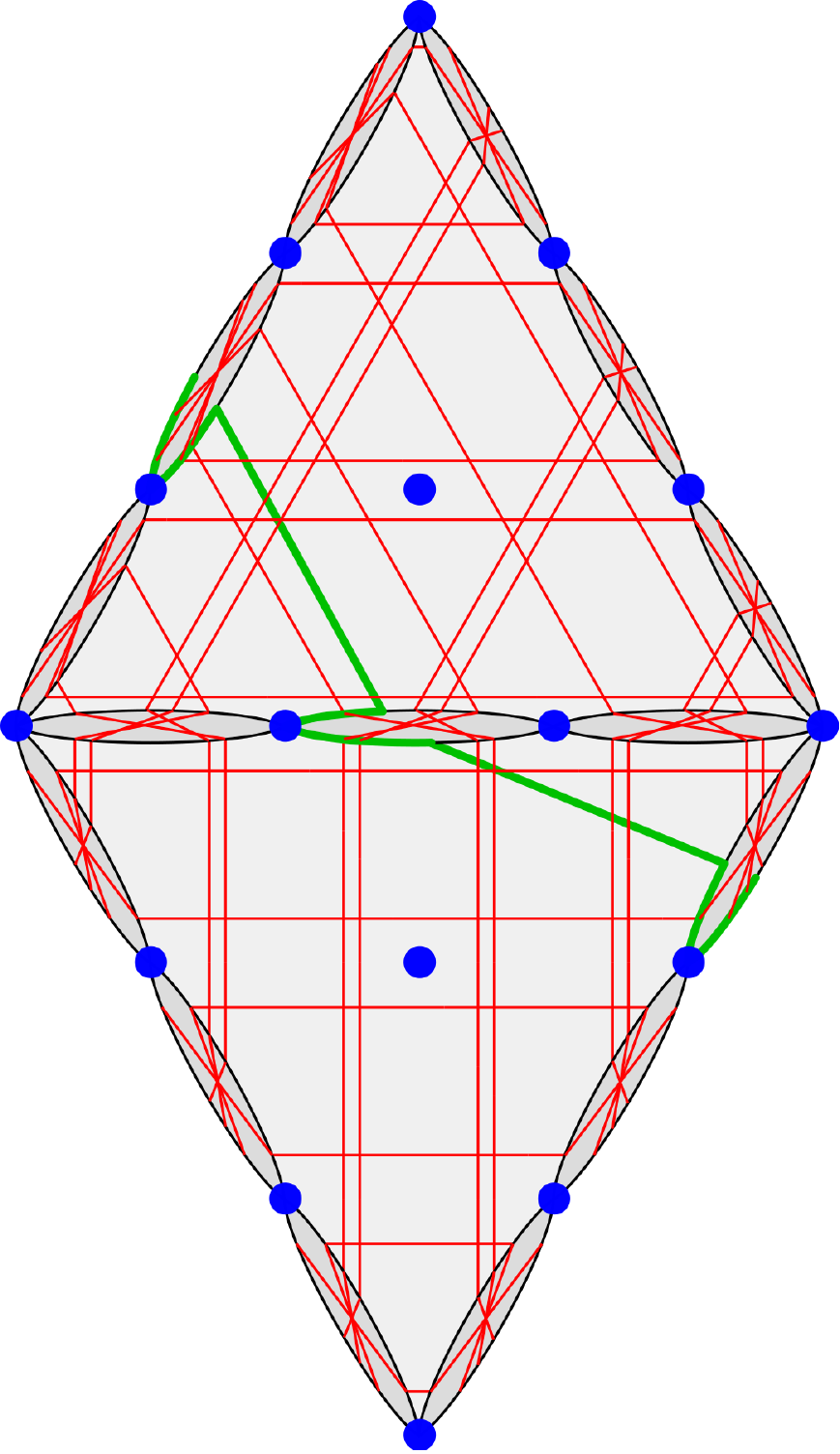}
    \caption{The surface with eyes after the Finsler metrics are replaced by wallsystems.\label{fig:surface_with_eyes_and_walls}}
  \end{subfigure}
  \caption{Discretization of a self-reverse Finsler surface.}
\end{figure}


In the first two steps we will convert the Finsler surface $(M,F)$ in a polyhedral-Finsler surface $(M,\overline F)$ that has $\Area_\uHT(M,\overline F)<2\underline L^2$ and fills without shortcuts a polyhedral-Finsler curve $(C,\underline G)$ of length $2\underline L$. Additionally we will ensure that the metrics $\overline F$ and $\underline G$ are $\delta$-integral for some small number $\delta>0$. The approximation is possible by Thm.~\ref{thm:equivalence_smooth_polyhedral}, but we will summarize the process here.

\proofstep{Polyhedral approximation} In this step we will obtain the metrics $\overline F$ and $\underline G$ with the properties promised above, except for the $\delta$-integrality. We do it as follows. Subdivide the Finsler surface $(M,F)$ into triangles, so that it becomes a piecewise-Finsler surface, that is a surface made of flat triangles with a Finsler metric $F_T$ on each triangular face $T$. Then we subdivide each triangle into smaller triangles, so that in each triangle $T$, the norms $F_x$ and $F_{x'}$ at different points $x,x'\in T$ are very similar, more precisely, we have $F_{x'}\leq \lambda F_x$, where $\lambda>1$ is a number that we may choose, arbitrarily close to 1. This is possible because the Finsler metric $F$ is continuous
. Then, in each triangle $T$, we choose any point $y\in T$ and replace all the slightly different norms $F_x$ by a single norm $\overline F_T:=\lambda\,F_y$, that is slightly larger than all the norms $F_x$. In this way we obtain a polyhedral-Finsler metric $\overline F$ that is slightly larger than the original Finsler metric $F$, more precisely, we have $F\leq\overline F\leq\lambda F$. Therefore the polyhedral-Finsler surface $(M,\overline F)$ still fills without shortcuts the curve $(C,G)$. Also, we can ensure that $\Area_\uHT\left(M,\overline F\right)<2L^2$ by choosing a small $\lambda$, because \[\Area_\uHT\left(M,\overline F\right)<\Area_\uHT(M,\lambda F)=\lambda^2\Area_\uHT(M,F).\] On the boundary curve $C=\partial M$, we approximate the Finsler metric $G$ by a polyhedral-Finsler metric $\underline G$ that is smaller than $G$, not larger, so that the polyhedral-Finsler curve $(C,\underline G)$ is still filled without shortcuts by the polyhedral-Finsler surface $(M,\overline F)$. The length of the curve $(C,\underline G)$ is a number $\Len_{\underline G}(C)=2\underline L$ that is smaller than the $\Len_G(C)=2L$, but only slightly, so that $\Area_\uHT\left(M,\overline F\right)<2\underline L^2$, therefore we still have a counterexample to the FAC, but of polyhedral-Finsler kind.

\proofstep{Integral approximation} In this step we modify the polyhedral-Finsler metrics $\overline F$ and $\underline G$ so that they become $\delta$-integral. The modified metrics will keep the same names to avoid complicating the notation.

The modifications are done as follows. Choose a small number $\delta>0$. For each triangular face $T$ of $M$, to ensure that the norm $F_T$ is $\delta$-integral, we replace the dual unit ball $B_{\overline F_T}^*$ by a slightly larger polygon whose vertices coordinates are integral multiples of $\delta$. Similarly, for each edge $E$ of the boundary curve $\partial M=C$, we replace the dual unit ball $B_{\underline G_E}^*$ (which is an interval) by a sligthly smaller interval whose endpoints are integer multiples of $\delta$. If $\delta$ is small enough, these modifications can be done without breaking any of the conditions, therefore we still have a counterexample to the FAC.


At this point we have a $\delta$-integral polyhedral-Finsler surface $(M,\overline F)$ with $\Area_\uHT(M,\overline F)<2\underline L^2$ that fills without shortcuts a $\delta$-integral polyhedral-Finsler curve $(C,\underline G)$ of length $\Len_{\underline G}(C)=2\underline L$.

\proofstep{Portals and eyes}
Choose a large integer $m>0$. In each triangular face $T$ of $M$, introduce linear coordinates so that $T$ becomes the standard triangle $T=\Conv\{(0,0),(0,m),(m,m)\}\subseteq\RR^2$. Note that each edge of $M$ contains $m+1$ integral points, called \term{portals}, that divide the edge into $m$ \term{fences} that are open segments, not including endpoints. The distance between boundary points is an infimum of lengths of paths along $M$, and our intention now is to prohibit those paths that cross the fences, so that the shortest path between any pair of points is a polygonal curves with breakpoints at the portals. To achieve this we slit the surface $M$ along each fence. More precisely, we cut the triangular faces apart from each other, and also cut the cycle $C$ away from the boundary $\partial M$, except at the portals. In this way we produce a lot of small holes in the surface. The boundary of each hole is a bigon; it has two sides, whose lengths may differ. (In the case of boundary eyes, the side along $C$ is never shorter than the side along $\partial M$.) We fill each hole isometrically using a Euclidean hemisphere, called an \term{eye}.

Let $(\overline M,\overline F)$ be the surface with the hemispherical eyes inserted. It is completely determined by the original polyhedral-Finsler surface $(M,\overline F)$ and by $m$. The topologies of both surfaces are the same. The surface $(\overline M,\overline F)$ is not exactly a piecewise-Finsler surface because it is not made of flat triangles in $\RR^k$, but it can be subdivided into triangles and embedded in $\RR^k$, therfore we will treat it as a piecewise-Finsler surface. Note that all the pieces that form $M$ are glued isometrically along the shared edges, and the attachment of $\partial\overline M$ to $C$ is isometric: distances between boundary points of $M$ coincide with the corresponding distances along $C$. 
 The area $\Area_\uHT(\overline M,\overline F)$ will be $<2L^2$ if $m$ is large enough, since the gain in area by introducing the eyes is proportional to $\frac 1m$ because:
\begin{itemize}
\item along each edge the number of eyes is $m$ and
\item the uHT area of each hemispheric eye is proportional to $\frac 1{m^2}$, considering that the area of a hemisphere is twice the squared perimeter, and the perimeter of each eye is in turn proportional to $\frac 1m$.
\end{itemize} Finally, note that since each hemispheric eye fills its boundary isometrically, it follows that every two points $x,y\in\partial M$ are joined by a shortest path that avoids the interiors of the eyes, and is a concatenation, with breakpoints at portals, of boundary segments and straight segments along triangular faces. The length of each straight segment that goes from portal to portal along a triangular face of $\overline M$ is an integer multiple of $\frac \delta m$. To transform this lengths into integer multiples of $4$, we multiply the metric $\overline F$ (without changing its name) by the factor $\frac{4m}\delta$. Then we have $\Len_{\overline F}(\partial\overline M)=2n$ for some $n\in\NN$ and $\Area_\uHT(\overline M,\overline F)<2n^2$.

\proofstep{Replacement of Finsler metric by wallsystems} Each triangular face $T$ of the surface $(\overline M,\overline F)$ is a piece $\Conv\{(0,0),(0,m),(m,m)\}$ of normed plane, where the distances between integral points are multiples of $4$. We may now replace the continuous norm by a straight wallsystem $W_T$, without modifying the lengths of straight segments between portals nor the total area of the face. This follows from Lemma~\ref{thm:discretize_plane_selfrev} and Remark~\ref{rmk:square_cut_along_diag}, since the face $T$ is decomposed into pieces that are integer translates of the triangles $\Conv\{(0,0),(0,1),(1,1)\}$ and $\Conv\{(0,0),(1,0),(1,1)\}$, whose uHT area according to $W_T$ equals their  Holmes--Thompson area according to the metric $\overline F$.

Also, we break the boundary $\partial\overline M$ into segments at the portal points, and on each segment $I$ we replace the metric of $\partial\overline M$ by a wallsystem $W_I$ (a finite set of points) that has the same length. Finally, on each hemispheric eye $E$ between two faces $T,T'$ (or between a face $T$ and a boundary edge $I$), we replace the Euclidean metric by a wallsystem $W_E$ that matches the existing wallsystems $W_T$ and $W_T'$ (or $W_T$ and $W_I$) along the boundary $\partial E$, so that $(E,W_E)$ is a walled surface that fills isometrically its boundary of length $2k$ and has $\Area(E,W_E)=\frac{k(k-1)}2$. Wallsystems of this kind have been constructed in Example~\ref{ex:discretized_euclidean_hemisphere}.

In this way we get a piecewise-differentiable wallsystem $W$ on the whole surface $M$. Note that when we replace the metric $\overline F$ by $W$, the area of each eye $E$ is reduced, according to the formula \[\Area_\uHT(E,W_E)=2k^2-2k=\Area_\uHT(E,\overline F)-\Len_W(\partial E),\] where $2k:=\Len_W(\partial E)=\Len_{\overline F}(\partial E)$.

The surface $(\overline M,W)$ is an isometric filling of its boundary. Indeed, for each curve $\gamma$ along $\overline M$ that joins two points $x,y\in\partial\overline M\setminus W$, we will show that \begin{equation}\label{eq:no_shortcut}\Len_W(\gamma)\geq d_{(\partial\overline M,W)}(x,y).\end{equation} We may assume that $\gamma$ avoids the interiors of the eyes and is a concatenation, with breakpoints at portals, of boundary pieces and straight segments along triangular faces. We may further assume that $\gamma$ visits no boundary portals except at its endpoints, because otherwise we can break $\gamma$ at any such point and prove inequality~\eqref{eq:no_shortcut} separately for each of the pieces. If $\gamma$ goes along the boundary, then it clearly satisfies the inequality. Otherwise, $\gamma$ is a concatenation of straight segments along triangular faces, and its endpoints $x,y$ are portals on $\partial\overline M$. In this case, inequality~\eqref{eq:no_shortcut} follows from the analogous inequality \[\Len_{\overline F}(\gamma)\geq d_{(\partial\overline M,\overline F)}(x,y),\] because the lengths $\Len_W$ and $\Len_{\overline F}$ coincide on $\gamma$ (since they coincide on each straight segment of $\gamma$), and $d_{(\partial\overline M,W)}(x,y)=d_{(\partial\overline M,\overline F)}(x,y)$ since $x,y$ are portals. 

To show that the isometric filling $(\overline M,W)$ is in fact a counterexample to the discrete filling area conjecture, we bound its area
\begin{align*}
\Area_\uHT(\overline M,W)
&=\underbrace{\Area_\uHT(\overline M,\overline F)}_{<2n^2}-\underbrace{\sum_{E\text{ eye}}\Len_W(\partial E)}_{>\Len_W(\partial\overline M)}\\
&<2n^2-\Len_W(\partial\overline M)\\
&=2n^2-2n.
\end{align*}

To finish the construction we must introduce a smooth structure on the surface $\overline M$ and modify the wallsystem $W$ so that it becomes smooth rather than piecewise-differentiable. This can be done without affecting the area and boundary distances.
\end{proof}

\newpage
\section{Discretization of self-reverse Finsler metrics on surfaces}

In this section we will prove that the discrete FAC for walled surfaces implies the continuous FAC for surfaces with self-reverse Finsler metric, by showing how to convert a counterexample of the continuous FAC into a counterexample of the discrete FAC. This conversion is possible because every self-reverse Finsler metric can be approximated by a wallsystem, according to the theorem stated below.

To state the theorem, it is convenient to define scaled wallsystems. If $W$ is any wallsystem on a surface $M$ and $\delta\geq 0$ is a number (called a \term{unit of length}), then we define the \term{scaled wallsystem} $\delta W$, that is the same wallsystem $W$ scaled by the factor $\delta$, so that $\Len_{\delta W}(\gamma):=\delta\,\Len_W(\gamma)$ for any generic piecewise-differentiable curve $\gamma$ in $M$, and also $\Area(M,\delta W):=\delta^2\Area(M,W)$ and $\Area_\uHT(M,\delta W):=\delta^2\,\Area_\uHT(M,W)$.

\begin{maintheorem}[Discretization of self-reverse Finsler metrics on surfaces]\label{thm:discretize_selfreverse} Let $(M,F)$ be a compact Finsler surface with self-reverse metric, and let $\mu>1$ and $\varepsilon_{\mathrm{length}},\varepsilon_{\mathrm{area}}>0$. Then there exists a wallsystem $W$ on $M$ and a unit of length $\delta>0$ such that \begin{equation}\label{eq:area_approx}\Area_{\uHT}(M,F)\leq\Area_\uHT(M,\delta W)<\Area_{\uHT}(M,F)+\varepsilon_{\mathrm{area}}.\end{equation} and such that \begin{equation}\label{eq:length_approx}\minlen_F(\gamma)-\varepsilon_{\mathrm{length}}\leq\minlen_{\delta\,W}(\gamma)<\mu\,\minlen_F(\gamma)+\varepsilon_{\mathrm{length}}\end{equation} for each compact curve $\gamma$ in $M$ with no endpoints in $W$. If $\gamma$ is a closed curve, then we have the stronger inequality $\minlen_{\delta\,W}(\gamma)\geq\minlen_F(\gamma)$.
\end{maintheorem}

The right-hand side inequality of \eqref{eq:length_approx} requires a lot of work and is not used for the applications to systolic inequalities and filling area of the circle. Therefore we will prove directly that the discrete FAC implies the continuous FAC, and later we will explain the additional steps that are necessary to complete the proof of Thm.~\ref{thm:discretize_selfreverse}. For the moment, we will just discuss roughly how this kind of approximation can be obtained. 

Consider first the case in which $(M,F)$ is a $C^1$ Riemannian surface. By the Nash-Kuiper theorem, the surface can be $C^1$-immersed in Euclidean space $\RR^3$ preserving lengths. Therefore we assume that $M$ is an immersed $C^1$ surface in $\RR^3$, and the length of each curve in $M$ is the standard Euclidean length. Then the wallsystem approximation can be obtained as follows. For technical reasons, it is convenient to approximate our surface $M\subseteq\RR^3$ by a polyhedral surface (made of flat triangular faces). 
The walls are obtained by intersecting the new polyhedral surface with a large number of random planes (independent and uniformely-distributed\footnote{Each plane through the ball $B(0,r)$ is generated as follows: choose a random normal direction $v$ (with uniform distribution in the unit sphere) and a random $t\in[-r,r]$ (also with uniform distribution), and from these two data produce the plane $\{x\in\RR^3:\langle x,v\rangle=t\}$.}) that intersect a ball that encloses the surface. The expected number of intersections of $W$ with each curve $\gamma$ is proportional to the length of $\gamma$, by Crofton-Barbier's formula (the version for 3-dimensional Euclidean space), 
and the actual number is approximately proportional to the expected value, by the strong law of large numbers. The number of self-crossings of $W$ is also approximately proportional to the area of $M$, by another Crofton-Barbier formula for area. The proportionality ratios can be computed as quotients of certain integrals. With some technical work, 
one can complete the proof that every Riemannian metric can be approximated by a wallsystem.

If the metric on the surface is Finsler, no theorem says that the surface can be embedded preserving lengths into Euclidean or $L_1$ spaces (where Crofton formulas hold), therefore we approximate the surface by a Finsler polyhedral surface whose flat faces are triangles cut from different normed planes, and we work on each of these faces separately. As mentioned before, Crofton's formula has been generalized by Blaschke \cite{blaschke1935integralgeometrie} to normed planes where the metric is self-reverse. This means that for each such normed plane there is a (translation-invariant) measure on the set of lines, such that the length of each straight segment in the plane is proportional to the measure of the set of lines that cross it. 
The measure of the set of lines that cross any given triangle is finite (equal to the perimeter of the triangle, as happens for any convex figure), so restricting to this set we get a probability distribution that can be used to generate the straight walls through the triangle, and this wallsystem will approximate the Finsler distances correctly. Also, the expected number of self-crossings of the wallsystem in the triangle will be proportional to its Holmes--Thompson area, by the formula of Schneider and Wieacker (\cite{schneider1997integral}, see also the more concise \cite[Eq. 5,6,7]{schneider2001crofton}). This is enough to approximate areas correctly as well, and with some technical work one can prove, based on these facts of integral geometry, that every self-reverse Finsler metric can be approximated by a wallsystem as stated in Thm.~\ref{thm:discretize_selfreverse}. However, we will proceed in another way.

\subsection{Discretization of planes with self-reverse integral norms} 
The following lemma subsumes in a discrete form the integral-geometric formulas of Barbier-Crofton, Blaschke and Schneider-Wieacker discussed above. 
The lemma is due to Schrijver~\cite{schrijver1993graphs}, except for the last sentence about area.



\begin{lemma}[Discretization of self-reverse integral norms on the plane, or walled torus lemma]\label{thm:discretize_plane_selfrev} Let $K=-K\subseteq\RR^2$ be a symmetric integral convex polygon. Define on $\RR^2$ the self-reverse integral seminorm $\|v\|_K:=\max_{p\in K}\langle p,v\rangle$, whose dual unit ball is $K$. Then there exists on $\RR^2$ a \term{straight wallsystem} $\widetilde W$ (that is, a wallsystem made of straight walls) that is $\ZZ^2$-periodic and satisfies \begin{equation}\label{eq:length_wallsystem_plane}d_{\widetilde W}(x,x+v)=2\|v\|_K\quad\text{for every }x\in\RR^2\setminus\widetilde W\text{ and }v\in\ZZ^2.\end{equation}
More precisely, a straight $\ZZ^2$-periodic wallsystem $\widetilde W$ satisfies the last equation if and only if the wallsystem $W=\widetilde W/\ZZ^2$ on the torus $\TT^2=\RR^2/\ZZ^2$ consists of $k$ walls $(W_i)_{0\leq i<k}$ of respective homotopy classes $w_i=J(p_{i+1}-p_i)\in\ZZ^2$ (for $0\leq i<k$), where $J=\left(\begin{smallmatrix}0&-1\\1&\phantom{-}0\end{smallmatrix}\right)$ and $(p_i)_{0\leq i<2k}$ are the integral points of $\partial K$, listed in cyclic order.

Moreover, the number of self-crossings of such a wallsystem $\widetilde W$ in the unit square $[0,1]^2$ (assuming none of them lies on the border of the square) equals the area of $K$.
\end{lemma}


\begin{proof} Let $\widetilde W$ be a straight, $\ZZ^2$-periodic wallsystem on $\RR^2$, and consider its projection $W$ on the torus $\TT^2=\RR^2/\ZZ^2$. It is a wallsystem made of walls $(W_i)_{0\leq i<k}$. Let $w_i\in\ZZ^2$ be the homotopy class of each wall $W_i$ (with some arbitrary orientation); it is a primitive integral vector.\footnote{An integral vector $w\in\ZZ^d$ is called \term{primitive} unless it is an integer multiple $w=kw'$ (with $k>1$ integer) of an integral vector $w'\in\ZZ^d$. Note that in the plane, an integer vector $w\in\ZZ^2$ is primitive if and only if its coordinates $w_0,w_1\in\ZZ$ are coprime integers.} Reversing this process, any sequence of primitive integral vectors $w_i$ can be turned into a wallsystem $W$ on the torus and a periodic wallsystem $\widetilde W$ on the plane.

For each $v\in\ZZ^2$ and any $x\in\RR^2\setminus\widetilde W$, the distance $d_{\widetilde W}(x,x+v)$ is equal to the length $\Len_{\widetilde W}[x,x+v]$ of the straight segment from $x$ to $x+v$. We denote this length $\|v\|_W$, anticipating that it does not depend on $x$. (The skeptic reader may consider $x$ fixed.) We will see that this length depends only on the vector $v$ and the homotopy classes $w_i$ of the walls. After understanding the contribution of each wall $W_i$ to this length, we will be able to select the homotopy classes $w_i$ in order to obtain $\|-\|_W=2\|-\|_K$.

We first compute the length $\|v\|_W$ when $W$ consists of a single wall of certain class $w\in\ZZ^2$; note that the opposite vector $-w$ would represent the same wall. We claim that \begin{equation}\label{eq:cross-torus}\|v\|_W=|\langle Jw,v\rangle|=\max_{p\in[-Jw,Jw]}\langle p,v\rangle=\|v\|_{[-Jw,Jw]}.\end{equation} This is easy to see if $w=(1,0)$, and can be proved in general by reducing to the easy case. The reduction is done by applying to $\widetilde W$ and to $v$ an unimodular transformation of the plane\footnote{An \term{unimodular transformation} is a bijective linear trasformation of $\RR^n$ of the form $v\mapsto Av$, given by an invertible matrix $A\in M_n(\ZZ)$ such that $A^{-1}\in M_n(\ZZ)$ (which happens if and only if $|\det A|=1$). Such a transformation preserves the $n$-dimensional measure of sets, maps the lattice $\ZZ^n$ bijectively into itself, and induces a \term{linear automorphism} $[x]\mapsto[Ax]$ of the torus $\RR^n/\ZZ^n$. In the plane, any primitive integral vector $(p,q)\in\ZZ^2$ can be mapped unimodularly to $(1,0)$ using $A=\left(\begin{smallmatrix}\phantom{-}a&b\\-q&p\end{smallmatrix}\right)$, where $a,b\in\ZZ$ are chosen so that $ap+bq=1$.}, and noting that such a transformation does not modify the value of \[|\langle Jw,v\rangle|=|w_0v_1-w_1v_0|=|\det(w,v)|,\] which is the area of the parallelogram spanned by the vectors $v$ and $w$. With this analysis we see that the lemma holds in the case when $K$ is just a segment $[p,-p]$ where $p\in\ZZ^2$ is primitive, since such $p$ can be expressed in the form $p=Jw$. If $p$ is not primitive, then we can also prove the lemma by breaking $p$ as sum of $n$ copies of a primitive vector $p_0=Jw$, and let $W$ be made of $n$ straight walls of class $w$. So far, $W$ has no self-crossings and $|K|=0$.

Next in complexity is a wallsystem $W$ made of two straight walls whose homotopy classes are primitive integral vectors $w,w'\in\ZZ^2$. In this case, \begin{align*}\|v\|_W
&=\max_{p\in[-Jw,Jw]}\langle p,v\rangle+\max_{p'\in[-Jw',Jw']}\langle p',v\rangle\\
&=\max_{p\in[-Jw,Jw]\oplus [-Jw',Jw']}\langle p,v\rangle
\end{align*}
where $A\oplus B:=\{a+b:a\in A\text{ and }b\in B\}$ is the \term{Minkowski sum} of subsets of a vector space. With this wallsystem we can satisfy equation~\eqref{eq:length_wallsystem_plane} when $K$ is any parallelogram $[-Jw,Jw]\oplus[-Jw',Jw']$. Note also that the area of such parallelogram is $4|\det(w,w')|$, which is four times the number of crossings between the walls, so the lemma holds when $K$ is a parallelogram each of whose sides is the double of a primitive integral vector.

To prove the lemma in the general case we decompose the boundary of $K$ as a cyclic concatenation of vectors $Jw_i$, where $w_i\in\ZZ^2$ is a primitive integral vector for each $i\in\ZZ_{2k}$ (note that $w_{i+k}=-w_i$), so that $K=\bigoplus_{0\leq i<k}[-Jw_i,Jw_i]$. Then let the wallsystem $W$ be made of $k$ straight walls $W_i$ (for $0\leq i<k$), each $W_i$ of homotopy class $w_i$. 
The same argument used in the previous cases proves that $\|v\|_W=\|v\|_K$ here as well. Regarding areas, observe that the number of crossings of the wallsystem will be $\sum_{0\leq i<j<k}|\det(w_i,w_j)|$, where each term $|\det(w_i,w_j)|$ is one-quarter of the area of the parallelogram $[-Jw_i,Jw_i]\oplus[-Jw_j,Jw_j]$. The polygon $K$ can be tiled\footnote{To \term{tile} a plane figure is to decompose it as union of other figures whose intersection has zero area.} by translated copies of these parallelograms $[-Jw_i,Jw_i]\oplus[-Jw_j,Jw_j]$ (one copy of each); this is proved by induction in $k$ (exercise). 
Therefore, the number of crossings of $W$ equals $|K|$, as promised.
\end{proof}

\begin{figure}
  \centering
  \begin{subfigure}[b]{0.45\textwidth}
    \centering
    \def\svgwidth{.8\linewidth}
    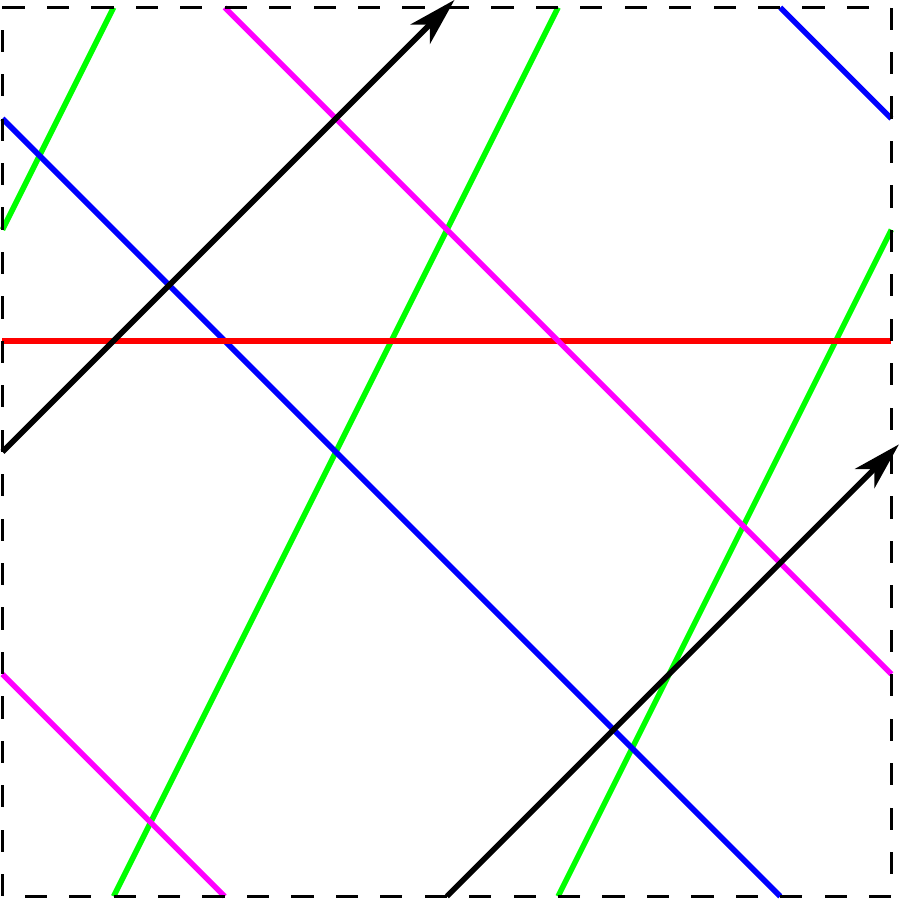
    \caption{Fundamental domain of a torus with a wallsystem $W$ made of straight walls of class $(1,0)$, $(1,2)$ and $2\times(-1,1)$. Also shown is a straight curve $\gamma$ of class $(1,1)$ and length $\Len_W(\gamma)=6$.}
  \end{subfigure}
  \hspace{.025\textwidth}
  \begin{subfigure}[b]{0.50\textwidth}
    \centering
    \def\svgwidth{\linewidth}
    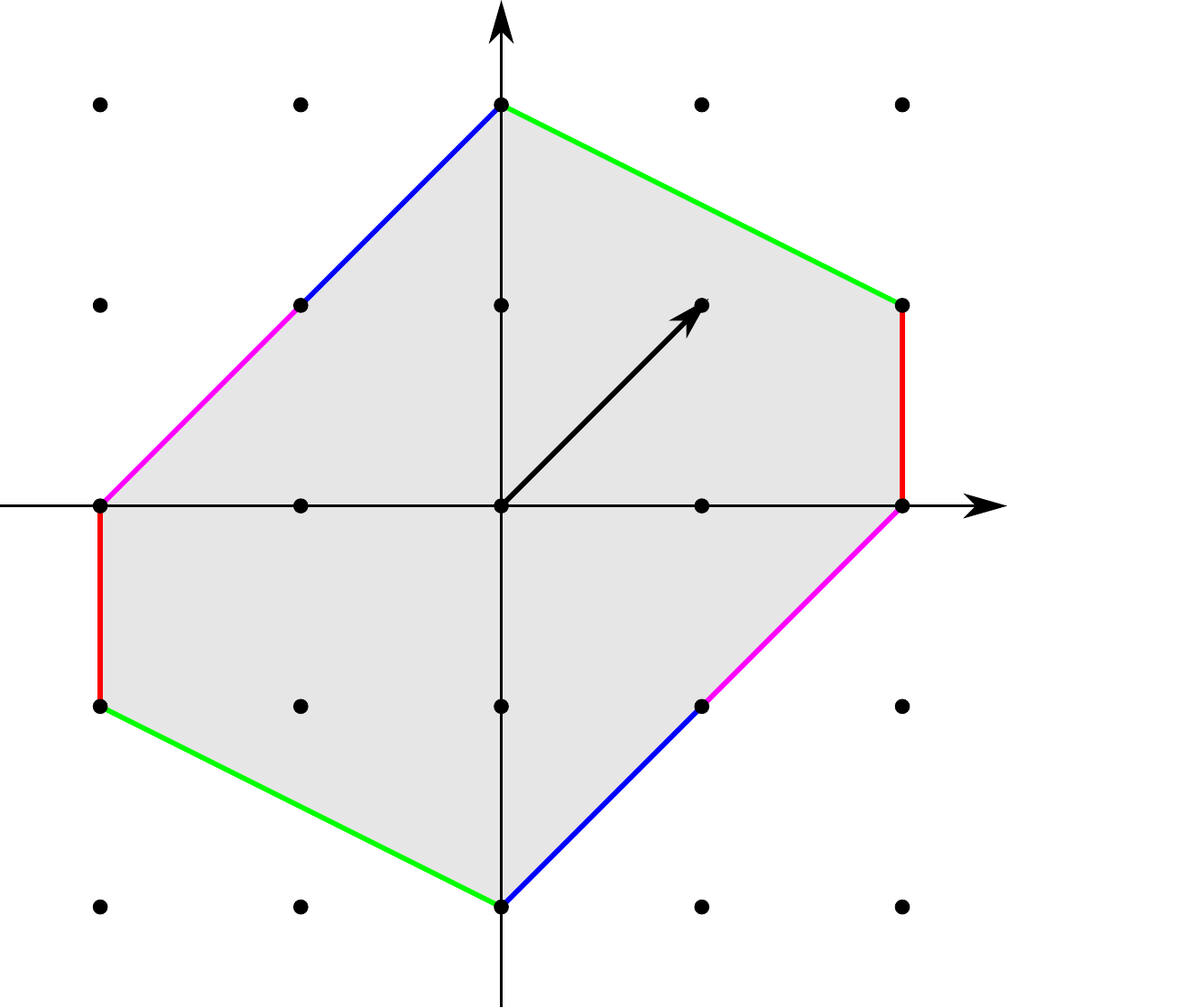
    \caption{Dual unit ball $K$ of the length function of $W$, and vector $v=[\gamma]=(1,1)$ with $\|v\|_K=\max_{p\in K}\langle p,v\rangle = 6$.}
  \end{subfigure}
  \caption{Example of a straight wallsystem on a torus and the dual unit ball of its length function.}
\end{figure}

\begin{remark}\label{rmk:square_cut_along_diag} We will later need to produce a wallsystem $\widetilde W$ as in the lemma, with the additional property that the wallsystem has the same number of self-crossings on each side of the diagonal $x+y=1$ of the unit square $[0,1]^2$, and no self-crossings on the sides or diagonal of the square. This is possible if $W$ has an even number of walls of each homotopy class (or, equivalently, if the vertices of $K$ have even coordinates), because in this case we can make $\widetilde W$ symmetric with respect to the center $\left(\frac 12,\frac 12\right)$ of the unit square. Note that we still have the freedom of lateral displacement for all walls, after their homotopy class has been specified and the symmetry imposed. We always choose generically so that no selfcrossings of $\widetilde W$ occur on the sides or the diagonal of the unit square.
\end{remark}

\subsection{Proof of the discretization theorem for self-reverse Finsler metrics on surfaces}

We will now prove Theorem~\ref{thm:discretize_selfreverse}. The construction of the wallsystem is very similar to the one of the previous proof, but it has an additional step to tidy up the result and ensure that we get a smooth wallsystem on the original surface $M$, and not a piecewise-smooth wallsystem on a modified surface homeomorphic to $M$.

Additionally, in each approximation step we will have to enforce some bounds that ensure that the minlengths of curves do not grow too much. The critical step is when we introduce the fences (that are immediately replaced by eyes). This increases the minlength of curves, because a curve that crosses fences must deviate to avoid the fences. To show that the minlength does not increase very much, we must show that in each homotopy class there is a curve whose length is nearly minimal, and that crosses the fences a reasonable number of times. For this purpose we will employ Lemma~\ref{thm:polyg_approx_curves}. We proceed to the details.

\begin{proof}[Proof of Thm.~\ref{thm:discretize_selfreverse}] Choose $\mu_{\mathrm{integ}},\,\mu_{\mathrm{polyg}},\,\mu_{\mathrm{portals}}>1$ and $\varepsilon_{\mathrm{polyg}},\,\varepsilon_{\mathrm{portals}}>0$ so that \begin{equation}\label{eq:u237d}\mu=\mu_{\mathrm{integ}}\,\mu_{\mathrm{polyg}}\,\mu_{\mathrm{portals}}\text{ and }\varepsilon_{\mathrm{length}}=\varepsilon_{\mathrm{polyg}}+\varepsilon_{\mathrm{portals}}.\end{equation}

\proofstep{Integral polyhedral approximation of the surface} There exists a triangulation $\overline M$ of $M$ and a polyhedral-Finsler metric $\overline F$ on $\overline M$ such that $F\leq\overline F\leq\mu_{\text{integ}} F$ and $\Area_\uHT(\overline M,\overline F)<A+\varepsilon_{\mathrm{area}}$, where $A:=\Area_\uHT(M,F)$. This follows from Lemma~\ref{thm:polyhed_approx}. Additionally, by Lemma~\ref{thm:polyhed_integral_approx}, we may ensure that the metric $\overline F$ is $\delta'$-integral for some $\delta'>0$.

\proofstep{Polygonal approximation of curves} By Lemma~\ref{thm:polyg_approx_curves}, there exist $\kappa,\alpha\geq 0$ such that every curve $\gamma$ is homotopic to a polygonal $\overline\gamma$ of length \begin{equation}\label{eq:49oeu8}\Len_{\overline F}(\overline\gamma)\leq\mu_{\mathrm{polyg}}\,\Len_{\overline F}(\gamma)+\varepsilon_{\mathrm{polyg}},\end{equation} made of \begin{equation}\label{eq:2390cuu3}k\leq\kappa\Len_{\overline F}(\gamma)+\alpha\end{equation} straight segements. 

\proofstep{Portals and eyes} Choose a large integer $m>0$. On each triangular face $T$ of $\overline M$, introduce linear coordinates so that $T=\Conv\{(0,0),(0,m),(m,m)$. Let $S$ be the set of integral points of all the faces of $\overline M$; they are called the integral points of the surface $\overline M$. The points of $S$ on the \emph{interior} edges of $\overline M$ are called ``portals'', and they divide each interior edge into $m$ segments called ``fences''. We slit the surface $\overline M$ along each fence and fill the hole with a Euclidean hemisphere called an ``eye''. Then we perform the lattice subdivision of each triangular face $T$ of $\overline M$ into $m^2$ small triangles $\widehat T$, each of them a translate copy of $\frac 1m T$ or $-\frac 1m T$. Let $(\widehat M,\widehat F)$ be the surface that results after subdividing $(\overline M,\overline F)$ in this way and inserting the hemispheric eyes. Note that $(\widehat M,\widehat F)$ is not a piecewise-Finsler surface, because it is not made of flat triangles and embedded in $\RR^k$. However, it is made of Finsler pieces, therefore we can define piecewise-differentiable curves, their length and the area of the surface.

We choose $m$ large enough so that $\Area_\uHT(\widehat M,\widehat F)<A+\varepsilon_{\mathrm{area}}$, and so that the diameters of each of the small triangles according to the metric $\overline F$ (which equals the length of the longest side) are $\leq\nu$, where $\nu>0$ is number small enough so that
\begin{equation}\label{eq:uu8923}4\nu<\varepsilon,\end{equation}
\begin{equation}\label{eq:3u23i87}\mu_{\mathrm{polyg}}+2\,\kappa\,\nu\leq\mu_{\mathrm{polyg}}\,\mu_{\mathrm{portals}}, \quad\text{ and}\end{equation}
\begin{equation}\label{eq:au4y98}2(\alpha+1)\nu\leq\varepsilon_{\mathrm{portals}}.\end{equation}

\proofstep{Replace Finsler metric by wallsystem} Note that if $T$ is a triangular face of $\overline M$, the length of any straight $[x,y]$ segment in $T$ whose endpoints are integral is an integer multiple of $4\delta$, where $\delta:=\frac{\delta'}{4m}$. By Lemma~\ref{thm:discretize_plane_selfrev}, we may replace the metric $\overline F_T$ by a straight wallsystem $W_T$, so that $\Area_\uHT(T,\delta W_T)=\Area_\uHT(T,\overline F)$ and $\Len_{\delta W_T}[x,y]=\Len_{\overline F}[x,y]$ if $[x,y]$ is a straight segment in $T$ with integral endpoints. We also replace the Euclidean metric in each eye $E$ by a wallsystem $W_E$ that matches, along the boundary $\partial E$, the wallsystems $W_T$, $W_T'$ of the triangular faces $T,T'$ of $\overline M$ that are separated by the eye $E$. Additionally, each wallsystem $W_E$ should fill isometrically its boundary of legnth $k$ and have $\Area(E,W_E)=\frac{k(k-1)}2$. Wallsystems of this kind have been constructed in Example~\ref{ex:discretized_euclidean_hemisphere}.

Let $\widehat W=\bigcup_{T\in\Sigma_{\overline M^2}}W_T\cup\bigcup_{E\textrm{ eye}}W_E$. Note that $\widehat W$ is not a wallsystem on $\widehat W$ because we have not defined the notion of wallsystem on this kind of piecewise-smooth surface, however, we can call $\widehat W$ a ``piecewise-smooth wallsystem'' and define the length $\Len_{\widehat W}(\gamma)$ of any generic piecewise-differentiable curve $\gamma$ on $\widehat M$ as the sum of the lengths of the pieces, and define the area $\Area(\widehat M,\widehat W)$ as the sum of the areas of the pieces of the surface $\widehat M$.

\proofstep{Tidy up} Note that we have a piecewise-smooth wallsystem $\widehat W$ on the modified surface $\widehat M$, and we promised a smooth wallsystem on the original surface $M$. In this step we will fix this by replacing each eye $E$ and its wallsystem $W_E$ by a more careful construction.

The first step is to remove all the eyes. We are left with the triangular faces $T$ of the polyhedral surface $\overline M$ and a union of wallsystems $\overline W=\bigcup_{T\in\Sigma_{\overline M}^2}W_T$. Note that $\overline W$ is not a wallsystem on $\overline M$ because it has endpoints on the interior edges of $M$. All these endpoints are located along the portals.

Let $I$ be a fence, and let $\widehat T,\widehat T'$ be the two small triangular faces of $\widehat M$ that share the edge $I$. Let $D_I\subseteq\widehat T\cup\widehat T'$ be a closed disk that is smoothly embedded in the surface $M$ and contains all the endpoints of $W$ that are located along the segment $I$ and contains no self-crossings of $\overline W$. We construct this disk $D_I$ for each fence $I$, making sure that $D_I\cap D_{I'}$ whenever $I\neq I'$.

Let $\widetilde M$ be the surface obtained from $M$ by removing the interios of all the disks $D_I$. Note that $\overline W\cap\widetilde M$ is a wallsystem on $\widetilde M$. Let $W$ be a wallsystem on $M$ that coincides with $\overline W$ on $\widetilde M$, and such that on each walled surface $(D_I,W\cap D_I)$ fills isometrically its boundary of length $2n$ and has $\Area(D_I,W\cap D_I)=\frac{n(n-1)}2$. Additionally, the curve $I\cap D_I$ should be a shortest curve in $D_I$ according to the wallsystem $W\cap D_I$.

This finishes the construction of the wallsystem $W$ on the surface $M$. We will now verify that it has the correct area and minlength function.

\proofstep{Bounding the area} We have to prove that $A\leq\Area(M,W)<A+\varepsilon_{\mathrm{area}}$, where $A=\Area_\uHT(M,F)$. When we replace the Finsler metric $F$ by the larger polyhedral-Finsler metric $\overline F$, we have $A\leq\Area_\uHT(\overline M,\overline F)<A+\varepsilon_{\mathrm{area}}$, as stated above. Then when we introduce the eyes, the area increases, but we still have $A\leq\Area_\uHT(\widehat M,\widehat F)<A+\varepsilon_{\mathrm{area}}$, as stated above as well. Next, when we replace the Finsler metric $\widehat F$ by a wallsystem, the area of the eyes decreases slightly, therefore we still have $A\leq\Area_\uHT(\widehat M,\widehat W)<A+\varepsilon_{\mathrm{area}}$. Finally, when we tidy up the wallsystem $\widehat W$ to obtain $W$, the area does not change, therefore $A\leq\Area(M,W)<A+\varepsilon_{\mathrm{area}}$, as we had to prove.

\proofstep{Lower bound for minlength of curves} Let $\gamma$ be a curve in $M$. We will first show that $\minlen_{\delta W}(\gamma)\geq\minlen_F(\gamma)$ if $\gamma$ is closed or has integral endpoints. This follows because the minlength of a curve of this kind does not diminish in any of the steps of the approximation process, which are
\begin{itemize}
\item replacement of Finsler metric $F$ by polyhedral-Finsler metric $\overline F$,
\item insertion of eyes to obtain new surface $(\widehat M,\widehat F)$,
\item replacement of metric $\widehat F$ by wallsystem $\delta\widehat W$, and
\item tidying up to obtain wallsystem $W$ on original surface $M$.
\end{itemize}
If $\gamma$ is a curve with any endpoints $x,y$, we will show that $\minlen_{\delta W}(\gamma)\geq\minlen_F(\gamma)-\varepsilon_{\mathrm{length}}$ as follows. Let $T,T'\in\Sigma_{\widehat M}^2$ be small triangles that contain $x$ and $y$, respectively, and let $x',y'$ be respective vertices of $T$ and $T'$. Note that the diameter of the triangles $T$ and $T'$ is $\leq\nu$ with respect to both $F$ and $\delta W$, therefore the curve $\beta=[x',x]*\gamma*[y,y']$ satisfies the inequalities \begin{align*}\minlen_F(\beta)&\leq\minlen_F(\gamma)+2\nu,\\\minlen_{\delta W}(\gamma)&\leq\minlen_{\delta W}(\beta)+2\nu.\end{align*} It follows that \begin{align*}\minlen_F(\gamma)
&\leq\minlen_F(\beta)+2\nu\\
&\leq\minlen_{\delta W}(\beta)+2\nu\quad\text{because }\beta\text{ has integral endpoints}\\
&\leq\minlen_{\delta W}(\gamma)+4\nu\\
&\leq\minlen_{\delta W}(\beta)+\varepsilon_{\mathrm{length}}\quad\text{by }\eqref{eq:uu8923},\end{align*} as we had to prove.

\proofstep{Upper bound for minlength of curves} We will show that every piecewise-differentiable compact curve $\gamma$ in $M$ is homotopic to a piecewise-differentable curve $\widetilde\gamma$ such that $\Len_{\delta W}(\widetilde\gamma)\leq\mu\Len_F(\gamma)+\varepsilon_{\mathrm{length}}$. This implies that $\minlen_W[\gamma]<\mu\minlen_F[\gamma]+\varepsilon_{\mathrm{length}}$ for every curve $\gamma$ in $M$.

We consider only the case in which $\gamma$ is a path that goes from a point $x$ to a point $y$, since the case of closed curves is easier. Note that \begin{equation}\label{eq:7823doe}\Len_{\overline F}(\gamma)\leq\mu_{\mathrm{integ}}\,\Len_F(\gamma)\end{equation} because $\overline F\leq\mu_{\mathrm{integ}}\,F$.

As stated above in \eqref{eq:49oeu8} and \eqref{eq:2390cuu3}, there exists a polygonal path $\overline\gamma=([x_i,x_{i+1}])_{0\leq i<k}$ in $\overline M$, homotopic to $\gamma$, made of \begin{equation}\label{eq:ufg3ie}k\leq\kappa\,\Len_{\overline F}(\gamma)+\alpha\end{equation} straight segments $[x_i,x_{i+1}]$, and of length \begin{equation}\label{eq:u43r09g}\Len_{\overline F}(\overline\gamma)\leq\mu_{\mathrm{polyg}}\,\Len_{\overline F}(\gamma)+\varepsilon_{\mathrm{polyg}}.\end{equation}

We will now approximate this polygonal path $\overline\gamma$ by a homotopic polygonal path $\widehat\gamma=[x_0,y_0]*([y_i,y_{i+1}])_{0\leq i<k}*[y_k,x_k]$ whose vertices $y_i$ are integral. The vertices $y_i$ are chosen as follows. For each $i$, let $T_i$ be the minimal simplicial face of $\widehat M$ that contains the point $x_i$, and let $y_i$ be any vertex of $T_i$. This implies that for each $y$, the point $y_i$ is contained in all the faces of $\overline M$ that contain $x_i$, therefore the face of $\overline M$ that contains the segment $[x_i,x_{i+1}]$ also contains the two vertices $y_i$, $y_{i+1}$ and the segment $[y_i,y_{i+1}]$. This implies that the curve $\widehat\gamma$ is homotopic to $\overline\gamma$. Note also that since the diameter (according to the metric $\overline F$) of the face $T_i$ is $\leq\nu$, we have $\Len_{\overline F}[x_i,y_i]\leq\nu$, therefore $\Len_{\overline F}[y_i,y_{i+1}]\leq\Len_{\overline F}[x_i,x_{i+1}]+2\nu$ for each $i$, which implies that \begin{equation}\label{eq:934nrotu}\Len_{\overline F}(\widehat\gamma)\leq\Len_{\overline F}(\widehat\gamma)+(2k+2)\,\nu.\end{equation}

When we replace the polyhedral surface $\overline M$ by the surface with eyes $\widehat M$, the curve $\widehat\gamma$ remains intact because it does not cross the fences. Furthermore, when we replace the metric $\overline F$ by the wallsystem $\delta\widehat W$, the length of the segments $[y_i,y_{i+1}]$ does not change because the points $y_i$ are integral. Finally, the start and end bits $[x_0,y_0]\subseteq T_0$ and $[y_k,x_k]\subseteq T_k$ still have $\Len_{\delta W}\leq\nu$ because the diameter of the triangles $T_0$ and $T_k$ according to the wallsystem $\delta\widehat W$ equals the length of the longest side, which has not changed when we replaced $\overline F$ by $\delta\widehat W$, therefore it is still $\leq\nu$. We conclude that after we replace $\overline F$ by $\delta\widehat W$, the length of $\widehat\gamma$ satisfies the same bound~\eqref{eq:934nrotu} as before, namely, \[\Len_{\delta\widehat W}(\widehat\gamma)\leq\Len_{\overline F}(\widehat\gamma)+(2k+2)\,\nu.\]

Finally, in the step of tyding up, the minlength of every curve that has integral endpoints (in particular, the polygonal $([y_i,y_{i+1}])_{0\leq i<k}$) remains unchanged, and the diameter of each small triangle $T\in\Sigma_{\widehat M}^2$ remains unchanged, therefore there exists a piecewise-smooth curve $\widetilde\gamma$, homotopic to $\widehat\gamma$, that satisfies the same length bound \[\Len_{\delta W}(\widetilde\gamma)\leq\Len_{\overline F}(\widehat\gamma)+(2k+2)\,\nu.\] Putting this together with the previous inequalities, we get
\begin{align*}\Len_{\delta W}(\widetilde\gamma)
&\leq\Len_{\overline F}(\widehat\gamma)+2(k+1)\nu\\
&\leq\mu_{\mathrm{polyg}}\,\Len_{\overline F}(\gamma)+\varepsilon_{\mathrm{polyg}}+2(\kappa\Len_{\overline F}(\gamma)+\alpha+1)\nu
\text{ by }\eqref{eq:u43r09g}\text{ and }\eqref{eq:ufg3ie}\\
&=\underbrace{(\mu_{\mathrm{polyg}}+2\kappa\nu)}_{\substack{\leq\mu_{\mathrm{polyg}}\,\mu_{\mathrm{portals}}\\\text{by }\eqref{eq:3u23i87}}}\underbrace{\Len_{\overline F}(\gamma)}_{\substack{\leq\mu_{\mathrm{integ}}\,\Len_F(\gamma)\\\text{by }\eqref{eq:7823doe}}}
+\varepsilon_{\mathrm{polyg}}+\underbrace{2(\alpha+1)\nu}_{\substack{\leq\varepsilon_{\mathrm{portals}}\\\text{by }\eqref{eq:au4y98}}}\\
&\leq\mu\,\Len_F(\gamma)+\varepsilon_{\mathrm{length}}\quad\text{by }\eqref{eq:u237d}
\end{align*}
as we had to prove.
\end{proof}

\newpage
\section{Continuization of square-celled surfaces}

In this section we prove that the continuous FAC for polyhedral-Finsler surfaces with self-reverse metric implies the discrete FAC for square-celled surfaces.

\begin{theorem}\label{thm:FinslerFAC_implies_walledFAC}
Let $M$ be a square-celled surface that fills isometrically its boundary of length $2n$ and has $\Area(M)=m<\frac{n(n-1)}2$. Then there is a polyhedral-Finsler surface $(\overline M,\overline F)$, homeomorphic to $M$, that fills isometrically its boundary of length $2L=2n$ and has $\Area_\uHT(\overline M,\overline F)=4m+2n<2L^2$.
\end{theorem}

The proof is similar to the proof of Thm.~\ref{thm:equiv_even_noneven_FAC}.

\begin{proof} Let $M$ be a square-celled surface that fills isometrically its boundary of length $2n$ and has $\Area(M)=m$. Draw the dual wallsystem $W$ in red. Consider each square cell $Q\in M^2$ as a copy of the standard unit square $[0,1]^2\subseteq\RR^2$. Endow the square $Q$ with the $\ell_1$ metric of $\RR^2$.

Additionally, for each wall $w$ that is not closed, do as follows. Choose an edge $e\in M^1$ that is intersected by that wall $w$. Slit the surface $M$ along $e$, and fill the resulting hole with a Euclidean hemisphere or one of the polyhedral-Finsler hemispheres constructed in Example~\ref{examp:polyhedral_hemispheres}. As before, this hemisphere $E$ inserted in $M$ is called an \term{eye}. After this interruption is done on each of the $n$ walls that are not closed, the modified surface is denoted $\overline M$. Note that the uHT area of each square is 4 and the uHT area of each eye is 2, as calculated in Example~\ref{ex:uHT_areas}, therefore the sum of the uHT areas of the squares and eyes is $4m+2n$.

The surface $\overline M$ is not exactly a polyhedral-Finsler surface, even if the eyes are polyhedral-Finsler, because the surface is made of squares and polyhedral pieces, rather than triangles, and it is not embedded in a vector space $\RR^q$. We may subdivide the squares and eyes into triangles and embed the surface $\overline M$ in $\RR^q$ to obtain a polyhedral surface $\overline M$ with a polyhedral-Finsler metric $\overline F$ that fills isometrically its boundary and has $\Area_\uHT(\overline M,\overline F)=4m+2n$, as promised in the statement of the theorem. However, to prove that $\overline M$ fills isometrically its boundary, it is convenient to leave it as it is, made of squares and eyes. For $k=0,1,2$, let $\overline M^k$ be the set of $k$-cells. The 0-cells are the vertices, the 1-cells are edges and the 2-cells are either square cells or hemispheric eyes. Each cell $Q$ has an insertion map $\iota_Q:Q\to\overline M$.

A \term{piecewise-differentiable curve} $\gamma$ on the surface $\overline M$ is a curve expressed as a concatenation of finitely many paths $\overline{\gamma_i}=\iota_{Q_i}\gamma_i$, where each $\gamma_i$ is a piecewise-differentiable curve in a cell $Q_i$ of $M$. The paths $\gamma_i$ are called the pieces of $\gamma$. A decomposition of a curve $\gamma$ of this kind into curves $\gamma_i$ is called a valid decomposition. The length of a piecewise-differentiable curve $\gamma$ in $M$ is the sum of the lengths of the pieces $\gamma_i$ of a valid decomposition of $\gamma$, and the length of each piece $\gamma_i\subseteq Q_i$ is measured using metric of $Q_i$, that is either the $\ell_1$ metric on $[0,1]^k$ or the piecewise-Finsler metric on the hemispheric eye.

To finish the proof we must show that the surface $\overline M$, with the definition of length given above, is an isometric filling of its boundary. Let $x,y\in\partial\overline M$, and let $\gamma$ be a piecewise-differentiable curve on $\overline M$ that goes from $x$ to $y$, decomposed validly into pieces $\gamma_i$ as explained above. We ask whether $\gamma$ is a shortcut. Note that if $\gamma$ is contained in the 1-skeleton graph $\overline M^{\leq 1}$, then it is not a shortcut, because the original square-celled surface is an isometric filling of its boundary.

\proofstep{Claim:} The curve $\gamma$ is not a shortcut, in fact, $\gamma$ can be homotopically pushed to the 1-skeleton graph $\overline M^{\leq 1}$ without increasing its length.

\begin{proof}[Proof of Claim] Recall that each piece $\gamma_i$ of the curve $\gamma$ is contained in a cell $Q_i$ that is a square cell or an eye or an edge of $\overline M$. (If $Q_i$ is a vertex, then the piece $\gamma_i$ can be omitted.) Color in red the pieces $\gamma_i$ such that $Q_i$ is a 2-cell (either a square cell or an eye), and color in blue the pieces $\gamma_i$ such that $Q_i$ is a 1-cell (an edge) of $\overline M$. Let $p_2$ be the number of red pieces and let $p_1$ be the number of blue pieces. Our objective is to homotopically modify $\gamma$ without increasing its length, until it is entirely blue.

Let $\gamma_i\subseteq Q_i$ be a red piece of $\gamma$ whose endpoints are contained in the boundary $\partial Q_i$ of the 2-cell $Q_i$. Note that $\gamma_i$ can be pushed to the boundary $\partial Q_i$ without increasing its length or modifying its endpoints if $Q_i$ is an eye or $Q_i$ is a square and the endpoints of $Q_i$ are contained in the same side or in adjacent sides of $Q_i$. The only other possibility is that $Q_i$ is a square and the endpoints of $\gamma_i$ are contained in opposite sides of this square and none of them is a vertex; in this case $\gamma_i$ will be called \term{wall-like}.

We apply repeatedly to $\gamma$ the following reduction process, until no longer possible:
\begin{enumerate}
\item IF it is possible to merge two consecutive pieces $\gamma_i$, $\gamma_{i+1}$ into one, keeping a valid decomposition of $\gamma$, we do so. Note that if it is not possible to merge consecutive pieces, then each red piece $\gamma_i\subseteq Q_i$ has its endpoints on the boundary $\partial Q_i$, otherwise the following or preceding piece $\gamma_{i\pm 1}$ would be contained in the same 2-cell $Q_i$ and it would be possible to combine the two pieces into one.
\item ELSE IF there is a piece $\gamma_i$ that is red (therefore, it has its endpoints on $\partial Q_i$) but not wall-like, then we push the piece $\gamma_i\subseteq Q_i$ to the boundary $\partial Q_i$ without increasing its length, as explained above, and decompose it as one or two blue pieces.
\end{enumerate}

Note that we can only apply this reduction process finitely many times, because each time we do so, one of the two numbers $r$, $r+2b$ is strictly reduced and the other does not increase. When we cannot reduce any more, we can say the following about the curve $\gamma$:
\begin{itemize}
\item Every red piece $\gamma_i$ is wall-like. This implies that $Q_i$ is a square and each endpoint of $\gamma_i$ is contained in one side of this square, but is not a vertex.
\item Two consecutive wall-like pieces cannot go back and forth along a square, otherwise it would be possible to combine them into one.
\item No red piece $\gamma_i$ is preceded or followed by a blue piece $\gamma_{i\pm 1}$. Otherwise, the edge $Q_{i\pm 1}$ would be the side of the square $Q_i$ that contains the endpoint/startpoint of $\gamma_i$, and it would be possible to merge the piece $\gamma_{i+1}$ into the piece $\gamma_i$.
\end{itemize}

We conclude that $\gamma$ is entirely blue or entirely made of wall-like red pieces. In the second case, the curve $\gamma$ must run parallel to a wall, that is, it must cross the same squares in the same order and direction as one of the $n$ non-closed walls of $W$. However, each wall has been interrupted by a hemispheric eye, therefore this possibility is excluded. We conclude that $\gamma$ is entirely blue, as we had to prove.
\end{proof}
This finishes the proof that (\ref{polyhedral_counterex}) implies (\ref{smooth_counterex}) in Thm.~\ref{thm:FinslerFAC_implies_walledFAC}.
\end{proof}

We can now put together the proof of Thm.~\ref{thm:FinslerFAC_implies_walledFAC}.

\begin{proof}[Proof of Thm.~\ref{thm:FinslerFAC_implies_walledFAC}] We have already proved that equivalence between the propositions (\ref{smooth_counterex}), (\ref{piecewise_counterex}) and (\ref{polyhedral_counterex}). Now we must show how to transforms a nonshortcutting filling $(M,F)$ of a curve $(C,G)$ of any of the classes (smoothly convex Finsler, piecewise-Finsler or polyhedral-Finsler) into an isometric filling $(\widehat M,\widehat F)$ of $(C,G)$ of the same class.

In all cases the surface $(\widehat M,\widehat F)$ will be obtained from $(M,F)$ by attaching a collar along the boundary $\partial M$, however, the details of the construction depend on the class of $(M,F)$. The collar is always a band of the form $B=C\times I$ where $I=[0,h]$ for some small number $h>0$. The surface $\widehat M$ is formed by joining $M$ with $B$ and gluing each point $x\in\partial M$ with the point $(x,h)\in B$. The boundary of $\widehat M$ consists of the points $(x,0)\in B$, that are identified with $x\in C$ so that $\partial\widehat M=C$.

If $M$ and $C$ are polyhedral manifolds (that is, when $(M,F)$ is piecewise-Finsler or polyhedral-Finsler), then the curve $C$ is made of edges $E$, and the band $B=C\times I$ is made of rectangles $E\times I$. In order to ensure that $\widehat M$ is a polyhedral surface, each of these rectangles $E\times I$ can be divided into two triangles by drawing a diagonal, but it is convenient to leave this step for the end. If $M$ and $C$ are $C^k$ manifolds, then the band $B$ is a $C^k$ surface that must be glued to $M$ in a $C^k$ way. We skip the details of this gluing.

On the band $B=M\times I$ we construct a metric $K$ such that
\begin{align}
K(v)&\geq G(v_0)\text{ for any }v_1\text{ and}\label{eq:2p9f89}\\
K(v)&=G(v_0)\text{ if }v_1=0\label{eq:f332e}
\end{align}
for any tangent vector $v=(v_0,v_1)\in T_{x_0,x_1}B=T_{x_0}C\times T_{x_1}I$. (Note that $T_{x_1}=I$.) The construction of $K$ is as follows. In the piecewise-Finsler or polyhedral-Finsler cases we may define $K(v):=\max\{G(v_0),|v_1|\}$. If the metrics are smoothly convex and self-reverse metrics, then we define $K$ as the Riemannian metric $K(v):=\sqrt{G(v_0)^2+v_1^2}$. If we deal with directed smoothly convex metrics, then we may define $K$ as a Randers metric\footnote{A \term{Randers metric} is a Finsler metric of the form $F(v):=\sqrt g(v,v)+\beta(v)$, where $g$ is a Riemannian metric and $\beta$ is a differential 1-form such that $\beta(v)<\sqrt g(v,v)$ for every $v$.} $K(v):=\sqrt{\overline G(v_0)^2+v_1^2}+G(v_0)-\overline G(v_0)$, where $\overline G(v_0):=\frac{G(v_0)+G(-v_0)}2$. In this way the conditions \eqref{eq:2p9f89} and \eqref{eq:f332e} are always met.

Finally, to construct the metric $\widehat F$ on the whole surface $\widehat M$, we must glue the metric $F$ defined on $M$ with the metric $K$ defined on $B$. In the piecewise-Finsler or polyhedral-Finsler cases the metric $\widehat F$ is simply defined so that it coincides with $F$ on $M$ and concides with $K$ on $B$. In the smoothly convex case, we can obtain a smoothly convex Finsler metric $\widehat F$ that is greater than $F$ on $M$ and greater than $K$ on $B$ with a technique similar to the one used to prove Lemma~\ref{thm:smoothing_piecewise-Finsler}: we first extend each of the metrics $F$ and $B$ beyond the region where they are defined, and then we multiply each metric by a smooth functions $M\to[0,1]$ that equals 1 in the original region of definition of the metric and vanishes at distance $\geq\delta$ from the region. We obtain two smooth semimetrics $F^\delta$ and $K^\delta$ on $\widehat M$ that must be added up to obtain the smoothly convex metric $\widehat F$.

It is easy to verify that the surface $(\widehat M,\widehat F)$ is an isometric filling of $(C,G)$, and that its area is $<A$ if the width $h$ of the band $B$ is small enough (and in the smoothly convex case, the gluing of the metrics is done using a small number $\delta$).
\end{proof}

\newpage
\section{Proof of the equivalence between discrete FAC and continuous FAC}\label{sec:equiv_discrete_continuous_FAC}

In this section we put together the results of the previous two sections.

\begin{theorem} The discrete FAC for walled surfaces is equivalent to the continuous FAC for Finsler surfaces with self-reverse metric, and the equivalence holds separately for each topological class of fillings of the circle.
\end{theorem} 

\begin{proof} Suppose there exists a counterexample to the continuous FAC, that is a Finsler surface $(M,F)$ with self-reverse metric and $\Area_\uHT(M,F)<2L^2$ that fills without shortcuts a Finsler curve $(G,C)$ of length $\Len_G(C)=2L$. This implies, by Thm.~\ref{thm:walledFAC_implies_FinslerFAC}, that there exists a counterexample to the discrete FAC, that is, a walled surface $(\overline M,W)$, homeomorphic to $M$, that fills isometrically its boundary of length $2n$ and has $\Area(M,W)<\frac{n(n-1)}2$.

Reciprocally, suppose there exists a walled surface $(\overline M,W)$ that fills isometrically its boundary of length $2n$ and has $\Area(\overline M,W)<\frac{n(n-1)}2$. Then by Thm.~\ref{thm:making_cellular} there exists a cellular walled surface $(\overline M',W')$, homeomorphic to $M$ or with simpler topology, that fills isometricaly its boundary of length $2n$ and has $\Area(\overline M',W')<\frac{n(n-1)}2$. Endow the surface $\overline M'$ with a square-celled decomposition dual to $W'$. By Thm.~\ref{thm:equiv_wallFAC_squareFAC}, the square-celled surface $\overline M'$ fills isometrically the cycle graph $C_{2n}$ and has $\Area(\overline M')=\Area(\overline M',W')<\frac{n(n-1)}2$.

By Thm.~\ref{thm:FinslerFAC_implies_walledFAC}, there exists a polyhedral-Finsler surface $(M',F')$, homeomorphic to $\overline M'$, that has $\Area_\uHT(M',F')<2L^2$ and fills without shortcuts its polyhedral-Finsler boundary of length $2L$. By Thm.~\ref{thm:FinslerFAC_implies_walledFAC}, we can transform this polyhedral-Finsler counterexample into a Finsler counterexample $(M,F)$, homeomorphic to $M'$. Note that topology of this surface is possibly simpler than the topology of of the origian surface $\overline M$. However, we can recover the original topology by performing a connected sum with small copies of the projective plane or the torus.
\end{proof}



\newpage
\section{Wallsystems on the disk}\label{sec:wallsystems_disk}

In this section we prove that the discrete FAC for walled disks using Steinitz's algorithm for transforming wallsystems into pseudoline arrangements. We also discuss a theorem of Lins \cite{lins1981minimax} that implies the FAC for walled disks. All the combinatorics presented here is classical. The geometric interpretation in terms of length and area is new.

If $(M,W)$ is a surface with a wallsystem, it is sometimes possible to reduce its area without affecting the distance between boundary points by performing certain operations $\Ste_i$, with $i=3,2,1,0$. These operations will be here called \term{Steinitz reductions} or \term{Steinitz operations}.\footnote{The operations where first used implicitly by Steinitz for answering the question of which graphs can be the 1-skeletons of convex polyhedra in $\RR^3$ \cite{steinitz1916polyeder,steinitz1934vorlesungen,grunbaum2003convex}. The operations were also discovered by \cite{schrijver1991decomposition}, who was studying the problem of routing wires of certain homotopy classes as edge-disjoint paths along a graph on a surface. The same operations were also employed to simplify resistor networks on a disk \cite{curtis1998circular,verdiere1996reseaux2}, so they are also called ``electrical moves''.} Each Steinitz reduction $\Ste_i$ modifies the wallsystem $W$ to get a new wallsystem $W'$, which only differs from $W$ on a small disk $D$. In each case, the small disk $(D,W\cap D)$ is an isometric filling of its boundary of length $2i$, and $(D,W'\cap D)$ is a new isometric filling that has the same or less area. The reductions are:

\[\begin{array}{rcccl}
\Ste_3:&\disktrigonup&\to&\disktrigondown&\text{ flip a clear trigon}\\
\Ste_2:&\diskbigon&\to&\diskcrossing&\text{ replace a clear bigon by a crossing}\\
\Ste_1:&\diskmonogon&\to&\diskline&\text{ eliminate a clear monogon}\\
\Ste_0:&\diskzerogon&\to&\diskempty&\text{ delete a clear zerogon.}
\end{array}\]
(A \term{clear polygon} in $(M,W)$ is an embedded closed disk $D\subseteq M$ such that $D\cap W=\partial D$. It is called a $k$-gon if it contains $k$ self-crossings of $W$ on its boundary.) 
Note that the $\Ste_2$ operation generally changes the homotopy classes of the walls. If instead we perform the operation 
\[\begin{array}{rcccl}
\Sho_2:&\diskbigon&\to&\diskuncrossedv&\text{ replace clear bigon by non-crossing segments,}
\end{array}\] that replaces two wall segments that cross twice by two wall segments that do not cross, then we get another set of operations $\Sho_i$ (where $\Sho_i:=\Ste_i$ for $i=0,1,3$) called \term{curve-shortening reductions} or \term{curve-shortening operations} that are used to reduce the number of crossings of a set of curves without changing their homotopy classes \cite{hass1994shortening}. 

According to Steinitz's Lemma~\ref{thm:Steinitz} stated below (and illustrated in Fig.~\ref{fig:tightening_wallsystem_disk}), the Steinitz reductions are capable of transforming any wallsystem on a disk into a \term{pseudoline arrangement (PLA)}, that is, a wallsystem on the disk whose walls are simple paths that cross each other at most once.\footnote{For information on pseudoline arrangements see \cite{levi1926teilung,felsner2018pseudoline,grunbaum1972arrangements}, but note that some authors define PLA only on the projective plane rather than on disks.}
(More generally, a wallsystem on any contractible (possibly non-compact) surface is a pseudoline arrangement if it does not contain closed walls, nor walls that have self-crossings, nor pairs of walls that cross each other more than once.) A pseudoline arrangement on the disk is called \term{complete} if every two walls cross, or, equivalently, if the endpoints of each wall are antipodal (they differ in $n$ units when we number the wall endpoints along the boundary cyclically, modulo $2n$).

\begin{figure}
  \centering
  \begin{subfigure}[t]{0.3\textwidth}
    \centering
    \includegraphics[width=.9\textwidth]{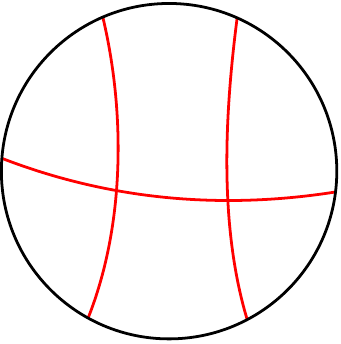}
    \caption{This disk with a wallsystem fills non-isometrically its boundary of length $2n=6$.}
  \end{subfigure}
  \hspace{.03\textwidth}
  \begin{subfigure}[t]{0.3\textwidth}
    \centering
    \includegraphics[width=.9\textwidth]{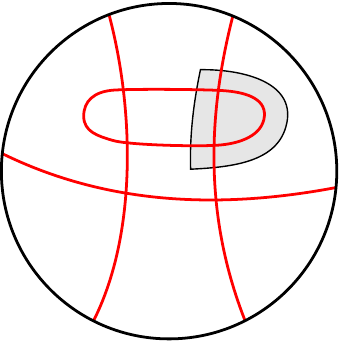}
    \caption{Here we added a wall so that the filling becomes isometric. The area of this filling can be reduced without affecting boundary distances by performing on the shaded region an $\Ste_2$ operation, that replaces an empty bigon (or double crossing) by a single crossing.}
  \end{subfigure}
  \hspace{.03\textwidth}
  \begin{subfigure}[t]{0.3\textwidth}
    \centering
    \includegraphics[width=.9\textwidth]{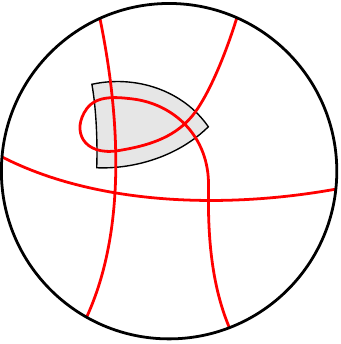}
    \caption{Now we could perform again an $\Ste_2$ operation to eliminate another bigon, but to show other possibilities, we perform instead on the shaded disk an $\Ste_3$ (or Yang-Baxter) operation, that flips a triangle. The area and boundary distances are not changed.}
  \end{subfigure}

  \begin{subfigure}[t]{0.3\textwidth}
    \centering
    \includegraphics[width=.9\textwidth]{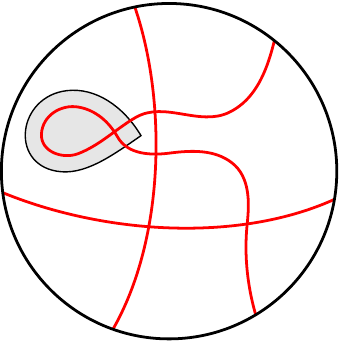}
    \caption{After performing the $\Ste_3$ operation, there is a small empty monogon in the shaded region, in which we can perform an $\Ste_1$ operation to eliminate it.}
  \end{subfigure}
  \hspace{.03\textwidth}
  \begin{subfigure}[t]{0.3\textwidth}
    \centering
    \includegraphics[width=.9\textwidth]{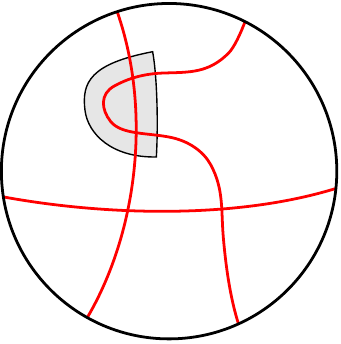}
    \caption{An empty bigon is left in the shaded region, and we perform again the $\Ste_2$ operation to eliminate it.}
  \end{subfigure}
  \hspace{.03\textwidth}
  \begin{subfigure}[t]{0.3\textwidth}
    \centering
    \includegraphics[width=.9\textwidth]{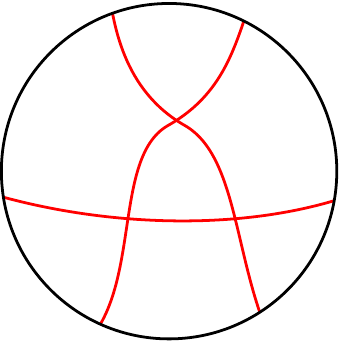}
    \caption{Our final wallsystem is a pseudo-line arrangement. It has $\frac{n(n-1)}2=3$ crossings, which is the smallest possible number for wallsystems on the disk that isometrically fill their boundary of length $2n=6$.}
  \end{subfigure}
  \caption{Some wallsystems on a disk: a non-isometric filling, an isometric but non-optimal filling, and a sequence of Steinitz reductions that reduce its area without affecting the isometric property.}
\label{fig:tightening_wallsystem_disk}
\end{figure}

We now state some important properties of wallsystems on disks.

\begin{lemma}[Steinitz\footnote{The techniques to prove this are due to Steinitz \cite{steinitz1916polyeder,steinitz1934vorlesungen}; see \cite{grunbaum2003convex}.}]\label{thm:Steinitz} Any wallsystem on a disk can be turned into a pseudoline arrangement using the elementary operations $\Ste_i$ or the operations $\Sho_i$.
\end{lemma}

\begin{lemma}[Levi\footnote{This is similar to Levi's enlargement theorem \cite{levi1926teilung} described in \cite{grunbaum1972arrangements}. The differences are that here we have a pseudoline arrangement on a disk rather than on the projective plane, and we state the existence of a curve from $x$ to $y$ that crosses no wall more than once, rather than a full new pseudoline that crosses each wall once.}]\label{thm:Levi} If $W$ is a pseudoline arrangement on a disk $M$, then the distance $d_W(x,y)$ between any two points $x,y\in M\setminus W$ is equal to the number of walls that separate $x$ from $y$. Moreover, any smooth generic curve from $x$ to $y$ can be homotoped generically until it becomes a shortest curve, without ever increasing its length.
\end{lemma}

\begin{lemma}[Ringel\footnote{This is Ringel's theorem \cite{ringel1956teilungen}.}]\label{thm:Ringel} If two pseudoline arrangements $W$, $W'$ on a disk connect the same pairs of boundary points, then they can be obtained from each other by $\Ste_3$ moves.
\end{lemma}

\begin{lemma}[Funicular formula\footnote{This formula appears in work by Sylvester \cite{sylvester1890funicular}, also in work on Hilbert's fourth problem by Alexander \cite{alexander1978planes} and Ambartzumian \cite{ambartzumian1976note}, and also in Arcostanzo's work on simple Finsler disks \cite{arcostanzo1992metriques,arcostanzo1994metriques}.}]\label{thm:funicular} If $W$ is a pseudoline arrangement on a disk $M$, and $[x,x']$, $[y,y']\subseteq \partial M$ are non-overlapping counterclockwise boundary segments with endpoints not on $W$, 
then the number $n_W(x,x',y,y')$ of walls of $W$ that go from $[x,x']$ to $[y,y']$ is determined by the \term{funicular formula} \[n_W(x,x',y,y')=\tfrac 12(d_W(x,y)+d_W(x',y')-d_W(x,y')-d_W(x',y)),\] In consequence, the boundary pairing of $W$ is determined by the distance function $d_W$.
\end{lemma}

For the convenience of the reader, we will collect the proofs of these lemmas in the next subsection. Now we will use the lemmas to show that the filling area conjecture is true for square-celled surfaces homeomorphic to the disk.

\begin{theorem}[FAC for square-celled disks]\label{thm:FAC_disks} Any square-celled disk $M$ that fills isometrically a cycle graph of length $2n$ has at least $\frac{n(n-1)}2$ square cells.
\end{theorem}

\begin{proof} Let $M$ be a square-celled disk that fills isometrically $C_{2n}$, and let $W$ be the dual wallsystem. According to Steinitz's Lemma~\ref{thm:Steinitz}, we can perform Steinitz operations $\Ste_i$ on $W$ so that the resulting wallsystem $W'$ is a pseudoline arrangement. We just need to show that it is complete, since this will imply that $W'$ has $\frac {n(n-1)}2$ crossings, and therefore $W$ has at least this many crossings as well. We know that $(M,W')$ is still an isometric filling of the boundary of length $2n$, since the Steinitz operations preserve the boundary distances.

Let $x,y\in\partial M\setminus W'$. According to Levi's Lemma~\ref{thm:Levi}, since $W'$ is a pseudoline arrangement, the distance $d_{W'}(x,y)$ equals the number of walls that separate $x$ from $y$. On the other hand, since $(M,W')$ is an isometric filling, the distance between the points $x$ and $y$ equals their distance along the boundary. In particular, if the boundary points $x,y$ are antipodal (that is, if their distance along the boundary is $n$), then they must be separated by all $n$ walls. This implies that each wall conects antipodal endpoints, so each wall crosses every other wall, as we had to prove. 
\end{proof}

(Note that the last paragraph, involving Levi's lemma, could be spared by applying the Funicular Formula~\ref{thm:funicular}. This does not make the proof more direct because, as we will see, the funicular formula is in turn a consequence of Levi's lemma.)

Pu's inequlity for square-celled surfaces (that every square-celled projective plane with systole $n$ has at least $\frac{n(n-1)}2$ cells) is equivalent to the last theorem\footnote{Proof: exercise. Similar to the continuous case.} and has been proved (in fact, in a stronger form) by Lins~\cite{lins1981minimax}. His proof is similar to the one given here, but his tightening process uses instead the operation of splitting crossings of the wallsystem. We will discuss this later after Steinitz algorithm, in Remark~\ref{thm:cusp_tightening}

Another proof of the discrete FAC for square-celled disks can be obtained using discrete differential forms to imitate Ivanov's proof \cite{ivanov2011filling}. This will be done in Section~\ref{sec:squarecelled_cyclic_content}.

The fact that the discrete FAC holds for square-celled disks can be generalized as follows.

\begin{theorem} Let $M$ be a square-celled disk that is tight (its dual wallsystem is a PLA). Then $M$ attains minimum area among all square-celled disks $M'$ that isometrically replace $M$ (those that have the same boundary and non-smaller boundary distances).
\end{theorem}

\begin{proof} By Lemma~\ref{thm:completing_PLA} below, we can extend $M$ to a square-celled hemisphere $\overline M$ (that is, a square-celled disk whose dual wallsystem is a \term{complete} PLA), formed by $\frac{n(n-1)}2$ squares. Let $M'$ be any square-celled disk with the same boundary and non-smaller boundary distances than $M$. Replace $M$ by $M'$ in $\overline{M}$, obtaining a new square-celled disk $\overline{M'}$ that fills $C_{2n}$ isometrically. By Theorem~\ref{thm:FAC_disks}, this disk $\overline{M'}$ has at least $\frac{n(n-1)}2$ cells. This implies that $M'$ cannot have less cells than $M$.
\end{proof}

\begin{remark} In fact, the same argument shows that the minimality of any tight square-celled disk $M$ holds among all isometric replacements $M'$ that have a topology for which the FAC is true (for example, Möbius bands, as we will see).
\end{remark}

\begin{lemma}[Completability of pseudoline arrangements]\label{thm:completing_PLA} Every pseudoline arrangement $W$ on a disk $M$ can be extended to a complete pseudoline arrangement $\overline{W}$ on a larger disk $\overline{M}\supseteq M$.
\end{lemma}

The proof of this lemma is also done in the next subsection.

\subsection{Steinitz's algorithm for tightening wallsystems on the disk}

The lemmas stated in the previous subsection will be proved here.

\begin{proof}[Proof of Steinitz's Lemma~\ref{thm:Steinitz}] On a disk $M$, let $W$ be a wallsystem that is \emph{not} a PLA.

\proofstep{Claim 0 (existence of a badgon):} The walls form the boundary of an embedded badgon in $M$.

(An \term{embedded $k$-gon} is a closed disk $B$ topologically embedded in the interior of $M$, whose boundary $\partial B$ is divided into $k$ \term{sides} at its \term{vertices} (the points where $\partial B$ is not smooth). The \term{inner walls} of an emedded polygon $B$ are the pieces of walls in the interior $B^\circ$, together with any endpoints they may have on $\partial B$, to make them compact curves. If a vertex $v$ of the polygon $B$ is an endpoint of an inner wall, then $B$ is called \term{concave} at $v$; otherwise $B$ is \term{convex} at $v$. An \term{embedded badgon} is an embedded $k$-gon with $k<3$ that is convex at its vertices. An embedded polygon is called \term{clear} if it has no inner walls.)

\proofstep{Claim 1 (classification of minimal badgons):} Moreover, any $B$ that is minimal among embedded badgons is either a clear zerogon, or a clear monogon, or a bigon whose inner walls go from one side to the other side forming a pseudo-line arrangement.

\proofstep{Proof of Claim 0:} To find an embedded badgon that is convex at its vertices, we do as follows:
\begin{itemize}
\item Case 0: If $W$ has a closed wall \emph{that is simple}, then this wall is the boundary of an embedded zerogon, by the Jordan-Schoenflies theorem.
\item Case 1: If $W$ has a non-closed wall $w:[a,b]\to M$ that crosses itself, then we travel along it starting at $a$, until the first time $t_1\in (a,b)$ in which we revisit a point that we have already visited at some time $t_0\in(a,t_1)$. The piece $w([t_0,t_1])$ is the boundary of an embedded monogon that is convex at its vertex.
\item Case 2: If all the walls are simple paths, then let $w:[a,b]\to M$ and $w'$ be two walls that cross each other more than once. The wall $w'$ divides $M$ into two parts, let $P$ be the one that does not contain $w(a)$. We travel along $w$ starting at $a$, and let $t_0$ be the first time in which we enter $P$, and let $t_1$ be the first time in which we exit $P$. The segment $w([t_0,t_1])$ together with a piece of $w'$, form the boundary of an embedded bigon that is convex at its vertices.
\item Case 0': The remaining possibility is that $W$ has a non-simple closed wall, $w:[a,b]\to M$ (parametrized so that $w(a)=w(b)$ is not a wall crossing). We travel along $w$ starting from $t=a$, as in Case 1, until the first time $t_1\in (a,b)$ in which we revisit a point that we have already visited at some time $t_0\in (a,t_1)$. The piece $w([t_0,t_1])$ is the boundary of an embedded monogon $B$, which may not be convex at its vertex. If it is convex, then we are done. If it is not, we travel along the curve $w$ during the interval $[t_0,t_1]$ (forming the boundary $\partial B$), and continue past time $t_1$ inside $B$ until the first point $t_2>t_1$ in which we revisit a point $w(s)$ that we have already visited at some time $s\in(t_0,t_2)$. There are three possibilities $s<t_1$, $s=t_1$ and $s>t_1$, but in each of these subcases, the curve $w([s,t_2])$ is the boundary of a convex monogon. (Check it by drawing.)
\end{itemize}

This finishes the proof Claim 0 that if the wallsystem is not a PLA, then we can find an embedded badgon $B$ that is convex at its vertices.

\proofstep{Proof of Claim 1:} Assume $B$ is minimal. The inner walls of $B$ must form a PLA (otherwise there would be a smaller badgon embedded in the interior of $B$). If $B$ is a zerogon or monogon, then any inner wall would form, together with part of the boundary of $B$, a convex bigon. If $B$ is a bigon, but one of the inner walls starts and ends at the same side of $\partial B$, then this wall together with a piece of the side form the boundary of a smaller bigon. This proves Claim 1 that any minimal badgon $B$ is either a clear zerogon or monogon, or is a bigon whose inner walls are pseudolines that go from one side to the other.

We now show how to simplify any wallsystems that is not a PLA. If there is a clear minimal badgon, then we can readily simplify the surface by an operation $\Ste_i$ with $i<3$. Otherwise there is a minimal bigon $B$ with walls inside, and we will use reversible operations $\Ste_3$ to clear its interior, so that the bigon can then be eliminated by applying the $\Ste_2$ operation. Moreover, the $\Ste_3$ operations will shrink the bigon $B$ by moving any chosen boundary side towards the other, without affecting the walls that are initially inside $B$. To apply the first (and subsequente) $\Ste_3$ operations, we must find some empty triangle along the chosen boundary side of $B$. This is possible due to the following fact.

\proofstep{Claim 2 (spelling out bigons,\footnote{I call it this way because it shows that one can read the $n$ pseudolines (that go from one side of the bigon to the other) as a sequence of adjacent transpositions (represented by crossings) that permute $n$ ordered objects. For more on this viewpoint see \cite{zhao2007bruhat,bjorner2005combinatorics}}, \cite[Lemma 1.2]{hass1994shortening}):} 
Let $T$ be an embedded triangle formed by walls, convex at its vertices $p,q,q'$. Assume the walls (at least one) inside $T$ go from side $pq$ to side $pq'$ forming a pseudoline arrangement (so the side $qq'$ has no crossings). Then each of the sides $pq$ and $pq'$ of $T$ contains a side of a clear triangle formed by inner walls of $T$.

Note that this claim can also be applied to a minimal bigon, which can be considered as a triangle $T$ whose side $qq'$ is degenerated to a single point.

To complete our proof of Steinitz's lemma, we reproduce the proof of claim 2 given by Hass-Scott.

\proofstep{Proof of claim 2:} It is by induction on the number $n$ of pseudolines inside $T$. The lemma clearly holds if $n=1$, so assume $n>1$. To find an empty triangle adjacent to the side $pq$, we travel along the edge $pq'$ until the first crossing point $p'$ with a pseudoline $w$. The pseudoline $w$ cuts from $T$ a smaller triangle $T'$ with vertices $p,p',r$, to which the lemma can be recursively applied, because it contains strictly less than $n$ pseudolines (since it doesn't contain $w$). So we know that a segment of $pr$ (that is contained in $pq$) is side of a clear triangle formed by walls and contained in $T'$ (and also in $T$, as we had to show).

This finishes the proof that on a disk, the number of crossings of any wallsystem that is not a PLA can be reduced using the operations $\Ste_i$ or the operations $\Sho_i$.
\end{proof}

\begin{proof}[Proof of Levi's Lemma~\ref{thm:Levi}] Let $\gamma$ be a curve from $x$ to $y$. If $\gamma$ does not cross any wall twice, then we are done. Otherwise, we will show how to homotopically shorten it. Let $\gamma'$ be a segment of $\gamma$ of the shortest possible length $d+2$ that does cross a wall $w$ twice. This means that $\gamma'$ crosses $w$, then crosses a sequence $(w_i)_{0\leq i<d}$ of walls that are different from each other and from $w$, and finally crosses $w$ again. Let $x'$ and $y'$ be the first and last crossing of $\gamma'$ with $w$. The curve $\gamma'$ and the segment of $w$ from $x'$ to $y'$ are the two sides of a bigon $B$, and the pieces of walls $w_i$ inside $B$ go from side to side forming a PLA. If $d=0$, the curve $\gamma'$ can be shortened by moving it and performing a $\Sho_2$ operation with $w$. Otherwise we employ Claim 2 (spelling out bigons), according to which there is a clear triangle in $B$ that has a side on $\gamma'$ and two sides on $W$, that meet at a vertex $v$. Then we can move $\gamma'$ across $v$, performing an $\Ste_3$ operation that reduces either the number of wall self-crossings inside $B$ or the number $d$. Eventually, $B$ becomes clear ($d=0$) and we can reduce the length of $\gamma'$ as before.
\end{proof}

\begin{proof}[Proof of Ringel's Lemma~\ref{thm:Ringel}] On a disk $M$, let $W,W'$ be two PLA that connect the same pairs of boundary points. We will show that $W$ can be transformed into $W'$ using the $\Ste_3$ operation. The proof is by induction on the number $n$ of pseudolines. When $n=0$ the claim is obvious. Assume the claim is true for certain $n$. Let $W$ have $n+1$ pseudolines $(w_i)_{0\leq i\leq n}$, and let $(w_i')_{0\leq i\leq n}$ be the pseudolines $W'$ (numbered so that $w_i'$ has the same endpoints as $w_i$. By induction hypothesis, if we delete the lines $w_n$ and $w_n'$, then we can transform $W\setminus\{w_n\}$ into $W'\setminus\{w_n'\}$ by a sequence of $\Ste_3$ moves. But we can also do so in presence of the additional pseudoline $w_n$, always moving it out of any triangle of $W\setminus\{w_n\}$ that needs to be flipped. At the end of these operations, we have the first $n$ pseudolines in the position $w_i'$, and only the last pseudoline $w_n$ differs from $w_n'$. We can simultaneously draw $w_n$ and $w_n'$, and observe that they form a sequence of bigons, 
crossed from side to side by some of the remaining $n$ pseudolines $w_i'$. On each of these bigons $B$, we can employ again Claim 2 (spelling out bigons) to move $w_n$ towards $w_n'$.
\end{proof}

\begin{remark} If we mantain a numbering of the pseudolines during the transformation of $W$ into $W'$, we may observe that some triangle (determined by 3 numbers) is flipped more than once during the process. This in general cannot be avoided, as shown in \cite[Fig. 3]{felsner2000theorem}.
\end{remark}

\begin{proof}[Proof of the Funicular Formula~\ref{thm:funicular}] Let $W$ be a pseudoline arrangement on a disk $M$, and let $x,x',y,y'\in\partial M\setminus W$ be boundary points in cyclical order. We have to show that the number of walls $n_W(x,x',y,y')$ that go from $[x,x']$ to $[y,y']$ is given by the formula \[2\,n_W(x,x',y,y')=d_W(x,y)+d_W(x',y')-d_W(x,y')-d_W(x',y).\] Each of the four distances that appear on the right can be computed using Levi's lemma. The contribution of each wall $w$ to the right hand side of the equation is 2 if the wall goes from $[x,x']$ to $[y,y']$, and 0 if $w$ has endpoints in only one or none of the two intervals $[x,x']$ and $[y,y']$.
\end{proof}

\begin{proof}[Proof that PLAs can be completed (Lemma~\ref{thm:completing_PLA})] Consider a square-celled disk $M$ that is tight, so its dual wallsystem is a pseudoline arrangement. If there are two walls that start at two consecutive edges $e,e'$ of the boundary but do not cross, then we can extend $M$ to a new tight disk by attaching a square along the boundary edges $e,e'$. To finish the proof we will prove the following
\proofstep{Claim:} If every pair of walls of a PLA that have consecutive endpoints cross, then the PLA is complete.
\proofstep{Proof:} Let $(a_i)_{0\leq i<2n}$ be the wall endpoints, 
and let $\tau$ be the involutive permutation\footnote{A permutation $\tau$ is involutive if $\tau(\tau(i))=i$ for every $i$.} of the numbers $0\leq i<2n$ such that every two points $a_i$, $a_{\tau(i)}$ are connected by a wall $w_i=w_{\tau(i)}$. It sufficies to show that $w_0$ crosses all the other walls $w_i$ (that have $i\neq 0,\tau(0)$), since the same argument would show that any wall crosses all other walls. By symmetry, it is enough to prove that $w_0$ crosses $w_i$ whenever $0<i<\tau(0)$. The proof is by induction on $i$. Since $w_1$ crosses $w_0$, it follows that $\tau(1)>\tau(0)$. Since $w_2$ crosses $w_1$, it follows that $\tau(2)>\tau(1)$. Repeating this we show that $\tau(i)>\tau(0)$ for all $0<i<\tau(0)$, so $w_i$ crosses $w_0$, as claimed.
\end{proof}

\begin{exer}\footnote{This is a theorem of \cite{kalmanson1975edgeconvex}.} 
Consider $\ZZ_{2n}$ as the set of vertices of the cycle graph $C_{2n}$. Show that a function $d:\ZZ_{2n}\times\ZZ_{2n}\to\NN$ is the boundary distance function of some square-celled disk if and only if it has the following properties:
\begin{itemize}
\item $d(x,x)=0$ for each $x\in\ZZ_{2n}$.
\item $d$ is eikonal on each of its variables. (An \term{eikonal 0-form} on a bipartite graph $G$ is a function $f:V(G)\to\ZZ$ on its set of vertices such that $f(v)-f(v')=\pm 1$ whenever the vertices $v,v'$ are neighbors.)
\item $d$ satisfies the \term{disk inequality}: $d(x,y)+d(x',y')\geq d(x,y')+d(x',y)$ whenever $x,x',y,y'\in\ZZ_{2n}$ are in cyclic order.
\end{itemize}
(Note that the disk inequality implies the positivity of $d$ and the triangle inequality.)
\end{exer}

\subsection{Tightening wallsystems on the disk by smoothing cusps}
Another way to tighten a wallsystem $W$ on a disk $M$ (turn it into a PLA without reducing boundary distances) was given by Lins \cite{lins1981minimax}. Instead of Steinitz reductions, it uses the operation of \term{splitting a crossing}, also called \term{uncrossing} \[\diskcrossing\to\diskuncrossedv.\] Note that each crossing admits two splittings, called \term{conjugate splittings}: \[\diskuncrossedh\from\diskcrossing\to\diskuncrossedv.\] The method of Lins is as follows. Assume that the wallsystem $W$ (more precisely, its image in $M$) has no contractible components. (These components can be eliminated at the beginning since they contribute nothing to the boundary distances.) If the wallsystem is not a PLA, then, as in Steinitz method, one can find a clear monogon (which is easy to eliminate by an appropriate splitting), or a minimal bigon, whose inner walls are pseudolines that go from side to side. The bigon can be eliminated by splitting one of its vertices in the appropriate way (that is, not in the way that turns the bigon into a monogon). This operation will be here called \term{cusp smoothing}.

\begin{lemma}[\cite{lins1981minimax}]\label{thm:cusp_tightening} For any minimal bigon of a wallsystem on a disk, cusp smoothing does not reduce boundary distances.
\end{lemma}

Lins' lemma will not be used for our theorems.

Comment: The method of Lins was extended by Schrijver \cite{schrijver1991decomposition} to any wallsystem $W$ on an \emph{orientable} surface $M$. The minimal bigons are found in the universal cover $\left(\widetilde M,\widetilde W\right)$, but the uncrossings are performed on $W$. Schrijver showed that cusp smoothing does not reduce the minlength function of $W$ \cite[Main~Lemma]{schrijver1991decomposition}.

\begin{remark} I call this operation ``cusp smoothing'' for the following reason. Let $M$ be a bigonal square-celled surface, that is, a square-celled disk whose dual wallsystem contains a minimal bigon with $n$ walls that go from side to side, and without any extra square cells outside the bigon. Let $a,b$ be the vertices of $M$ at the locations \[a\to\diskbigon\from b\] (Other walls are not shown in the figure.) Compute for $i=0,\dots$ the vertex set $V_i=\{x\in M^{\leq 1}:d(a,x)=i\}$. Or, better, in the continuous Finsler surface $M$ (made of square-cells with $\ell_1$ metric) compute the set $S_i=\{x\in M:d(a,x)=i\}$. The points of $V_i$ are the vertices of the polygonal curve $S_i$ called a \term{front},\footnote{We could call it ``BFS front'' to emphasize its discrete nature, where ``BFS'' stands for ``breadth-first search''.} made of diagonals of the square cells, that goes from one side of the bigonal surface $M$ to the other. This curve is simple for all but the last value $i=n+1$, when the front reaches the opposite vertex $b$ and folds on itself forming a cusp. This cusp is eliminated when we perform the cusp smoothing. (I have not drawn a figure.) 
The situation here is similar to what happens on a Riemannian surface $M$ when we fix a point $a$ and consider balls $B(a,r)$, starting with $r=0$ and increasing $r$ until we reach a first point $b$ that is conjugate to $a$. (Assume that the exponential map is injective until that value of $r$.) For this radius $r$, the boundary of the ball $B(a,r)$ develops a singularity (that is actually milder than a cusp\footnote{Generically, the front looks like the graph of $f(x)=x^{4/3}$ at $x=0$ 
\cite[Fig.~8]{arnold1995invariants}.
}). If the metric is slightly reduced on a small neighborhood of a vector tangent to the geodesic that goes from $a$ to $b$, then the conjugate point is eliminated or delayed (moved outside the ball $B(x,r)$, beyond $b$), and the singularity on $\partial B(a,r)$ is smoothed away. Note that the Riemannian metric becomes Finsler.\footnote{Questions: Is it possible to define geodesics and BFS fronts on a discrete disk that is not tight? Is it true that a square-celled disk is tight if and only if the shortest path problem admits the semigreedy solution (for every two vertices $x,y$, every shortest path from $x$ to $y$ can be obtained from any other path by elementary homotopies that do not increase length)?}
\end{remark}

\newpage
\section{Differential forms on square-celled surfaces}\label{sec:squarecelled_cyclic_content}

In this section we develop exterior calculus (integration, differentiation and exterior product of differential forms) on square-celled surfaces in order to construct a discrete version of Ivanov's proof of the FAC for Finsler disks (the version that appears in \cite{ivanov2011filling}, which differs from the first one \cite{ivanov2001two}). Unlike Ivanov's argument, this discrete version is limited to self-reverse metrics. A directed version is also possible but I have not yet written it down.

The definitions of discrete differential forms, integration and exterior derivative are straightforward. We will also give a notion of exterior product that is not associative, but satisfies the identity $\diff(f\theta)=(\diff f)\wedge\theta$ when $f$ is a 0-form and $\theta$ is a 1-form, which is sufficient for our purposes. This discrete calculus was discovered by Mercat working on the Ising model in statistical physics \cite[Def. 2]{mercat2001discrete} (see also \cite{mercat2001period,smirnov2010discrete})
, but we will give a self-contained account since what we need is simple.

\subsection{Chains, differential forms and integration} If $M$ is a square-celled surface, then for each $k=0,1,2$: \begin{itemize}
\item Let $\overrightarrow{M^k}$ be the set of \term{oriented $k$-cells} of $M$. In more detail, if $k=1$, the set $\overrightarrow{M^1}$ contains ordered pairs $v=[x,y]$, and we denote $-v=[y,x]$. Similarly, for $k=2$, the set $\overrightarrow{M^2}$ contains squares $Q=[x,y,z,w]$, and we consider two squares equal (resp. opposite) if they are related by a cyclic (resp. reversed cyclic) reordering of the vertices, so $[x,y,z,w]=[y,z,w,x]=-[w,z,y,x]$. Finally, for $k=0$, the set of signed vertices $\overrightarrow{M^0}$ contains the symbols $[x]$ and $-[x]$ for every vertex $x\in M^0$.
\item Let $\Sigma^k(M)=\Sigma^k(M,\QQ)$ be the vector space of (rational) \term{$k$-chains}, which are rational linear combinations of oriented $k$-cells. We regard $\overrightarrow{M^k}$ as a subset of $\Sigma^k(M)$.
\item Finally, let $\Omega^k(M)$ be the space of \term{$k$-forms}, which are functions $\theta:\overrightarrow{M^k}\to\QQ$ such that $\theta(-\sigma)=-\theta(\sigma)$ for each oriented $k$-cell $\sigma$.
\end{itemize}

For each $k$-form $\theta\in\Omega^k(M)$ and for each $k$-chain $S=\sum a_i\sigma_i$ we define the \term{integral} \[\int_S\theta:=\sum_ia_i\theta(\sigma_i).\] Note that $\int_{(-)}\theta$ is a $k$-cochain, that is, a linear functional on the vector space of $k$-chains, and in this way $k$-forms are in fact equivalent to $k$-cochains.\footnote{Compare this with the situation in differential geometry, where one finds great difficulties when attempting to define or characterize differential $p$-forms as cochains with certain properties. See for example the introduction to Whitney's book \cite{whitney1957geometric} and Federer's review~\cite{federer1958geometric} of the same book.} 
However, we will retain the distinction to enable the notation of integrals, which is familiar and suggestive to differential geometers.

\subsection{Boundaries, exterior derivatives and Stokes' formula} We define the \term{boundary} of directed edges and squares: if $v=[x,y]\in\overrightarrow{M^1}$ then $\partial v=y-x\in\Sigma^0(M)$, and if $Q=[x,y,z,w]\in\overrightarrow{M^2}$, then $\partial Q=[x,y]+[y,z]+[z,w]+[w,x]$. The boundary maps extend to linear operators \[\Sigma^2(M)\stackrel{\partial}{\to}\Sigma^1(M)\stackrel{\partial}{\to}\Sigma^0(M),\] which allow us to define also the adjoint \term{exterior derivative or differential} operators \[\Omega^2(M)\stackrel{\diff}{\from}\Omega^1(M)\stackrel{\diff}{\from}\Omega^0(M)\] by Stokes' formula \begin{equation}\int_\sigma\diff\theta:=\int_{\partial\sigma}\theta.\label{eqn:Stokes}\end{equation}

In more detail, if $f\in\Omega^0(M)$ is a 0-form and $v=[x,y]\in\overrightarrow{M^1}$, we define $(\diff f)(v):=\int_{\partial v}f=f(y)-f(x)$, and the Stokes formula follows. For example, any path $\gamma=(x_0,x_1,\dots,x_n)$ yields a 1-chain $\sum_{0\leq i<n}[x_i,x_{i+1}]$, that we will also denote $\gamma$. Then \[\int_\gamma\diff f=\sum_i f(x_{i+1})-f(x_i)=f(x_n)-f(x_0)=\int_{\partial\gamma}f,\] as expected. 

If the square-celled surface $M$ is oriented, then any finite set of squares $A\subseteq M^2$ yields a 2-chain $\sum_{Q\in A} Q$, also denoted $A$, where each square $Q$ has the orientation induced from $M$. In particular, $M$ is a 2-chain and $\int_M\diff\theta=\int_{\partial M}\theta$ for any 1-form $\theta\in\Omega^1(M)$.

\subsection{Exterior products} If $f$ is a 0-form and $\theta$ is a 1-form, then we define a 1-form $f\theta$ by the formula \[(f\theta)(v):=\frac{f(x)+f(y)}2\theta(v)\text{ for each }v=(x,y)\in\overrightarrow{M^1}.\] (Note that the equation $(fg)\theta=f(g\,\theta)$ in general won't hold if $f,g$ are functions and $\theta$ is a 1-form.) Similarly, if $\nu$ is a 2-form, define the 2-form $f\nu$ by the formula \[(f\nu)(Q):=\frac{f(a)+f(b)+f(c)+f(d)}4\,\nu(Q)\] for each square $Q=(a,b,c,d)$. Finally, if $\phi,\psi$ are 1-forms, we let their \term{exterior product} $\phi\wedge\psi$ be the 2-form determined by the formula \[(\phi\wedge\psi)(Q)=\frac{\phi(v)+\phi(v')}2\,\frac{\psi(w)+\psi(w')}2-\frac{\phi(w)+\phi(w')}2\,\frac{\psi(v)+\psi(v')}2,\] for any oriented square $Q=(a,b,c,d)\in\overrightarrow{M^2}$ with sides $v=[a,b]$, $v'=[d,c]$, $w=[a,d]$ and $w'=[b,c]$ (see Fig.~\ref{fig:square_Q}). Note that this value is unchanged (resp. changes sign) if we permute the letters $a,b,c,d$ cyclically (resp. anticyclically), so the 2-form is well-defined. Also observe that exterior product of 1-forms is bilinear and antisymmetric.

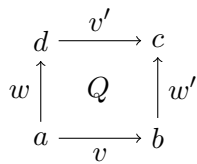
\begin{figure}\centering
\begin{tikzpicture}[every node/.style={midway}]
\matrix[column sep={4em,between origins}, row sep={2em}] at (0,0) {
\node(d) {$d$}; & \node(c) {$c$};\\
\node(a) {$a$}; & \node(b) {$b$};\\};

\draw node {$Q$};

\draw[->] (a) -- (b) node[anchor=north]  {$v$};
\draw[->] (b) -- (c) node[anchor=west]  {$w'$};
\draw[<-] (c) -- (d) node[anchor=south]  {$v'$};
\draw[<-] (d) -- (a) node[anchor=east]  {$w$};
\end{tikzpicture}
\caption{An oriented square called $Q$.}\label{fig:square_Q}
\end{figure}

\begin{lemma} If $f$ is a 0-form and $\theta$ is a 1-form, then \[\diff(f\theta)=\diff f\wedge\theta+f\diff\theta.\]\end{lemma}

(In particular, if $f,g$ are 0-forms, then $\diff(f\wedge dg)=\diff f\wedge\diff g$ because $\diff\diff g=0$.)

\begin{proof} To show that the 2-forms $\diff(f\theta)$ and $\diff f\wedge\theta+f\diff\theta$ are equal, we evaluate them on a square $Q=(a,b,c,d)$. Define $v,v',w,w'$ as above. Then \begin{align*}(\diff(f\theta))(Q)
=&\int_{\partial Q}f\theta\\
=&\phantom{+}\frac{f(a)+f(b)}2\theta(v)+\frac{f(b)+f(c)}2\theta(w')\\
 &-\frac{f(c)+f(d)}2\theta(v')-\frac{f(d)+f(a)}2\theta(w),\end{align*}
is equal (check it!) to the sum of \begin{align*}(\diff f\wedge\theta)(Q)
&=\frac{\diff f(v)+\diff f(v')}2\;\frac{\theta(w)+\theta(w')}2
 -\frac{\diff f(w)+\diff f(w')}2\;\frac{\theta(v)+\theta(v')}2\\
&=\phantom{+}\frac{f(b)-f(a)+f(c)-f(d)}2\;\frac{\theta(w)+\theta(w')}2\\
&\phantom{=}-\frac{f(d)-f(a)+f(c)-f(b)}2\;\frac{\theta(v)+\theta(v')}2\\
\end{align*} and
\begin{align*}(f\diff \theta)(Q)
&=\frac{f(a)+f(b)+f(c)+f(d)}4\big(\theta(v)+\theta(w')-\theta(v')-\theta(w)\big),\end{align*} as we had to show. 
\end{proof}

The distance function to a fixed vertex is always an \term{eikonal 0-form}, that is, a 0-form $f$ such that $\diff f(v)=\pm 1$ for every $v\in\overrightarrow{M^1}$. An eikonal 0-form is called \term{regular} (resp. \term{singular}) on a square $Q$ if it attains values $a,a+1,a+2,a+1$ (resp. $a,a+1,a,a+1$) on the vertices of $Q$; these are the only two possibilities. For our main theorem we need to compute the exterior product of the differentials of two eikonal 0-forms. 

\begin{lemma}\label{lemma:wedge_value_square} Let $Q=(y_j)_{j\in\ZZ_4}$ be an oriented square and let $f,g$ be eikonal 0-forms on $Q$ that attain their minimum values at $y_i,y_j$ respectively. Then \[\frac 12(\diff f\wedge\diff g)(Q)=\begin{cases}
+1&\text{if }f,g\text{ are regular and }j=i+1\\
-1&\text{if }f,g\text{ are regular and }j=i-1\\
\phantom{+}0&\text{in other cases}.\end{cases}\]
\end{lemma}

\begin{proof} Let $Q=(a,b,c,d)$ and let $v,v',w,w'$ as before (see Fig.~\ref{fig:square_Q}). If one of the 0-forms, say, $f$, is not regular on $Q$, then $\diff f(v)+\diff f(v')=0$ so $(\diff f\wedge\diff q)(Q)=0$. Now assume that both $f$ and $g$ are regular. If they attain their minimum value (restricted to $Q^0$) on equal (or opposite) vertices of $Q$, then they have the same differential (or opposite), and again $(\diff f\wedge\diff q)(Q)=0$ because exterior product is antisymmetric. The only remaining possibility is that $f$ and $g$ attain their minimum on consecutive vertices of the square, say, $a$ and $b$, respectively. Then \begin{align*}(\diff f\wedge\diff g)(Q)
=&\underbrace{\frac{\diff f(v)+\diff f(v')}2}_{=1}\underbrace{\frac{\diff g(w)+\diff g(w')}2}_{=1}\\
&-\underbrace{\frac{\diff f(w)+\diff f(w')}2}_{=1}\underbrace{\frac{\diff g(v)+\diff g(v')}2}_{=-1}\\
=&2.
\end{align*} The remaining cases are obtained from this one by symmetries.
\end{proof}

\subsection{Proof of the FAC for square-celled disks using cyclic content}

Now we can prove our main theorem of this subsection.

\begin{theorem}\label{thm:cyclic_content_below_area} Let $M$ be an oriented square-celled disk, and let $P=(x_i)_{i\in\ZZ_m}$ be a sequence of boundary points in positive cyclic order. Define the (discrete) Ivanov \term{cyclic form} \[\omega_{\textrm{cyclic}}^P:=\sum_{i\in\ZZ_m}\frac 12\;\diff f_{x_i}\wedge\diff f_{x_{i+1}},\] where $f_x$ is the 0-form given by $f_x(y):=d(x,y)$. Then \[\int_M\omega_{\textrm{cyclic}}^P\leq 4\left|M^2\right|,\] with equality if only if $M$ is tight and $P$ is a sufficiently full subset of the boundary. ($P$ is sufficiently full if it contains points of each of the four \term{quadrant regions} in which the tight disk is divided by each pair of walls that cross.) \end{theorem}

Before the proof, observe that the integral depends only on the distances between boundary points. Indeed, Ivanov's cyclic form is the exterior derivative of the 1-form $\sum_{i\in\ZZ_m}\frac 12 f_{x_i}\wedge\diff f_{x_{i+1}}$, so by Stokes' formula \[\int_M\omega_{\textrm{cyclic}}^P=\int_{\partial M}\left(\sum_{i\in\ZZ_m}\frac 12 f_{x_i}\wedge\diff f_{x_{i+1}}\right),\] and the right hand side depends only on the distances between boundary points.

\begin{proof} Let $Q=(y_j)_{j\in\ZZ_4}$ be a square cell of $M$, and for each $j\in\ZZ_4$, let $R_j\subseteq M^0$ be the set containing each vertex $x\in M^0$ that is strictly closer to $y_j$ than to any other vertex of $Q$, so that the function $f_x$ restricted to $Q$ is regular, with minimum value at $y_j$. Then by Lemma~\ref{lemma:wedge_value_square}, \begin{align*}\omega_{\textrm{cyclic}}^P(Q)
=\sum_{j\in\ZZ^4}\big(\phantom{+}&\#\{i\in\ZZ_m:x_i\in R_j\text{ and }x_{i+1}\in R_{j+1}\}\\
-&\#\{i\in\ZZ_m:x_i\in R_j\text{ and }x_{i+1}\in R_{j-1}\}\big).\end{align*}

If $M$ is tight, then the sets $R_j\ni y_j$ are the quadrant regions into which $M^0$ is divided by the two walls that cross at $Q$, and we see that $\omega_{\textrm{cyclic}}^P(Q)=4$ if $P$ visits the four regions $R_j$. (If $P$ visits 3 regions, then the result is 2, and if it visits 2 regions or less, then the result is 0.) We conclude that $\int_M\omega_{\textrm{cyclic}}^P=4\left|M^2\right|$ if $M$ is tight and $P$ is sufficiently full.

If $M$ is not necessarily tight, then we will prove that $\omega_{\textrm{cyclic}}^P(Q)\leq 4$ by showing that, for each $j$, the set $\{i\in\ZZ_m:x_i\in R_j\text{ and }x_{i+1}\in R_{j+1}\}\subseteq\ZZ_m$ contains at most one index $i$, because there are no four points $z_0,z_1,z_2,z_3\in\partial M$, in cyclical order along the boundary, such that $z_0,z_2\in R_j$ and $z_1,z_3\in R_{j+1}$. Indeed, assume these points $z_k$ do exist. Choose respective shortest curves $\gamma_k$ from each $z_k$ to $Q$. Observe that the curves $\gamma_0,\gamma_2$ are paths along $R_j$ ending at $y_j$, so they can be joined to obtain a curve along $R_j$ connecting $z_0$ to $z_2$. In the same way, we have a curve along $R_{j+1}$ connecting $z_1$ to $z_3$. These two curves must intersect because their endpoints are interlaced, but this is not possible because $R_j\cap R_{j+1}=\emptyset$. This proves that $\int_M\omega_{\textrm{cyclic}}^P\leq 4\left|M^2\right|$.

Finally, it remains to be proved that the inequality is strict if $M$ is not tight. This follows from the fact that one can tighten the surface $M$, obtaining a tight disk $M'$ with strictly smaller area but the same boundary distances, and hence the same value of the integral \[\int_M\omega_{\textrm{cyclic}}^P=\int_{M'}\omega_{\textrm{cyclic}}^P\leq 4\left|M'^2\right|<4\left|M^2\right|.\]
\end{proof}

\begin{remark} Observe that we only used tightening to prove that the inequality is strict in the case of non-tight disks.
\end{remark}

\begin{corollary} If a tight disk $M$ is replaced by a disk $\widetilde M$ with the same boundary $\partial\widetilde M=\partial M$ and boundary distances, then $\left|\widetilde M^2\right|\geq\left|M^2\right|$, and the inequality is strict if $\widetilde M$ is not tight.
\end{corollary}

\begin{proof} By Theorem~\ref{thm:cyclic_content_below_area}, \[4\left|\widetilde M^2\right|\geq\int_{\tilde M}\omega_{\textrm{cyclic}}^P=\int_{M}\omega_{\textrm{cyclic}}^P=4\left|M^2\right|\] and the inequality is strict if $\widetilde M$ is not tight.
\end{proof}

\begin{exer} If two tight disks $M,M'$ with the same boundary $C_{2n}$ have \[d_M(x,y)\leq d_{M'}(x,y)\text{ for every two vertices }x,y\in C_{2n},\] can we conclude that $\Area(M)\leq\Area(M')$? And if additionaly some $d_M(x,y)<d_{M'}(x,y)$, can we conclude that $\Area(M)<\Area(M')$?\footnote{Answers: yes and yes.}
\end{exer}

For orientable surfaces of genus 1 and with one boundary component, it is not true that the area is greater than the cyclic content. An example is a surface obtained as follows.

\begin{example}[twisted pants] From the Euclidean plane cut two equilateral triangles $T=ABC$, $T'=A'B'C'$  of the same size and stack them one on top of the other, with $A'$ on top of $A$, etc. Then cut away the corners of both triangles in the most symmetric way, so that each vertex becomes a small side (denoted by the same letter). Finally, glue the side $A$ with $A'$, not in the easy way, but with a twist of half a turn, as if trying to make a Möbius band, and do the same on the other corners, gluing $B$ with $B'$ and $C$ with $C'$. The resulting surface is orientable because all gluings have the twist. And its area is less than its cyclic content. (This counterexample is provided without proof, just for information.)
\end{example}

A square-celled version of the twisted pants can be constructed as well, but this is left as an exercise for the reader.

Before knowing this examples, I tried to prove the Finsler FAC by showing that the area is greater than the cyclic content. My conjecture was that $\Area\geq C_0+C_1+\dots$, where $C_0$ is the cyclic content and the $C_k$'s are ``higher cyclic contents'', defined below. This last formula may still be true, but it does not imply that the area is greater than $C_0$ because the higher cyclic contents may be negative (which is what happens in a square-celled version of the twisted pants described above).

\subsection{Higher cyclic contents}
Here I describe an idea for which I have not found any application so far.

Let $M$ be a compact orientable surface with a Finsler metric. For each pair of points $x,y\in M$, let \[e_0(x,y)\leq e_1(x,y)\leq\dots\leq e_k(x,y)\leq\dots\] be the minlengths of all homotopy classes of curves from $x$ to $y$, listed in increasing order. The number $e_k(x,y)$ is called the $k$-th \term{echo} from $x$ to $y$. Note that the zeroth echo is the distance $d(x,y)$. Define the $k$-th cyclic content $C_k$ in a similar way as the cyclic content, but computed using $e_k$ instead of $d$. More precisely: \[C_k=\lim_P\int_M\omega_{\textrm{cyclic}_k}^P,\] where $P$ is a partition of the boundary at points $x_i$ that is taken to the limit by refinement, and $x_{i'}$ is the point after $x_i$ in counterclockwise order (we write $i'$ instead of $i+1$ because there may be many boundary components), and \[\omega_{\textrm{cyclic}_k}^{P}:=\sum_i\frac 12\;\diff f_i^k\wedge\diff f_{i'}^k,\] where $f_i^k(x):=e_k(x_i,x)$.

Is it true that $\Area(M,F)\geq\sum_{k\geq 0}C_k$?

The same inequality may hold in a square-celled surface. In this case there is no need for a limit in $P$ as we can take as $P$ the full set of boundary vertices. The numbers $C_k$ should vanish for sufficiently high $k$. Maybe even the equality $\Area(M,F)=\sum_{k\geq 0}C_k$ holds when the surface is tight (the walls on the universal cover form an infinite pseudoline arrangement) and all walls go from boundary to boundary.

\newpage
\section{Isometric fillings homeomorphic to the Möbius band and systolic inequality for the Klein bottle}\label{sec:mobius_klein}

In this section we will prove a new case of the FAC: when the filling is homeomorphic to the Möbius band and the metric is Riemannian or self-reverse Finsler. Previously, the FAC had only been stated for orientable fillings.

\begin{theorem}[FAC for Möbius bands with self-reverse Finsler metric]\label{thm:fac_mobius_selfrev} If a Möbius band $M$ with self-reverse Finsler metric fills isometrically its boundary of length $2L$, then $\Area_{\uHT}(M)\geq 2L^2$.
\end{theorem}

There is a \emph{flat} Finsler Möbius band that fills isometrically its boundary and has the same area as the hemisphere. It is obtained by cutting a square $[0,L]\times[0,L]$ of the $\ell_\infty$ plane and gluing the left and right sides by the identification $(0,y)\sim(L,L-y)$.

If one closes this surface by gluing $(x,0)\sim(x,L)$, then one obtains a Klein bottle that was conjectured systolically optimal by \cite{sabourau2016optimal}. This can be proved with the same analysis employed for the proof of Thm.~\ref{thm:fac_mobius_selfrev}.

\begin{theorem}[optimal systolic inequality for self-reverse Finsler metrics on the Klein bottle]\label{thm:systolic_klein_selfrev} If $M$ is a Klein bottle with a self-reverse Finsler metric, and every non-contractible curve in $M$ has length $\geq L$, then $\Area_{\uHT}(M)\geq 2L^2$.
\end{theorem}

By Theorems~\ref{thm:walledFAC_implies_FinslerFAC} and~\ref{thm:discretize_selfreverse} the two last theorems are corollaries of their discrete versions:

\begin{theorem}[discrete FAC holds for Möbius bands]\label{thm:fac_mobius_walls} Let $W$ be an even wallsystem on the Möbius band $M'$, such that $(M',W)$ fills isometrically its boundary of length $2n$, where $n\in\NN$ is even. Then $W$ has at least $\frac{n(n-1)}2$ self-crossings.
\end{theorem}

\begin{theorem}\label{thm:systolic_klein_walls} Let $W$ be an even wallsystem on the Klein bottle $M$, such that every non-contractible closed curve $\gamma$ in $M$ has $\Len_W(\gamma)\geq n$ where $n\in\NN$ is even. Then $W$ has at least $\frac{n^2}2$ self-crossings.
\end{theorem}

To prove Theorem~\ref{thm:fac_mobius_walls} we will close the surface to obtain a Klein bottle with a wallsystem, reducing the problem to a sort of systolic inequality.

\begin{lemma}[Closing isometric fillings of the circle] Let $(M',W')$ be a filling of the circle, with boundary of length $2n$, such that $W'$ forms a wallsystem $W$ on the closed surface $M$ obtained from $M'$ by identifying each boundary point with its antipodal. Then 
\begin{itemize}
\item The walled surface $(M',W')$ is an isometric filling of its boundary if and only if every generic closed curve $\gamma$ on $M$ that crosses $\partial M'$ an odd number of times has $\Len_{W}(\gamma)\geq n$. 
\item In fact, if the filling is not isometric, there is a \emph{simple} closed curve $\gamma$ in $M$ with $\Len_W(\gamma)<n$ that crosses $\partial M'$ exactly once.
\end{itemize}
\end{lemma}

(The curve $\gamma$ in $M$ is called ``generic'' if it is immersed, self-transverse, transverse to $W$ and to $\partial M'$, and avoids the self-crossing points of $W$ and the crossings points of $W$ with $\partial M'$.)

\begin{proof} If the filling $(M',W')$ is not isometric, then there is a generic curve $\gamma'$ that joins antipodal points of the boundary and has $\Len_{W'}(\gamma')<n$. Moreover, we can ensure that $\gamma'$ is simple: if $\gamma'$ visits the same point twice, at two instants $t_0<t_1$, then we can shorten and simplify $\gamma'$ by skipping the piece $\gamma'|_{[t_0,t_1]}$. This creates a turning point, that can be immediately smoothed away. Repeating this simplification as necessary, we eventually obtain a simple curve $\gamma'$. This curve $\gamma'$ forms a simple closed curve $\gamma$ on $M$ with length $\Len_W(\gamma)=\Len_{W'}(\gamma')<n$ that crosses $\partial M'$ exactly once. 

If the filling $(M',W')$ is isometric, then we have to prove that every generic closed curve $\gamma$ on $M$ that crosses $\partial M'$ an odd number of times  has $\Len_W(\gamma)\geq n$. Let $\gamma$ be such a curve. If $\gamma$ crosses $\partial M'$ only once, then it can be cut at that point and become a curve $\gamma'$ in $M'$ that joins opposite points of the boundary, so $n\leq \Len_{W'}(\gamma')=\Len_W(\gamma)$, as we had to prove. If $\gamma$ crosses $\partial M'$ more than once, then we can eliminate a pair of consecutive crossings without increasing the length of $\gamma$, as follows. Let $x_0=\gamma(t_0)$ and $x_1=\gamma(t_1)$ be two consecutive crossings of $\gamma$ with the boundary $\partial M'$, and let $d$ be their distance along the boundary $(\partial M',\partial W')$. The piece $\gamma|_{[t_0,t_1]}$ of $\gamma$ must have $\Len_{W}(\gamma|_{[t_0,t_1]})\geq d$ because it is contained on $M'$, which is an isometric filling of its boundary. We can modify the curve $\gamma$ in an interval slighthly larger than $[t_0,t_1]$, replacing this portion by a new curve that remains near $\partial M'$ but avoids crossing it, 
and has length $d$. In this way we reduce the number of crossings of $\gamma$ with $\partial M'$ in two units, without increasing the length of $\gamma$. Repeating this as necessary we fall in the case of one crossing, that has been dealt with before. This finishes the proof that $\Len_W(\gamma)\geq n$. 
\end{proof}

Applying this lemma, we see that the discrete FAC for Möbius bands (Theorem~\ref{thm:fac_mobius_walls}) follows from the following theorem about wallsystems on the Klein bottle:

\begin{theorem}[discrete FAC for Möbius bands, closed version]\label{thm:fac_klein_walls} Let $M$ be a Klein bottle obtained by gluing antipodal boundary points of a Möbius band $M'$. Let $W$ be an even wallsystem on $M$ that crosses at least $n$ (an even number) times every simple generic closed curve $\gamma$ that crosses $\partial M'$ an odd number of times. Then $W$ has at least $\frac{n(n-1)}2$ self-crossings.
\end{theorem}

The two theorems~\ref{thm:systolic_klein_walls},~\ref{thm:fac_klein_walls} that remain to be proven in the rest of this section are lower bounds for the number of self-crossings of a wallsystem $W$ on the Klein bottle, assuming that certain simple closed curves $\gamma$ have $\Len_W[\gamma]\geq n$. As we will see, there are only four homotopy classes that contain simple curves, two of which have an odd crossing number with $\partial M'$. Before studying the curves on the Klein bottle, we discuss a method to simplify wallsystems for general compact surfaces.

\subsection{Tightening wallystems on surfaces} 

In this subsection we consider the problem of transforming a given wallsystem $W$ on a compact surface $M$ so as to reduce its number of self-crossings without modifying its minlength function. This problem was solved completely by Schrijver and de Graaf \cite{schrijver1991decomposition,graaf1997making,graaf1997decomposition,graaf1995characterizing,graaf1994graphs}, so we review their work.

If $\alpha,\beta$ are smooth generic curves on a surface $M$, we define
\begin{align*}
\ncr(\alpha)
&=\ \ \:\#\left\{\left\{t,t'\right\}\text{ such that }\alpha(t)=\alpha\left(t'\right)\text{ but }t\neq t'\right\}\\
&=\frac 12\,\#\left\{\left(t,t'\right)\text{ such that }\alpha(t)=\alpha\left(t'\right)\text{ but }t\neq t'\right\},\\
\ncr(\alpha,\beta)&=\#\left\{\left(t,t'\right)\text{ such that }\alpha(t)=\beta\left(t'\right)\right\},\\
\mincr(\alpha)&=\min_{\alpha'\sim\alpha}\ncr\left(\alpha'\right),\text{ and}\\
\mincr(\alpha,\beta)&=\min_{\substack{\alpha'\sim\alpha\\\beta'\sim\beta}}\ncr\left(\alpha',\beta'\right).
\end{align*} A wallsystem $W$ on a surface is \term{minimally crossing} or in \term{minimally crossing position} if 
\begin{itemize}
\item no wall $w\in W$ is contractible and
\item each wall $w\in W$ has $\ncr(w)=\mincr(w)$ and
\item each pair of walls $w,w'\in W$ has $\ncr(w,w')=\mincr(w,w')$.
\end{itemize}
It turns out that these conditions can be attained simultaneously by deleting the contractible walls and putting the remaining walls in certain position that is nearly geodesic with respect to an auxiliary Riemannian metric that has constant curvature and geodesic boundary. In fact, a stronger statement is true.

\begin{lemma}[{\cite[Thm. 1]{graaf1997making}}]\label{thm:graaf_making} On a compact surface $M$, any wallsystem $W$ can be put in minimally crossing position by isotopies and $\Sho_i$ operations.
\end{lemma}

The transformation of the wallsystem from its initial position to the minimally crossing position can be done (conjecturally) by letting the walls evolve by curvature flow, preceded by a slight perturbation that puts the walls in generic position so that the only ``topological events'' that take place are the operations $\Sho_i$. This process eliminates the contractible curves and makes the remaining curves nearly geodesic. The analysis of curvature flow is technically demanding, but de Graaf--Schrijver \cite{graaf1997making} (and also Hass--Scott \cite{hass1994shortening}, in the case of orientable surfaces) developed a more elementary curve-shortening process that also makes the wallsystem nearly geodesic, called \term{disk flow}. Disk flow is done in steps, and each step consist of choosing a small disk $D\subseteq M$ and straightening the part of $W$ that lays in $D$. (This straightening can be done using the combinatorial methods of Steinitz and Ringel that we described in Lemmas~\ref{thm:Steinitz} and~\ref{thm:Ringel}.)

We now apply the lemma, following \cite{graaf1997making}.

The reduction $\Sho_2$ may reduce the minlength function of $W$, but the $\Ste_2$ does not, so we can stop the curve-shortening process each time a $\Sho_2$ operation is about to take place, and perform instead the $\Ste_2$ operation. Then we restart the process. Since the number of crossings is reduced each time this happens, the process must end at some point, and then the walls are in minimally crossing position. This shows that the wallsystem can be put in minimally crossing position by means of $\Ste_i$ operations, which do not reduce the minlength function. Furthermore, it is possible to replace the $\Ste_i$ operations (with $i\neq 0$) by uncrossings. Indeed, the reductions $\Ste_2$ and $\Ste_1$ are easy to replace by an uncrossing. And if two wallsystems differ by a $\Ste_3$ operation, and one of them admits an uncrossing that does not change the minlength function, then the other one also does. (To see this one analyzes three cases, according to whether the crossing to be uncrossed is one of the three crossings involved in the $\Ste_3$ operation, and if so, in which of the two ways it is uncrossed.) 
Therefore, any wallsystem that is not minimally crossing can be uncrossed without reducing its minlength function.

\begin{lemma}\label{thm:tightening} Any wallsystem $W$ on a compact surface $M$ can be turned into a minimally crossing wallsystem $W'$ that has the same minlength function as $W$ using 
\begin{itemize}
\item $\Ste_i$ operations or
\item uncrossings and $\Ste_0$ operations.
\end{itemize}
\end{lemma}

Note that $W'$ is not homotopic to $W$ in general. Also note that the $\Ste_0$ operations can be avoided if all contractible connected components of $W$ (more precisely, of the image of $W$ in $M$) are deleted at the beginning.

A wallsystem is called \term{tight} if any uncrossing reduces the minlength function. To prove Theorems~\ref{thm:systolic_klein_walls} and~\ref{thm:fac_klein_walls} about wallsystems on the Klein bottle we may assume that the wallsystems are tight. Lemma~\ref{thm:tightening} says that a tight wallsystem must be minimally crossing. For a minimally crossing wallsystem, we can describe its minlength function explicitly, as follows.

\begin{lemma} If a wallsystem $W$ on a compact surface $M$ is in minimally crossing position, then every relatively closed curve $\gamma$ has \begin{equation}\label{eq:minlen_mincr_walls}\minlen_W(\gamma)=\sum_i\mincr(\gamma,w_i),\end{equation} where $w_i$ are the walls of $W$.
\end{lemma}

\begin{proof} A curve homotopic to $\gamma$ cannot cross $W$ strictly less than $\sum_i\mincr(\gamma,w_i)$ times. To show that this number of crossings can be attained by a curve $\gamma''\simeq\gamma$, we first move the system of curves $\gamma\cup W$ to minimally crossing position using $\Sho_i$ operations, according to Lemma~\ref{thm:graaf_making}. We obtain a system $\gamma'\cup W'$ where $\gamma'\simeq\gamma$ and $W'$ is made of walls $w_i'\simeq w_i$ such that $\ncr(\gamma',w_i')=\mincr(\gamma,w_i)$ for each $i$. During the reduction process, the only modifications of $W$ are $\Sho_3$ operations, since the other operations $\Sho_i$ with $i<3$ would imply that $W$ was not in minimally crossing position at the beginning. This implies that $W'$ has the same minlength function as $W$, so there is a curve $\gamma''\simeq\gamma'$ such that \[\Len_{W}(\gamma'')=\Len_{W'}(\gamma')=\sum_i\ncr(\gamma',w_i')=\sum_i\mincr(\gamma,w_i),\] as we had to show.
\end{proof}





\begin{remark} The problem of tightening a surface may belong to a ``higher geometric group(oid) theory''. Let $M$ be a closed surface. A square-celled decomposition of $M$ is a presentation of the groupoid $\Gamma$ of homotopy classes $[\gamma]$ of paths $\gamma$ (that have endpoints on the vertices). The edges $e$ are generators and the 2-cells are relators. The set of conjugacy classes $[[\gamma]]$ of \emph{closed} curves is in natural bijection with the set $Z$ of homotopy classes of closed curves in $M$. This set does not depend on the decomposition.

A basic problem addressed by geometric group theory is the following: given a \emph{conjugacy class} in $\Gamma$, presented as a string of edges $(e_i)_{0\leq i<l}$ such that $\gamma=\prod_ie_i$ is a closed curve, rewrite it using the \emph{minimum number $l$ of generators} without changing the conjugacy class $[[\prod_ie_i]]\in Z$. The minimum length $l$ of such an expression is what we call the minlength of $\gamma$; it depends only on $[[\gamma]]\in Z$. In the problem of tightening, we are presented with the \emph{groupoid} $\Gamma$ itself (as a square-celled decomposition of $M$), and we have to rewrite it (find a new decomposition) that has the \emph{minimum number of relators} (square cells), while keeping the minlength function on $Z$ fixed.
\end{remark}

\begin{remark} Note that geometric group theory is usually concerned with groups rather than groupoids. To obtain a group of paths from a square-celled surface we can choose a vertex $x_0$ and consider only the homotopy classes $[\gamma]$ of paths $\gamma$ that start and end at $x_0$. But this choice distorts the geometry. For example, a closed curve, represented by a conjugacy class in $\Gamma$, in general cannot visit the point $x_0$ without increasing its length. This is not a problem if we are content with coarse geometry (quasi-isometric invariants of the group), which is sufficient for studying the topological features of the manifold.

One way to recover the exactness lost by fixing the basepoint $x_0$ is to consider ``stable length'' $\lim_{k\to\infty}\frac 1k\,\minlen[\gamma^k]$, which equals the minlength of $[[\gamma]]$ (the homotopy class of $\gamma$ without fixed basepoint) if the surface is orientable. Additionally, to get a discrete notion of ``length'' of a closed curve one uses topology: the length of the curve is defined as the minimum genus of a filling surface. After this kind of precautions are taken, one can work with discrete methods and study exact (rather than coarse) questions, like rationality of stable commutator length \cite{calegari2009stable}.

Methods such as de Graaf-Schrijver tightening and the applications presented here (see also \cite{cossarini2016intersection}) suggest that one can obtain exact relations without passing to the stable limit, by discretizing the whole geometry (rather than just the topology) of the manifold.
\end{remark}

\subsection{Closed curves on the Klein bottle} 

Consider the Klein bottle as the quotient $M=\RR^2/\Gamma$ of its universal cover $\RR^2$ by the group of transformations $\Gamma$ generated by the unit vertical translation $U:(x,y)\mapsto(x,1+y)$ and the unit glide-reflection $R:(x,y)\mapsto(x+1,-y)$ along the line $y=0$. Let $\pi:\widetilde M\to M$ be the covering map. We note that $\Gamma$ also contains the unit glide-reflection $S=UR:(x,y)\mapsto(x+1,1-y)$ along the line $y=\frac 12$, and the horizontal translation $T=R^2=S^2:(x,y)\mapsto(x+2,y)$. We give names to the four transformations $R,S,T,U$, shown in Figure~\ref{fig:klein_curves}: 

\[\begin{array}{rrcccr}
R&:&(x,y)&\mapsto&(x+1,\phantom{1}-y)&\text{boRder}\\
S&:&(x,y)&\mapsto&(x+1,1-y)&\text{Soul}\\
T&:&(x,y)&\mapsto&(x+2,\phantom{2+}\ y)&\text{Trans}\\
U&:&(x,y)&\mapsto&(x\phantom{+2}\ ,1+y)&\text{Up}.\\
\end{array}\]

\begin{figure}
  \centering
  \includegraphics[width=.4\linewidth]{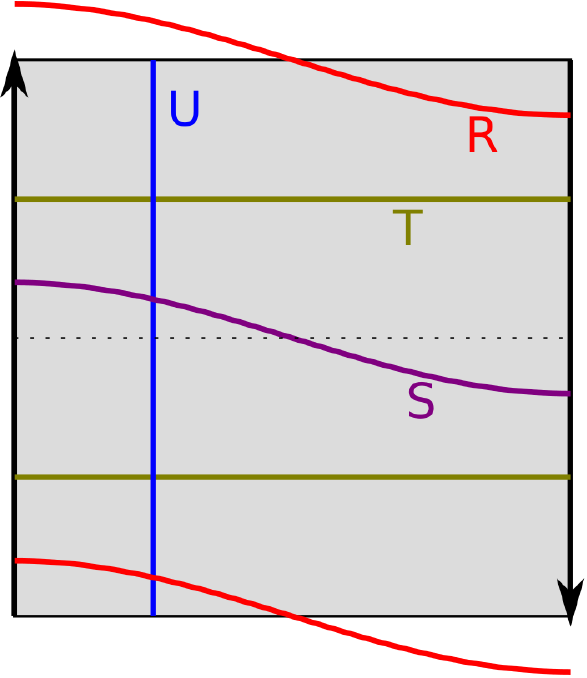}
  \caption{Simple closed curves on the Klein bottle.\label{fig:klein_curves}}
\end{figure}

The homotopy classes of closed curves in $M$ are in natural correspondence with the conjugacy classes of transformations $g\in\Gamma$ (as happens whenever any space $M$ is expressed as quotient of its universal cover by the action of a transformation group $\Gamma$). Indeed, each conjugacy class $[g]$ of a transformation $g\in\Gamma$ can be mapped to the homotopy class $[\gamma_0]$, where $\gamma_0$ is a closed curve in $M$ obtained by projecting via $\pi$ a curve $\widetilde\gamma_0:[0,1]\to\widetilde M$ that goes from any point $\widetilde{x_0}\in\widetilde M$ to the point $g\,\widetilde{x_0}\in\widetilde M$. The homotopy class $[\gamma_0]$ is independent of the chosen point $\widetilde{x_0}$ and curve $\widetilde{\gamma_0}$,\footnote{Proof: if $\widetilde{x_1}$ and $\widetilde{\gamma_1})$ are other choices of point and curve, and $\widetilde\beta$ is a curve $I\to \widetilde M$ that goes from $\widetilde{x_0}$ to $\widetilde{x_1}$, then the curve $g\,\widetilde{\beta}$ goes from $g\,\widetilde{x_0}$ to $g\,\widetilde{x_1}$ and since $\widetilde M$ is simply connected, there exists a map $\widetilde H:[0,1]^2\to\widetilde M$ that satisfies \begin{equation}
\widetilde H(0,-)=\widetilde{\gamma_0},\quad
\widetilde H(1,-)=\widetilde{\gamma_1},\quad
\widetilde H(t,0)=\widetilde\beta(t),\quad
\widetilde H(t,1)=g\,\widetilde\beta(t).\label{eq:lifted_homotopy}\end{equation} The map $H=\pi\circ\widetilde H$ is a homotopy from $\gamma_0=H(0,-)$ to $\gamma_1=H(1,-)$, where each intermediate curve $\gamma_t=H(t,-)$ is closed at the point $\beta(t)=\pi\,\widetilde\beta(t)=\pi\,g\,\widetilde\beta(t)$. This proves that $[\gamma_0]=[\gamma_1]$.} and depends only on the conjugacy class of $g$ in $\Gamma$.\footnote{Proof: If $h\in\Gamma$, note that the curve $h\,\widetilde{\gamma_0}$, that goes from $h\,\widetilde{x_0}=:\widetilde{x_1}$ to $h\,g\,\widetilde{x_0}=h\,g\,h^{-1}\,\widetilde{x_1}$, has the same projection by $\pi$ than $\widetilde{\gamma_0}$ (because $\pi=\pi\,h$), which shows that $g$ and its conjugate $h\,g\,h^{-1}$ are mapped to the same homotopy class $[\gamma_0]$.} The correspondence $[g]\mapsto[\gamma_0]$ is bijective.\footnote{Indeed, one can construct the inverse function $[\gamma_0]\mapsto[g]$ using the same construction: $\widetilde{\gamma_0}$ is a lift of $\gamma_0$ to $\widetilde M$ that goes from some point $\widetilde x_0$ to some point $g\,\widetilde{x_0}$, and this defines $g$. The conjugacy class $[g]$ does not depend on the choice of the lift $\widetilde{\gamma_0}$ (because any other lift is of the form $h\,\widetilde{\gamma_0}$ and goes from $h\,\widetilde{x_0}=:\widetilde{x_1}$ to $h\,g\,\widetilde{x_0}=h\,g\,h^{-1}\widetilde{x_1}$, which yields the transformation $h\,g\,h^{-1}$, that is conjugate to $g$) and is invariant by homotopies of the closed curve $\gamma_0$. Indeed, let $H:[0,1]^2\to M$ be a homotopy of closed curves from $\gamma_0$ to some $\gamma_1$ that carries the starpoint/endpoint along a path $H(t,0)=H(t,1)=:\beta(t)$. Lift $H$ to obtain a map $\widetilde H:[0,1]^2\to\widetilde M$ such that $\widetilde H(0,-)=\widetilde{\gamma_0}$. The curve $\widetilde\beta(t)=\widetilde H(t,0)$ is a lift of $\beta$ that starts at $\widetilde{x_0}$, and another lift of $\beta$ is the curve $t\mapsto H(t,1)$. Since it starts at $g\,\widetilde{x_0}$, it must be the curve $g\,\widetilde\beta$. This implies that the curve $\widetilde{\gamma_1}:=\widetilde H(1,-)$ (which is a lift of $\gamma_1$) goes from some point $\beta(1)=:\widetilde{x_1}$ to the point $g\,\beta(1)=g\,\widetilde{x_1}$, which shows that $[\gamma_1]$ is mapped to the same conjugacy class $[g]$ as $[\gamma_0]$, so the function $[\gamma_0]\mapsto[g]$ is well defined.}

We may then identify closed mobile curves with conjugacy classes in $\Gamma$. We classify them in the following way:

\begin{lemma}[Classification of closed curves on the Klein bottle] Every transformation $g\in\Gamma$ is conjugate to exactly one of the following transformations, or its inverse:
\begin{itemize}
\item A translation $T^tU^u:(x,y)\mapsto(x+2t,y+u)$, with $(t,u)\in\NN^2$.
\item A glide-reflection $R^{2k+1}$ or $S^{2k+1}$, where $k\in\NN$.
\end{itemize}
\end{lemma}

\begin{proof} Each $g\in\Gamma=\langle U,R\rangle$ is an affine transformation of $\RR^2$ that preserves the horizontal direction and either preserves or inverts the vertical direction, depending on whether it was generated using an even or odd number of times the transformations $R$ and $R^{-1}$. In the first case, $g$ is a translation whose horizontal component is even, so it is of the form $T^tU^u$. In the second case, $g$ is a glide-reflection $\rho_{c,d}:(x,y)\mapsto (x+c,d-y)$ of odd length $c$ along a horizontal line $y=\frac d2$, and it is conjugate to either $R^c$ or $S^c$, depending on whether $d$ is even or odd, because $T^nR^cT^{-n}=\rho_{c,2n}$ and $T^nS^cT^{-n}=\rho_{c,2n+1}$.
\end{proof}

\subsection{Crossing numbers of curves on the Klein bottle}
In this subsection we compute the numbers $\mincr(\gamma)$, $\mincr(\gamma,\gamma')$ for curves in the Klein bottle and in some simpler surfaces: the orientable band, the Möbius band, and the torus.

\begin{lemma}[Crossing numbers of curves on an orientable band]\label{thm:mincr_orientable_band} On an orientable band $A=\RR^2/T$, any two curves $\gamma\simeq T^k$ and $\gamma'\simeq T^{k'}$ (with $k,k'\in\NN$) have $\mincr\left(\gamma,\gamma'\right)=0$ and $\mincr(\gamma)=k-1$.
\end{lemma}

\begin{proof} It is easy to find curves $\gamma\simeq T^k$ and $\gamma'\simeq T^{k'}$ that do not cross each other, and such that $\ncr(\gamma)=k-1$, so we only need to prove that $\gamma$ cannot have strictly less than $k-1$ self-crossings. This follows from the next lemma.
\end{proof} 

\begin{lemma}[Splitting off orientation-preserving curves, \cite{schrijver1991decomposition}]\label{thm:splitoff_orient-pres} If $\gamma$ is an orientation-preserving noncontractible curve on a surface $M$, then every curve homotopic to $\gamma^k$ (with $k\geq 2$) can be split, by an uncrossing, into a curve homotopic to $\gamma$ and a curve homotopic to $\gamma^{k-1}$.
\end{lemma}

(It follows that any curve homotopic to $\gamma^k$ can be split, by $k-1$ uncrossings, into $k$ walls homotopic to $\gamma$.)

\begin{proof} See \cite[Prop.~4]{schrijver1991decomposition}.
\end{proof}

\begin{lemma}[Crossing numbers of curves on the Möbius band, \cite{graaf1997making}]\label{thm:mincr_mobius} On the Möbius band $B=\RR^2/R$, any two closed curves $\gamma\simeq R^{2k+1}$ and $\gamma'\simeq R^{2k'+1}$ (with $k,k'\in\NN$) have $\mincr(\gamma,\gamma')=\min\{2k+1,2k'+1\}$ and $\mincr(\gamma)=k$.
\end{lemma}

\begin{proof} To show that these numbers of crossings can be attained, for $k\in\NN$ we define a \term{standard curve of class $R^{2k+1}$} as follows. It is a polygonal $\gamma_{R^{2k+1}}$ made of $2k+1$ straight segments, so it is enough to describe $2k+1$ consecutive segments of its lift $\widetilde\gamma_{R^{2k+1}}$. We do so by giving the coordinates of the first $2k+2$ vertices, which are \[\begin{array}{rcrrrrrrrr}x&=&0,&1,&2,&3,&\dots,&2k-1,&2k,&2k+1\\y&=&1,&-1,&2,&-2,&\dots,&k,&-k,&-1.\end{array}\] Note that the first $2k$ segments of $\gamma_{R^{2k+1}}$ do not cross each other, and the last segment crosses $k$ of the previous segments, so the total number of self-crossings of this curve is $k$. For $\varepsilon>0$, we also define a curve $\gamma_{R^{2k+1}}^{\varepsilon}$ in the same way as $\gamma_{R^{2k+1}}$, but with the $y$ coordinates multiplied by $\varepsilon$. Note that \[\ncr\left(\gamma_{R^{2k+1}}^{\varepsilon}\,,\gamma_{R^{2k'+1}}^{\varepsilon'}\right)=2k'+1\quad\text{ if }\varepsilon'<\frac 1{k'}\,\varepsilon.\]
(Comment: Another curve $\gamma\simeq R^{2k+1}$ that attains $k$ self-crossings is $\widetilde\gamma(t)=((2k+1)t,\varepsilon\cos(c+\pi\,t))$, which seems to be the limit shape of a generic curve evolving by curvature flow.)

We now show that every curve $\gamma\simeq R^{2k+1}$ has at least $k$ self-crossings. The preimage $\pi^{-1}(\Image\gamma)$ is made of $2k+1$ curves $\widetilde\gamma_i$ with $0\leq i<2k+1$. Each of these curves satisfies $\widetilde\gamma_i(t+1)=R^{2k+1}\left(\widetilde\gamma_i(t)\right)$. Among the points of $\pi^{-1}(\Image\gamma)$, the $y$-coordinate attains its highest value $\overline y$ at some point $p=\gamma_0\left(\overline t\right)$, and the lowest value $-\overline y$ at the point $q=\gamma_0\left(\overline t+1\right)$. Note that each of the curves $\gamma_i$ (with $i\neq 0$) runs along the band $-\overline y\leq y\leq\overline y$, so it must cross the piece of $\widetilde\gamma_0$ that goes from $p$ to $q$, parametrized by the interval $\left[\overline t,\overline t+1\right]$. Therefore, \begin{align*}2\ncr(\gamma)
&=\#\left\{t\in[0,1]:\gamma(t)\text{ equals some }\gamma(t')\text{ with }t'\neq t\right\}\\
&=\#\,\left\{t\in\left[\overline t,\overline t+1\right]:\widetilde\gamma_0(t)\text{ equals some }\widetilde\gamma_i(t')\text{ with }i\neq 0\text{ or }t\neq t'\right\}\\
&\geq\#\,\left\{i\neq 0:\widetilde\gamma_i\text{ crosses }\widetilde\gamma_0|_{\left[\overline t,\overline t+1\right]}\right\}\\
&\geq 2k,
\end{align*}
which shows that $\ncr(\gamma)\geq k$.

With a similar method we show that any two curves $\gamma\simeq R^{2k+1}$ and $\gamma'\simeq R^{2k'+1}$ cross at least $\min\{2k+1,2k'+1\}$ times. Assume that among the points of $\pi^{-1}(\Image\gamma)\cup\pi^{-1}\left(\Image\gamma'\right)$, the $y$-coordinate attains its maximum value $\overline y$ at some point $p\in\pi^{-1}(\Image\gamma)$. Then $p=\widetilde\gamma\left(\overline t\right)$ for some lift $\overline\gamma:\RR\to\RR^2$ of $\gamma$, that satisfies $\widetilde\gamma(t+1)=R^{2k+1}$. This implies that the minimum value of the $y$-coordinate is $-\overline y$, attained at $q=\widetilde\gamma\left(\overline t+1\right)$. On the other hand, the set $\pi^{-1}\left(\Image\gamma'\right)$ is made of $2k'+1$ curves $\widetilde\gamma_i'$ that are properly embedded in the band $-\overline y\leq y\leq\overline y$. Each of these curves is crossed by the piece $\overline\gamma|_{\left[\overline t,\overline t+1\right]}$, that goes from $p$ to $q$. Therefore,
\[\ncr\left(\gamma,\gamma'\right)=\ncr\left(\widetilde\gamma|_{\left[\overline t,\overline t+1\right]},\pi^{-1}\left(\Image\gamma'\right)\right)\geq 2k'+1.\]
If the $y$-coordinate attains its maximum at some point of $\pi^{-1}\left(\Image\gamma'\right)$, rather than a point of $\pi^{-1}(\Image\gamma)$, then the same argument shows that $\gamma$ and $\gamma'$ cross at least $2k+1$ times. These two cases cover all the possibilities, so $\ncr\left(\gamma,\gamma'\right)\geq\min\{2k+1,2k'+1\}$.
\end{proof}

\begin{lemma}[Splitting off orientation-reversing curves]\label{thm:splitoff_orient-rev} If $\gamma$ is an orientation-reversing curve on a surface $M$ with $\chi(M)\leq 0$, then any curve homotopic to $\gamma^{2k+1}$ (with $k\geq 1$) can be split, by an uncrossing, into a curve homotopic to $\gamma^2$ and a curve homotopic to $\gamma^{2k-1}$.
\end{lemma}
(In consequence, any curve homotopic to $\gamma^{2k+1}$ can be split, by $k$ uncrossings, into $k$ curves homotopic to $\gamma^2$ and one curve homotopic to $\gamma$.)

\begin{proof} Let $\widetilde M$ be the universal cover, so that $M=\widetilde M/\Gamma$ for some group $\Gamma$ of transformations of $\widetilde M$. Let $R$ be the transformation associated to $\gamma$. This implies that $\gamma$ lifts to a closed curve in the surface $B=\widetilde M/R$, which is a Möbius band. We perform uncrossings on $\gamma$ preserving the orientation of $\gamma$ until no crossings remain. In the end we get $n+1$ noncontractible curves $\gamma_i$ (with $0\leq i<n$) of class $R^{e_i}$, with $\sum_ie_i=2k+1$, and possibly some contractible curves. More precisely, the noncontractible curves are numbered from outermost to innermost; the first $n$ curves $\gamma_i$ with $0\leq i<n$ have $e_i=\pm 2$, and the a central curve $\gamma_n$ has $e_0=\pm 1$. This is the only possibility because any curve of class nontrivial multiple of $R^2$ has self-crossings, any curve of class nontrivial odd multiple of $R^1$ has self-crossings, and any two curves of class $R^1$ have crossings, according to Lemma~\ref{thm:mincr_mobius}. There must be some $m$ outermost curves $\gamma_i$ (with $0\leq i<m$) such that $\sum_{0\leq i<m}e_i=2$. By undoing some of the uncrossings that we did, we can rejoin these $m$ outermost curves (and possibly some contractible curves) into a single curve of class $R^2$, and a remaining curve of class $R^{2k-1}$.
\end{proof}

\begin{lemma}[Crossing numbers of curves on the torus]\label{thm:mincr_torus} On the torus $\TT^2=\RR^2/\Gamma^+$ (where $\Gamma^+=\langle T,U\rangle$), two curves $\gamma\simeq T^tU^u$ and $\gamma'\simeq T^{t'}U^{u'}$ (with $(t,u)$, $(t',u')\in\ZZ^2$) have \[\mincr(\gamma,\gamma')=\left|\det\begin{pmatrix}t&t'\\u&u'\end{pmatrix}\right|=|tu'-ut'|,\] and $\mincr(\gamma)=k-1$ if $(t,u)$ is a sum of $k$ equal primitive vectors.
\end{lemma}
\begin{proof} To compute $\mincr(\gamma,\gamma')=|tu'-ut'|$, use disk flow \cite{graaf1997making} to show that the minimum number of crossings is attained when the curves are nearly geodesic and cross $|tu'-ut'|$ times. Or note that the signed number of crossings, which is a homological invariant of the (oriented) curves $\gamma$, $\gamma'$, is $tu'-ut'$, also attained when the curves are nearly geodesic and all crossings have the same sign.

Regarding self-crossings, it is easy to find a curve $\gamma$ that is nearly geodesic and has $\ncr(\gamma)=k-1$. To show that $\ncr(\gamma)\geq k-1$, we recall that, by Lemma~\ref{thm:splitoff_orient-pres}, $\gamma$ can be split into $k$ curves by $k-1$ uncrossings.
\end{proof}

\begin{lemma}\label{thm:mincr_klein} On the Klein bottle $\RR^2/\Gamma$, any two closed curves $\gamma$, $\gamma'$ have
\[\mincr(\gamma,\gamma')=\begin{cases}
2\,\max\{tu',ut'\}&\text{if }\gamma\simeq T^tU^u\text{ and }\gamma'\simeq T^{t'}U^{u'},\\
u(1+2k)&\text{if }\gamma\simeq T^tU^u\text{ and }\gamma'\simeq(R\text{ or }S)^{2k+1},\\
1+2\,\min\{k,k'\}&\text{if }\gamma\simeq R^{2k+1}\text{ and }\gamma'\simeq R^{2k'+1},\\
\phantom{1+2\min\{k,k'\}}&\text{or }\gamma\simeq S^{2k+1}\text{ and }\gamma'\simeq S^{2k'+1},\\
0&\text{if }\gamma\simeq R^{2k+1}\text{ and }\gamma'\simeq S^{2k'+1},
\end{cases}\]
and a single closed curve $\gamma$ has
\[\mincr(\gamma)=\begin{cases}
tu&\text{if }\gamma\simeq T^{t}U^{u}+(k-1)\text{ and }(u/k,v/k)\in\NN^2\text{ is primitive},\\
k&\text{if }\gamma\simeq (R\text{ or }S)^{2k+1}.\\
\end{cases}\]
\end{lemma}

\begin{proof} We begin with the orientation-preserving curves. To compute the minimum number of self-crossings of a curve $\gamma\simeq T^tU^u$ where $(t,u)\in\NN^2$ is primitive, we consider the torus cover $\TT^2=\RR^2/\Gamma^+\overset{\pi}\to M$, where $\Gamma^+=\langle T,U\rangle$. The preimage $\pi^{-1}(\Image\gamma)$ consists of an ``ascending'' curve $\gamma_+\simeq T^tU^u$ and a ``descending'' curve $\gamma_-=R\circ\gamma_+\simeq T^tU^{-u}$. (If $u=0$, the two curves are assigned the names $\gamma_+$, $\gamma_-$ arbitrarily.) Then \begin{align*}\ncr(\gamma)=\tfrac 12\ncr\left(\pi^{-1}(\Image\gamma)\right)
&\geq\tfrac 12\ncr\left(\gamma_+,\gamma-\right)\\
&\geq\tfrac 12|tu+tu|\qquad\text{ by Lemma \ref{thm:mincr_torus}}\\
&=tu.\end{align*}
and this number of crossings is attained if $\gamma$ is straight.

In a similar way we compute $\mincr(\gamma,\gamma')$ for two curves $\gamma\simeq T^tU^u$, $\gamma'\simeq T^{t'}U^{u'}$. The preimage $\pi^{-1}(\Image\gamma)$ consists of two curves $\gamma_+\simeq T^tU^u$, and $\gamma_-\simeq T^tU^{-u}$, and the preimage of $\gamma'$ is made of two curves $\gamma'_+\simeq T^{t'}U^{u'}$ and $\gamma'_-\simeq T^{t'}U^{-u'}$. Therefore
\begin{align*}
\ncr(\gamma,\gamma')
&=\ncr\left(\gamma_+,\pi^{-1}\left(\Image\gamma'\right)\right)\\
&=\ncr\left(\gamma_+,\gamma'_+\right)+\ncr\left(\gamma_+,\gamma'_-\right)\\
&\geq\mincr\left(\gamma_+,\gamma'_+\right)+\mincr\left(\gamma_+,\gamma'_-\right)\\
&=|tu'-ut'|+|tu'+ut'|\quad\text{ by Lemma \ref{thm:mincr_torus}}\\
&=2\max\{tu',t'u\}.
\end{align*}
This number of crossings is attained when $\gamma$ and $\gamma'$ are nearly straight, which proves that $\mincr(\gamma,\gamma')=2\max\{tu',t'u\}$.

To compute the number of self-crossings of a curve $\gamma\simeq T^tU^u$ where $(t,u)$ is not primitive, we split $\gamma$, by $k-1$ uncrossings, into $k$ curves of class $T^{t/k}U^{u/k}$, where $(t/k,u/k)\in\NN^2$ is primitive. After this splitting, each of the $\frac{k(k-1)}2$ pairs of curves crosses at least $2(t/k)(u/k)$ times, and each of the $k$ curves has at least $(t/k)(u/k)$ self-crossings, so there are in total at least $k^2(t/k)(u/k)=tu$ crossings. This implies that before the splitting there were at least $tu+k-1$ crossings. This number of crossings can be attained by a certain nearly geodesic curve $\gamma$, so $\mincr(\gamma)=tu+k-1$.

We now consider the orientation-reversing curves. A curve $\gamma\simeq R^{2k+1}$ has at least $k$ self-crossings, because it lifts to a closed curve of the same class in the Möbius band $\RR^2/R$, which self-crosses at least $k$ times according to Lemma~\ref{thm:mincr_mobius}. This number of crossings can be attained in the same way as in the Möbius band, using the standard curve $\gamma_{R^{2k+1}}^\varepsilon$, which proves that $\mincr(\gamma)=k$. Using the same cover we can prove that when $\gamma\simeq R^{2k+1}$ and $\gamma'\simeq R^{2k'+1}$, the formula $\mincr(\gamma,\gamma')=\min\left\{2k+1,2k'+1\right\}$ holds as in the Möbius band. When $\gamma\simeq S^{2k+1}$ and $\gamma'\simeq S^{2k'+1}$, we can use the Möbius band $\RR^2/S$ to show that $\mincr(\gamma)$ have the same values $\mincr(\gamma,\gamma')$ as in the other case. If $\gamma\simeq R^{2k+1}$ and $\gamma'\simeq S^{2k'+1}$, it is easy to see that $\mincr(\gamma,\gamma')=0$.

Finally, to compute $\mincr(\gamma,\gamma')$ when $\gamma\simeq T^tU^u$ and $\gamma'\simeq R^{2k'+1}$ we use again the torus cover $\TT^2\overset{\pi}\to M$. Note that $\pi^{-1}(\Image\gamma)$ consist of two curves $\gamma_+\simeq T^tU^u$ and $\gamma_-\simeq T^tU^{-u}$, and $\pi^{-1}(\Image\gamma)$ consists of single curve $\widetilde\gamma'\simeq\left(R^{2k+1}\right)^2=T^{2k+1}$. Then \begin{align*}\ncr(\gamma,\gamma')
&=\frac 12\ncr\left(\pi^{-1}(\Image\gamma),\pi^{-1}\left(\Image\gamma'\right)\right)\\
&=\frac 12\ncr\left(\gamma_+,\widetilde\gamma'\right)+\left(\gamma_-,\widetilde\gamma'\right)\\
&\geq u(2k+1)\quad\text{ by Lemma \ref{thm:mincr_torus},}\end{align*}
and this number of crossings can be attained by nearly geodesic curves.
\end{proof}

From these formulas it follows that only the curves of the four named classes $R,S,T,U$ can be simple, apart from the trivial contractible curve.

A wallsystem $W$ on a surface $M$ is called \term{simple-tight} unless it is possible to perform an uncrossing on $W$ that does not reduce the minlength $\minlen_{W}(\gamma)$ of any \emph{simple} curve $\gamma$.

\begin{lemma}\label{thm:simple-tight_klein} On the Klein bottle, every simple-tight wallsystem $W$ is minimally crossing and consists of $r\geq 0$ walls of class $R$, $s\geq 0$ walls of class $S$, and $m\geq 0$ walls of class $T^{t'}U^{u'}$ (where $(t',u')\in\NN^2$ a primitive vector). Therefore, the wallystem depends on four parameters $r,s,t=mt',u=mu'\in\NN$. Its area is \[\Area(W)=\frac{r(r-1)}2+\frac{s(s-1)}2+tu+su+ru.\] The minlengths of the simple curves are \begin{equation}\label{eq:minlen_simple-tight_klein}\begin{array}{lcl}
\minlen_W(R)=r+u,&\ &\minlen_W(T)=2u\\
\minlen_W(S)=s+u,&\ &\minlen_W(U)=2t+r+s.\end{array}\end{equation} The wallsystem is even if and only if the numbers $r,s,u$ have the same parity.
\end{lemma}

\begin{proof} Let $W$ be a simple-tight wallsystem. By lemma~\ref{thm:tightening} it must be minimally crossing. It consists of walls
\begin{IEEEeqnarray*}{lCll}
w_i&\simeq&U^{u_i}T^{t_i}&\text{ for }0\leq i<m,\\
w^R_j&\simeq&R^{2a_j+1}&\text{ for }0\leq j<r,\\
w^S_k&\simeq&S^{2b_k+1}&\text{ for }0\leq k<s,
\end{IEEEeqnarray*}
with $u_i,t_i,a_j,b_j\in\NN$.
The minlength of any closed curve $\gamma$ is the sum of its min-crossings with the walls, according to~\eqref{eq:minlen_mincr_walls}. In particular, according to the formulas for $\mincr(\gamma,w)$ on the Klein bottle (see Lemma~\ref{thm:mincr_klein}), the simple curves $R,S,T,U$ have
\begin{IEEEeqnarray*}{lCl}
\minlen_W(R)&=&
\textstyle{\sum_i}\underbrace{\mincr(R,w_i)}_{=u_i}+
\textstyle{\sum_j}\underbrace{\mincr\left(R,w^R_j\right)}_{=1}+
\textstyle{\sum_k}\underbrace{\mincr\left(R,w^S_k\right)}_{=0}\\
&=&\textstyle{\sum_i}u_i+r\\
\minlen_W(S)&=&
\textstyle{\sum_i}\underbrace{\mincr(S,w_i)}_{=u_i}+
\textstyle{\sum_j}\underbrace{\mincr\left(S,w^R_j\right)}_{=0}+
\textstyle{\sum_k}\underbrace{\mincr\left(S,w^S_k\right)}_{=1}\\
&=&\textstyle{\sum_i}u_i+s\\
\minlen_W(T)&=&
\textstyle{\sum_i}\underbrace{\mincr(T,w_i)}_{=u_i}+
\textstyle{\sum_j}\underbrace{\mincr\left(T,w^R_j\right)}_{=0}+
\textstyle{\sum_k}\underbrace{\mincr\left(T,w^S_k\right)}_{=0}\\
&=&2\,\textstyle{\sum_i}u_i\\
\minlen_W(U)&=&
\textstyle{\sum_i}\underbrace{\mincr(U,w_i)}_{=2\,t_i}+
\textstyle{\sum_j}\underbrace{\mincr\left(U,w^R_j\right)}_{=2\,a_j+1}+
\textstyle{\sum_k}\underbrace{\mincr\left(U,w^S_k\right)}_{=2\,b_k+1}\\
&=&2\,\textstyle{\sum_i}t_i+\textstyle{\sum_j}(2\,a_j+1)+\textstyle{\sum_k}(2\,b_k+1).
\end{IEEEeqnarray*}

Note that the minlengths of the curves $R,S,T$ are already as stated in \eqref{eq:minlen_simple-tight_klein}, if we define $u:=\sum_iu_i$ and $t:=\sum_it_i$.

We note the following facts:
\begin{itemize}
\item Each orientation-preserving wall $w\simeq T^{t_i}U^{u_i}$ must have $(t_i,u_i)\in\NN^2$ primitive. If this is not true, then by lemma~\ref{thm:splitoff_orient-pres} the wall can be split into $k$ walls of class $T^{t_i/k}U^{u_i/k}$ (where $k\in\NN$ and $(t_i/k,u_i/k)\in\NN^2$ is primitive). This does not decrease the minlength function of $W$, according to the formulas in Lemma~\ref{thm:mincr_klein}. (In fact, on any minimally crossing wallsystem on any surface we can perform the splitting of Lemma~\ref{thm:splitoff_orient-pres} without affecting the minlength function. This is done in \cite{graaf1994graphs} and follows from the formulas in~\cite[Thm.~7]{graaf1997making}, in the case of surfaces with strictly negative Euler characteristic.)
\item Each orientation-reversing wall must be of primitive class. (This implies that all numbers $a_j$ and $b_k$ are zero, so the minlength of $T$ is as stated in \eqref{eq:minlen_simple-tight_klein}.) Indeed, by Lemma~\ref{thm:splitoff_orient-rev}, each wall $w\simeq R^{2a_j+1}$ can be split into $a_j$ walls of class $R^2=T$ and one wall of class $R$. This only reduces the minlength of the non-trivial odd multiples of $R$, which are not simple. (This is also true on general surfaces; see \cite[Thm.~7]{graaf1997making}.) In particular, note that minlength of $U$ is not changed because $a_j$ decreases to $0$ but $\sum_it_i$ increases the same number of units.
\item Finally, all orientation-preserving walls $w_i\simeq T^{t_i}U^{u_i}$ must be homotopic to each other. If this is not true, we can perform uncrossings as follows. Orient each wall $w_i$ so that $w_i\simeq T^{t_i}U^{u_i}$ as oriented curves (and remember that $t_i,u_i\geq 0$). Consider the 2-to-1 torus cover $\TT^2\overset{\pi}\to M$. The preimage $\pi^{-1}(w_i)$ of each wall $w_i$ consists of two walls, a red ascending wall $w_i^+\simeq T^{t_i}U^{u_i}$ and a blue descending wall $w_i^-=R\left(w_i^+\right)\simeq T^{t_i}U^{-u_i}$. (If $u_i=0$, we choose the colors arbitrarily.) If there are red walls with different homotopy classes, then they cross at some point $p\in\TT^2$, according to the min-crossing formulas on the torus (Lemma~\ref{thm:mincr_torus}). Then we split the crossing of $W$ at $\pi(p)\in M$ in the way that respects the orientation of the walls. This implies that on the wallsystem $\pi^{-1}(W)$ on the torus $\TT^2$, the crossing between red walls at $p$ and the crossing between blue walls at $R(p)$ are both split. After the uncrossing, the list of homotopy classes $T^{t_i}U^{u_i}$ of the red walls $w_i^+$ (and the list of homotopy classes $T^{t_i}U^{-u_i}$ of the blue walls $w_i^-=R(w_i^+)$) may be different from before, but the sum $\sum_i(t_i,u_i)=\left(\sum_it_i,\sum_iu_i\right)=:(t,u)\in\NN^2$ has not changed, because after the red curves are made of the same oriented pieces as before.

The fact that $t_i,u_i\geq 0$ also remains true. Indeed, suppose some $u_i$ is negative after the uncrossing, and consider the minlength of $T$. Before the uncrossing we had $\minlen_W(T)=2\,\sum_iu_i$ and after the uncrossing we have $\minlen_W(T)\geq2\,\sum_i\left|u_i\right|$, because a wall $w_i^+\simeq T^{t_i}U^{u_i}$ with $(t_i,u_i)\in\ZZ^2$ on the torus has $\mincr(T,w_i)=|u_i|$, and $w_i^-\simeq T^{t_i}U^{-u_i}$ has $\mincr(T,U_i)=|u_i|$. Since the number $\sum_iu_i$ has not changed, and the number $\minlen_W(T)$ cannot be increased during the uncrossing, it follows that we still have $u_i\geq 0$ (and $\minlen_W(T)$ is unchanged). In a similar way we see that $t_i\geq 0$, because $\minlen_W(U)=\sum_i t_i+r+s$ before the uncrossing and $\minlen_WU\geq\sum_i\left|t_i\right|+r+s$ after the uncrossing. The minlengths of $R,S,T,U$ do not change, because they only depend on $r,s,t,u$.
\end{itemize}

This finishes the proof that every simple-tight wallsystem $W$ on the Klein bottle is as stated, in terms of the four parameters $r,s,t,u\in\NN$, and the minlengths of $R,S,T,U$ are as given in \eqref{eq:minlen_simple-tight_klein}. The number of self-crossings of the wallsystem $W$ is $\frac{r(r-1)}2+\frac{s(s-1)}2+tu+su+ru$ because:
\begin{itemize}
\item Each pair of walls of class $R$ crosses once, and this gives the term $\frac{r(r-1)}2$, and the term $\frac{s(s-1)}2$ is obtained in a similar way.
\item Each of the $m$ walls of class $T^{t'}U^{u'}$ self-crosses $t'u'$ times, and each pair of walls of this class crosses $2\max\{t'u',t'u'\}=2t'u'$ times, so we get $mt'u'+\frac{m(m-1)}2 2t'u'=m^2t'u'=tu$ crossings.
\item Each of the $k$ walls of class $T^{t'}U^{u'}$ crosses each of the walls of class $R$ at $u'$ points, so we get $rmu'=ru$ crossings in this way, and in a similar way we obtain the term $su$.
\end{itemize}
Regarding the parity of the wallsystem $W$, note that for the wallsystem to be even, it is necessary that $r,s,u$ have the same parity, because $\minlen_W(R)=r+u$ is even if and only if $r$ and $u$ have the same parity and $\minlen_W(S)=s+u$ is even if and only if $s$ and $u$ have the same parity. Reciprocally, if $r,s,u$ have the same parity, then the other simple curves $T,U$ have even minlength $\minlen_W(U)=2t+r+s$ and $\minlen_W(T)=2u$. Any closed curve can be broken into simple curves, so the wallsystem is even if and only if the numbers $r,s,u$ have the same parity.
\end{proof}

We are now ready to prove the main theorems of this section.

\begin{proof}[Proof of Thm.~\ref{thm:fac_klein_walls}] Let $W$ be a wallsystem on the Klein bottle such that $\Len_W(R)\geq n$ and $\Len_W(U)\geq n$. We must prove that $\Area(W)\geq\frac{n(n-1)}2$. We may split crossings that do not reduce the minlength of $R,S,T,U$ until the wallsystem $W$ is simple-tight, so by Lemma~\ref{thm:simple-tight_klein}, it depends on four numbers $r,s,t,u\in\NN$, and we have $\Len_W(R)=r+u$ and $\Len_W(T)=2t+r+s$. We must minimize $\Area(W)=\frac{r(r-1)}2+\frac{s(s-1)}2+u(t+s+r)$ under the restrictions $r+u\geq n$ and $2t+r+s\geq n$, while keeping $r,s,u$ with the same parity so that $W$ is even. We may assume that $r\leq n$, otherwise the first term of $Q$ is already greater or equal than $\frac{n(n-1)}2$. Then we perform the following operations that reduce the area of $W$ without violating the constraints:
\begin{itemize}
\item Reduce $r,s,t,u$ as much as possible, until $r+u=2t+r+s=n$, keeping $r,s,u$ with the same parity.
\item Reduce $s$ in two units and increase $t$ in one unit, and repeat this operation until $s$ is either $0$ (if $r,s,u$ are even) or $1$ (if $r,s,u$ are odd).
\end{itemize}
If $s=0$, then we can write in terms of $r$ the other variables $u=n-r$ and $t=\frac{n-r}2$, so the area is \begin{align*}\Area(W)
&=\frac{r(r-1)}2+(n-r)\left(\frac{n-r}2+r\right)\\
&=\tfrac 12\big(r(r-1)+(n-r)(n+r)\big)=\tfrac 12\left(n^2-r\right)\geq\tfrac 12\left(n^2-n\right),\end{align*} as we wanted to prove. 
In the case $s=1$, we have $u=n-r$ and $t=\frac{n-r-1}2$, so
\begin{align*}\Area(W)
&=\frac{r(r-1)}2+(n-r)\left(\frac{n-r}2+r-\frac 12\right)\\
&=\frac 12\left(n^2-r\right)-\frac12 (n-r)=\frac 12\left(n^2-n\right).\qedhere\end{align*}
\end{proof}

This finishes the proof of Thm.~\ref{thm:fac_mobius_selfrev} that every Möbius band $M$ with self-reverse Finsler metric $F$ that fills isometrically its boundary of length $2L$ has $\Area_\uHT(M,F)\geq 2L^2$.

\begin{proof}[Proof of Thm.~\ref{thm:systolic_klein_walls}] We must minimize the area of $W$ keeping the lengths of the four non-trivial simple curves $\geq n$ (because on every surface there is a shortest non-contractible curve that is simple). 
So we have to minimize the function $\Area(W)=\frac{r(r-1)}2+\frac{s(s-1)}2+u(r+s+t)$ of the variables $r,s,t,u\in\NN$ under the constraint that $u+r,u+s,2u,2t+r+s\geq n$, and $r,s,u$ are of the same parity. If $u\leq n$, then we perform on $W$ the following reductions:
\begin{itemize}
\item If $u+r>n$, we reduce $r$ in two units and increase $t$ in one unit, and repeat until $u+r=n$. In the same way we reduce $s$ and increase $t$ until $u+s=n$. Note that $r=s\leq\frac n2$ because $u\geq\frac n2$, according to one of the constraints.
\item Then we reduce $t$ as much as possible, until $2t+r+s=n$. This equality is attained because we ensured that $r+s\leq n$ in the previous step.
\end{itemize}
At this point we can express in terms of $r$ the other variables $s=r$, $t=\frac n2-r$ and $u=n-r$, so \[\Area(W)
=r(r-1)+(n-r)\left(\frac n2+r\right)
=\frac {n^2}2+\left(\frac n2-1\right)r\geq\frac{n^2}2.\]

In case when $u>n$, we do as follows. Since $u\geq n$, the constraints $r+u,s+u,2u\geq n$ are satisfied, so we only need to care about keeping $2t+r+s\geq n$ and the parity condition. If $u,r,s$ are odd, subtract one unit from each and increase $t$ in one unit so that $2t+r+s$ does not change. This reduces the area and leads us to the case in which $r,s,t$ are even and $u\geq n$. Then we repeateadly decrease $r$ or $s$ in two units and increase $t$ in one unit until $r=s=0$. At this point the only non-zero variables are $t,u\in\NN$ that satisfy $u\geq n$ and $2t\geq n$, so \[\Area(W)=tu\geq\frac{n^2}2.\qedhere\]
\end{proof}

This finishes the proof of Theorem~\ref{thm:systolic_klein_selfrev} that any self-reverse Finsler metric on the Klein bottle with systole $\geq L$ has $\Area_\uHT\geq 2L^2$.


\newpage
\section{Discretization of directed metrics on surfaces}\label{sec:fine}
In this section we present some theorems whose detailed proofs will be given in a later publication.

So far we have only discussed a discrete model when the Finsler metric on the surface is self-reverse. So, for example, our discrete analogue of Ivanov's cyclic content is also restricted to this case. However, Ivanov's original continuous argument works equally well for directed Finsler metrics. There is also a systolic inequality proved by Álvarez-Paiva--Balacheff--Tzanev~\cite{paiva2016isosystolic}, according to which the smallest area that a Finsler 2-torus of systole $L$ can have is $\frac32L^2$ if the metric is directed, and this contrasts with the smallest area that can be attained by a self-reverse metric, that is $2L^2$, as proved by Sabourau \cite[Thm. 12.1]{sabourau2010local}. Both bounds are attained by flat metrics.\footnote{The theorem of Paiva-Balacheff-Tzanev has the following discrete version: for each $n\in\NN$, if a lattice polygon touches each line $ax+by=n$ defined by a non-zero point $(a,b)\in\ZZ^2$, then the area of the polygon is at least $\frac32n^2$. Is there a discrete proof, that only uses lattice polygons? (See also \cite[Prob.~9.23]{beck2015computing}.)} Moreover, the systolic inequality for projective planes (Pu's inequality) has not been proved to hold for directed metrics; so far it has only been proved for self-reverse Finsler metrics \cite{ivanov2011filling}. Therefore it would be convenient to have a discrete model for directed Finsler metrics on surfaces (and for manifolds of higher dimension as well).

The idea of using a wallsystem made of random planes can only work for self-reverse metrics in dimension 2, since it relies on integral-geometric Crofton formulas that are only available when the normed plane embeds isometrically in $L^1$, or equivalently, when the dual unit ball can be approximated arbitrarily well by a \term{zonotope}, that is, a polyhedron that is the Minkowski sum of some straight segments. 
This happens to be the case for all self-reverse norms in dimension 2, but not for directed metrics nor in higher dimensions (see the preface to \cite{pogorelov1979hilbert}). So what shall we do?

The reduction of wallsystems on the disk using Steinitz moves was apparently rediscovered at least twice: first by Schrijver \cite{schrijver1991decomposition}, who, motivated by a problem of routing wires on integrated circuits, was studying a problem mathematically similar to ours (tightening wallsystems while preserving minlength function), and independently by Curtis--Ingerman--Morrow \cite{curtis1998circular} and Colin de Verdiere--Gitler--Vertigan \cite{verdiere1996reseaux2} who used them to simplify electrical networks on the disk. The electrical networks were provided with coefficients (the electrical resistance), and by discarding the coefficients we can get the square-celled surfaces and wallsystems that we considered here. 
In 2006, Postnikov described a cell decomposition of the totally positive Grassmanian that involved certain networks on the disk that are a directed analogue of the electrical networks \cite{postnikov2006total}. Based on the combinatorics of his networks (called perfectly oriented plabic graphs), and discarding the coefficients that Postnikov used (the directed analogue of electrical resistance), it is possible to construct discrete versions of directed Finsler surfaces, that will be called \term{fine surface}; they are the Poincaré duals of perfectly oriented trivalent plabic graphs. To keep this presentation self-contained, I will describe the discrete directed metrics independently of the work by Postnikov that originally inspired this discretization.

I call these discrete structures ``fine complexes'' (and, in particular, ``fine surfaces'') for reasons that will become clear below. However, I recently found that fine complexes are a simplicial sets, which where described already in 1950 by Eilenberg--Zilber~\cite{eilenberg1950semi}. More precisely, a fine complex is a simplicial set endowed with a metric that comes naturally from the simplicial set structure. Simplicial sets are widely employed to model the topology of spaces. However, their natural metric, that enables them to model the geometry as well, has not been exhibited before.

\subsection{Fine metrics on graphs and fine surfaces}

Let $G$ be a connected undirected graph and let $F$ be a \term{directed integral metric} on $G$, that is, a function $\overrightarrow E(G)\to\NN$ where $\overrightarrow E(G)$ is the set of directed edges of $G$. The length $\Len_F(\gamma)$ of a directed path or cycle $\gamma$ in $G$ is defined as the sum of the lengths of the directed edges of $\gamma$, and the distance $d_F(x,y)$ between two vertices $x,y\in V(G)$ is the minimum length of a path from $x$ to $y$. The function $d_F:V(G)\times V(G)\to\NN$ is a directed distance function on the set of vertices $V(G)$.

A \term{directed distance function} on a set $X$ is a function $d:X\times X\to[0,+\infty)$ such that $d(x,x)=0$ and $d(x,z)\leq d(x,y)+d(y,z)$ for every $x,y,z\in X$. Note that we may have $d(x,y)=0$, which does not imply that $d(y,x)=0$. The distance function is called \term{non-degenerate} if \[d(x,y)=0\text{ and }d(y,x)=0\text{ together imply that }x=y\] for every two points $x,y\in X$. The metric $F$ is called non-degenerate if $d_F$ is non-degenerate.

The \term{size} (or uHT volume) of an edge $e\in E(G)$ according to the integral metric $F$ is $\Vol_\uHT(e,F)=F(v)+F(-v)\in\NN$, where $v,-v$ are the two directed edges $v,-v$ that become $e$ when their orientation is discarded. An edge $e$ is called degenerate if its size is zero. 
Note that a non-degenerate edge may be refined (subdivided) into $\Vol(e)$ edges of size 1 without affecting the length of paths. An edge of size 1 is called a \term{fine edge} because it cannot be refined further. The metric $F$ is called a \term{fine metric} if all the edges are fine.

A fine metric $F$ on a graph $G$ is equivalent to an orientation of the graph edges, but the orientation is interpreted as follows. It is possible to travel along an edge in both directions. In the reverse direction the trip is free of charge, and in the forward direction it gives rise to a ``fine'' of one unit. Note that the distance $d_F$ is non-degenerate if and only if the orientation of the edges is acyclic. 

Our discrete directed surfaces will have fine graphs as 1-skeletons. What about 2-dimensional cells?

Recall that in the self-reverse theory, our way to impose a discrete metric on a surface $M$ is to embed a bipartite graph $M^{\leq 1}$ into $M$. We also required that each of the complementary regions be a four-sided topological disk. However, if the cell is $2n$-sided (and this would happen if we dualized a wallsystem where $n$ walls cross at a single point), then we can subdivide the $2n$-gon into $\frac{n(n-1)}2$ square-cells without modifying the distances between the original points (which corresponds to perturbing the walls so that the multiple crossing is decomposed into $\frac{n(n-1)}2$ simple crossings).

Something similar can be done in the directed theory. Suppose a fine graph $(G,F)$ is embedded in a surface $M$, so that each complementary region is a disk whose boundary is a fine cycle. 
We will subdivide this disk into fine triangles.

A \term{fine simplex} is a combinatorial simplex with a non-degenerate fine metric on its 1-skeleton, which is a complete graph. The non-degeneracy condition implies that the edges are directed acyclically, therefore they induce a total order on the vertices, the unique total order $\prec$ that satisfies \[d(x,y)=\begin{cases}1&\text{if }x\prec y,\\0 &\text{otherwise.}\end{cases}\] This order is called the \term{fine order} of the vertices. A \term{regular fine complex} is a combinatorial object made of fine simplices glued face to face, respecting the orientation of the edges. In particular, a \term{regular fine surface} is a combinatorial surface made of fine triangles glued side to side, respecting the orientation of the edges. 
The area of the surface is the number of fine triangles that it contains.

\subsection{Directed version of the filling area conjecture}

A fine cycle graph $C$ is denoted $C_{a,b}$ if $\Len_F(C^+)=a$ and $b=\Len_F(C^-)=b$, where $C^+$ is $C$ oriented in some way and $C^-$ is $C$ oriented in the opposite way. Note that there are many different fine cycles $C_{a,b}$.

An \term{isometric filling} of a fine cycle $C$ is a fine surface $M$ with $\partial M=C$ and such that $d_M(x,y)=d_C(x,y)$ for every two boundary vertices $x,y\in V(C)$.

A fine cycle $C=C_{a,b}$ can be filled isometrically with a fine disk.

\begin{exer} Using fine triangles, maybe more than necessary, construct a disk that fills isometrically the fine cycle $C_{a,b}$. 
\end{exer}

The \term{fine filling area} of the fine cycle $C=C_{a,b}$ is the minimum area of a fine surface that fills isometrically the cycle $C$. It can be proved that this number depends only on the lengths $a=\Len(C^+)$ and $b=\Len(C^-)$.


The directed versions of the continuous and discrete FAC are the following.

\begin{conjecture}[Directed Finsler FAC]\label{conj:directed_FAC} If a surface $M$ with a directed Finsler semimetric $F$ fills without shortcuts a Finsler closed curve $(C,G)$ with $\Len_G(C^+)=a$ and $\Len_G(C^-)=b$, then $\Area_\uHT(M,F)\geq\frac{ab}2$.
\end{conjecture}

\begin{conjecture}[Directed discrete FAC, or fine FAC]\label{conj:fine_FAC} Every fine surface $M$ that fills isometrically a fine cycle $C=C_{a,b}$ has $\Area(M)\geq 2ab-a-b$.
\end{conjecture}

\begin{figure}
  \centering
  \includegraphics[width=.4\linewidth]{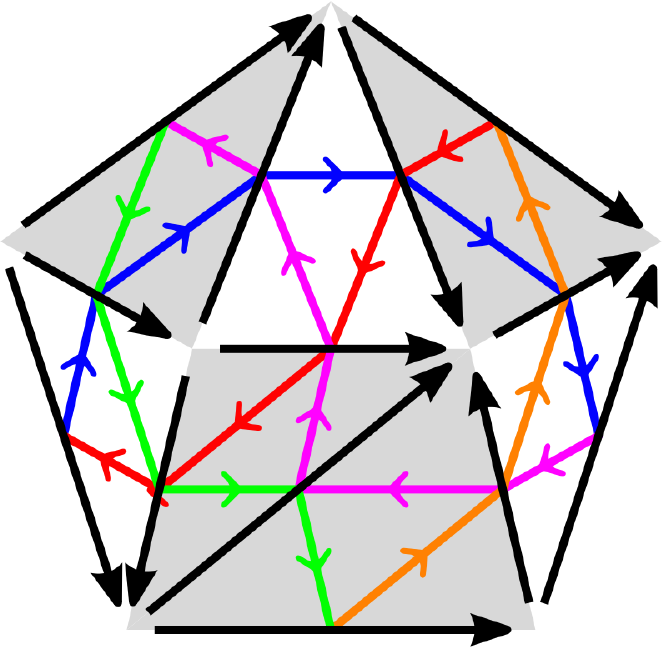}
  \caption{A fine disk $M$ that fills isometrically a cycle $C_{a,b}$ and has $\Area(M)=2ab-a-b$ for $a=2$ and $b=3$.}\label{fig:fine_hemisphere}
\end{figure}

My main theorems about these conjectures are the following.

\begin{theorem}\label{thm:fac_fine_disks} The directed discrete FAC holds when $M$ is homeomorphic to a disk.
\end{theorem}

\begin{theorem}[Equivalence between discrete and continuous directed FACs] Conjectures \ref{conj:directed_FAC} and \ref{conj:fine_FAC} are equivalent, and the equivalence holds separately for each topological class of surfaces $M$.
\end{theorem}

An outline of the proofs is given below and the detailed proofs will appear in another publication.

The proof that the fine FAC implies the directed Finsler FAC is similar in structure to the proof of the analogous statement for self-reverse metrics, Theorem~\ref{thm:walledFAC_implies_FinslerFAC}. The main change is that the Lemma~\ref{thm:discretize_plane_selfrev} is replaced by the following lemma.

\begin{lemma}[Discretization of directed integral seminorms on the plane, or fine torus lemma]\label{thm:discretize_plane_directed} Let $K\subseteq\RR^2$ be a nondegenerate integral convex polygon that contains the origin. Define the integral seminorm $v\in\ZZ^2\mapsto\|v\|_K:=\max_{\phi\in K}\phi(v)$. Then there is a $\ZZ^2$-periodic fine structure $\widetilde F$ on $\RR^2$ that contains exactly $4|K|$ fine triangles, up to integer translations, and such that \[d_{\widetilde F}(x,x+v)=\|v\|_K\] for every vertex $x$ of $\widetilde F$ and every integral vector $v\in\ZZ^2$.
\end{lemma}

\begin{figure}
  \begin{subfigure}[t]{.45\textwidth}
    \centering
    \includegraphics[width=.85\textwidth]{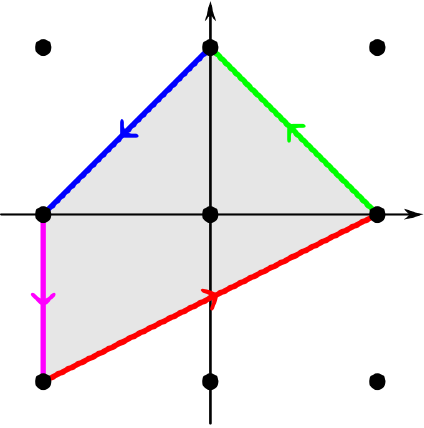}
    \caption{The integral polygon $K=\Conv\{(1,0),(0,1),(-1,0),(-1,-1)\}$, whose counterclockwise sides are the vectors $(-1,1)$ (green), $(-1,-1)$ (blue), $(0,-1)$ (magenta) and $(2,1)$ (red).}
  \end{subfigure}
  \hspace{.05\textwidth}
  \begin{subfigure}[t]{.47\textwidth}
    \centering
    \includegraphics[width=.8\textwidth]{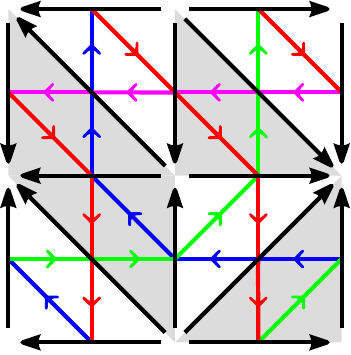}
    \caption{The periodic fine structure on $\RR^2$ restricted to the unit square $[0,1]^2$. If we glue the opposite sides of this square we obtain a fine torus whose Postnikov strands are closed curves of homotopy classes $(1,1)$ (green), $(-1,1)$ (blue), $(-1,0)$ (magenta) and $(1,-2)$ (red).}
  \end{subfigure}
  \caption{Example of an integral norm $\|-\|_K$ given by its dual unit polygon $K\subseteq\RR^2$ and a $\ZZ^2$-periodic fine structure on $\RR^2$ such that $d(x,x+v)=\|v\|_X$ for every vertex $x$ and every integral vector $v\in\ZZ^2$.}
\end{figure}

Note that the fine structure $\widetilde F$ yields a fine structure $F$ on the torus $M=\RR^2/\ZZ^2$. 
An interesting property of $F$ is that the homotopy classes of its Postnikov strands (defined in the next subsection) are the counterclockwise sides of the polygon $K$, broken into primitive integral vectors and rotated a quarter turn. The construction of the fine structure $\widetilde F$ has many steps in common with the proof of \cite[Thm.~2.5.i)]{goncharov2013dimers}, however, the proof that this fine structure has the desired area and distances is not trivial.

The proof that the directed Finsler FAC implies the fine FAC is based on a method for transforming every regular fine surface $M$ into a polyhedral-Finsler surface $(|M|,F)$. The polyhedral surface $|M|$ is just the geometric realization (as described by Milnor~\cite{milnor1957geometric}) of the fine surface $M$, considered as a simplicial set. It is a surface made of copies of the \term{standard triangle} $T=\{(x,y)\in\RR^2:0\leq x\leq y\leq 1\}$, with ordered vertices $(0,0)<(0,1)<(1,1)$. On each copy of this triangle $T$ we put the \term{standard seminorm} $\|(v_0,v_1)\|=\max\{v_0,v_1,0\}$ to obtain the polyhedral-Finsler semimetric $F$. If the original fine surface $M$ is an isometric filling of its boundary, then the surface $(|M|,F)$ becomes an isometric filling of its boundary after we insert some ``eyes''. The eyes are copies of the standard triangle where one of the sides is collapsed to a single point. Curiously, the surface with eyes is the geometric realization $(|M_{\mathrm{eyes}}|,F)$ of a modified simplicial set $M_{\mathrm{eyes}}$.

\subsection{Duality between fine surfaces and plabic graphs}

Let $M$ be an oriented fine surface and let $G$ be its 1-skeleton graph. We can construct a trivalent perfectly oriented plabic graph $G^*$ that is dual to $G$, as follows.\footnote{Postnikov's original plabic graphs \cite{postnikov2006total} are PLAnar and BIColored. Our plabic graphs are more general because they are embedded in a surface rather than in the plane.} The graph $G^*$ will have exactly one vertex inside each fine triangle of $M$. The vertex is black (and the triangle is also called black) if the boundary of the triangle is mostly counterclockwise (that is, it has two edges with counterclockwise orientation and only one edge with clockwise orientation), and white if the boundary is mostly clockwise. To obtain the edges of $G^*$ we rotate each fine edge $e$ of $M$ in the counterclockwise direction, about one quarter turn, obtaining a directed edge $e^*$ that connects the two colored vertices separated by $e$, or connects one vertex to the exterior of the surface if $e$ is a boundary edge of $M$. Note that every vertex of $G^*$ has valency 3. Each black vertex has two edges oriented outwards and one inwards, and each white vertex has two edges oriented inwards and one outwards.

Before outlining the proof of Theorem~\ref{thm:fac_fine_disks}, we translate to the language of fine surfaces some features of Postnikov's theory \cite{postnikov2006total} that are needed here:
\begin{itemize}
\item Any fine disk can be \term{tightened} or \term{reduced} using Postnikov's reductions,\footnote{Postnikov distinguished between reductions (that are irreversible operations) and moves (whose inverse operation is also a move), however, I call all the operations reductions.} that are operations that reduce the area without changing boundary distances.
\item On any fine surface, we define Postnikov's \term{trips} or \term{strands}, that are the maximal directed curves that follow the \term{rules of the road}:
\begin{itemize}
\item A strand may only intersect the 1-skeleton of $M$ at the midpoint of a fine edge $e$. If $e$ is an interior edge, then the strand must cross the edge, and if $e$ is a boundary edge, then the intersection must be the beginning or the end of the strand.
\item Inside each fine triangle with vertices $(v_i)_{0\leq i<3}$, a strand must go straight from the midpoint of one edge $[v_i,v_{i+1}]$ to the midpoint of the edge $[v_{i+1},v_{i+2}]$, that is cyclically next according to the fine order of the vertices. (The indices are considered integers modulo 3.)
\end{itemize}
Note that each strand is either closed or a path from boundary to boundary.

Note that when an edge $e$ separates two triangles of different color, the midpoint of $e$ is topologically a crossing between strands, and if the colors at either side of $e$ are the same, then the midpoint of $e$ is topologically a tangency between strands that have opposite directions. In this case, we may modify the point of crossing of the strands with the edge so that they do not intersect. In this way we obtain another version of the systems of strands, without tangencies.
\item A fine disk $M$ is tight (its area cannot be reduced by Postnikov reductions and moves) if and only if its strands form a \term{directed pseudoline arrangement} (or \term{DIPLA}); this means that 
for any two points $x,y\in M$, there is at most one trip from $x$ to $y$ along a strand, in the direction of the strand. Equivalently, the system of strands is a DIPLA if and only if they are simple paths that do not form any \term{directed bigon}, that is a bigon with vertices $x,y$ where the two sides are directed from $x$ to $y$. (Note that a finite set of maximal geodesics of a simple Finsler disk necessarily form a directed pseudoline arrangement.)
\item Two tight fine disks that have the same boundary are equivalent by Postnikov moves if and only if their Postnikov strands induce the same permutation of the boundary.
\item (This is not in Postnikov's paper.) On a fine surface $(M,F)$ one can consider the \term{short functions} (or \term{short 0-forms}), that are the functions $h:M^0\to\ZZ$ that satisfy $h(y)-h(x)\leq d_M(x,y)$ for each $x,y\in M^0$. More generally, one can define \term{short maps} $h:M^0\to Z^0$, that satisfy \[d_Z(h(x),h(y))\leq d_M(x,y)\qquad\forall x,y\in M^0.\] The above definition is obtained when $Z=\ZZ$ and $d_\ZZ(m,n)=\max\{n-m,0\}$, which is the \term{standard fine distance} in $\ZZ$.

The differential $\theta=\diff h$ of a short function $h:M^0\to\ZZ$ is an exact calibration, that is a kind of 1-form. A \term{1-form} is a function $\theta:\overrightarrow{M^1}\to\ZZ$ (where $\overrightarrow{M^1}$ is the set of directed edges of $G$) that satisfies $\theta(-e)=-\theta(e)$. A \term{calibration} is a 1-form that is closed (its integral over the boundary cycle of each oriented triangle is zero) and bounded above by the metric $F$ (in the sense that $\theta(e)\leq F(e)$ for each directed edge $e\in\overrightarrow{M^1}$).

For any calibration $\theta$ we can define a new fine metric $F'=F-\theta$. The distances $d_F$ and $d_{F'}$ on the set of vertices $M^0$ are different, but the shortest paths of both metrics are the same, and the triangular defects $d(x,y)+d(y,z)-d(x,z)$ also coincide. This implies that on each triangle, the ordering of the vertices induced by the fine structure only changes by a cyclic shift (check it!), so the Postnikov strands of both metrics $F$ and $F'$ are the same.\footnote{If the fine surface is oriented, one can color each fine triangle black if the fine total ordering of its vertices is anticlockwise, or white if it is clockwise, and this colors do not change when we switch from $F$ to $F'$.}

Reciprocally, if $(M,F)$ is a fine surface and $F'$ is a new fine metric on $(M,F)$ that induces the same cyclical order as $F$ on each triangle, then the difference $\theta=F-F'$ is a calibration 1-form with respect to $F$ (a closed 1-form bounded above by $F$).

These facts also hold for a simplicial complex of any dimension, with a fine metric on its 1-skeleton. (But note that Postnikov strands are only defined on surfaces.)

\item Fix a fine disk $(M,F)$, where $F$ is considered as a orientation of the edges. Let $a$ (resp. $b$) be the number of counterclockwise (resp. clockwise) boundary edges. The replacement of $F$ by $F'=F-\theta$ is called a \term{perfect reorientation} of $F$ in Postnikov's theory. It preserves the cyclic order of each triangle. Let $S_{F'}$ be the set of boundary edges that are oriented counterclockwise by $F'$; it also has $a$ elements. The family of all sets $S_{F'}$ (for all perfect reorientations $F'$ of $F$) is a matroid of rank $a$ on the set of $a+b$ boundary edges $E(\partial M)$, called the \term{positroid} of $(M,F)$.
\end{itemize}

A minimum fine disk that fills isometrically $C_{a,b}$ can be obtained as follows. 

\begin{example}[Minimal fine disk that fills isometrically $C_{a,b}$] Pick any fine disk that fills isometrically $C_{a,b}$, as described above. Then apply Postnikov's reductions. To compute the number of cells we do as follows. According to Postnikov \cite{postnikov2006total}, each fine disk corresponds to a $d$-dimensional cell of the totally non-negative Grassmanian $G_{k,n}^{\mathrm{tnn}}$, where $k=a$ and $n=a+b$. If the fine disk is reduced, it has $d+1$ vertices (because the dual plabic graph has $d+1$ faces, according to \cite[Thm.~12.7]{postnikov2006total}). Then the number of faces of the fine disk $(M,F)$ is $|M^2|=2d-a-b$.\footnote{This follows from Euler's formula $1=\chi(M)=v-e+f$, where $e=|M^0|$, $e=|M^1|$ and $f=|M^2|$ are the numbers of vertices, edges and triangular faces of the fine disk. To express $f$ in terms of $d$ we substitute $v=d+1$ and $e=\frac{3f+a+b}2$. The last formula is obtained by counting incidences of 2-dimensional faces on edges, taking into account that each triangular face incides on three edges, the ``exterior face'' incides on the $a+b$ boundary edges, and each edge is incided on twice.} Since the totally non-negative Grassmanian has dimension $k(n-k)=ab$, it follows that the fine disk is made of $2ab-a-b$ triangles at most. But we will see that no fine disk that fills isometrically $C_{a,b}$ can have less than $2ab-a-b$ cells.
\end{example}

The filling described above corresponds to the top cell of the totally non-negative Grassmanian.

\begin{proof}[Proof of Theorem~\ref{thm:fac_fine_disks}] Let $(M,F)$ be a fine disk that fills the cycle $C=C_{a,b}$ isometrically. We must prove that $|M^2|\geq 2ab-a-b$. Every short function $h_C:C\to\ZZ$ can be extended to a short function $h$ on $M$, and in fact the minimum extension is given by \term{Whithney's formula}\footnote{This formula was similarly used in the proof of \cite[Thm.~B]{cossarini2016intersection}.} \cite{whitney1934analytic} \[h(x)=\max_{y\in C}h_C(y)-d_F(x,y).\]
This implies that all reorientations of the boundary (that preserve the numbers $a=\Len(+\partial M)$ and $b=\Len(-\partial M)$) can be attained. In other words, the positroid of $(M,F)$ is full (the uniform matroid), so the fine disk $(M,F)$ corresponds to the top cell of the totally nonnegative Grassmanian \cite{postnikov2006total}. This implies that $(M,F)$ has at least $2ab-a-b$ triangles, because it has exactly $2ab-a-b$ triangles after reduction by Postnikov's operations.
\end{proof}

Another proof of the FAC for fine disks can be obtained using a discrete version of Ivanov's cyclic content based on discrete differential forms on triangulated surfaces. The wedge product that must be used is the one defined by Castrillon Lopez, described in~\cite{hirani2003discrete}.

Yet a third proof of the FAC for fine disks can be obtained by tightening and then using the following lemma.

\begin{lemma}[Boundary distance formula for tight fine disks]\label{thm:boundary_distance_fine_tight_disk} Let $(M,F)$ be a tight fine disk, let $x,y\in\partial M$, and let $\gamma$ be one of the two simple curves from $x$ to $y$ along $\partial M$. Then $d(x,y)=\Len(\gamma)-n$, where $n$ is the number of Postnikov strands that form an oriented bigon with $\gamma$. 
\end{lemma}

\begin{proof}[Proof (sketch)] The proof is by induction. As proved by Postnikov~\cite{postnikov2006total}, the tight disk $M$ is equivalent by Postnikov moves to a disk $M'$ that can be obtained from a single triangle by attaching boundary triangles or lozenges (pairs of triangles) along the boundary. Then to finish this proof one must show that the formula holds for a single triangle and continues to hold after each boundary triangle or lozenge is attached.
\end{proof}

In the detailed proof of the last lemma, the following fact is useful.

\begin{lemma}[Alternative for fine graphs]\label{thm:fine_alternative} Let $(G,F)$ be a non-degenerate fine graph and let $y,x,x'\in V(G)$ be vertices such that $x,x'$ are connected by an edge $e\in E(G)$. Then exactly one of the two following equalities holds: \[d(x,x')+d(x',y)=d(x,y)\quad\text{or}\quad d(x',x)+d(x,y)=d(x',y).\]
\end{lemma}

The proof of Lemma~\ref{thm:fine_alternative} is trivial but the fact is curious because an analogous proposition in continuous (Finsler) geometry would be ridiculous.

A corollary of Lemma~\ref{thm:boundary_distance_fine_tight_disk} is the following characterization of geodesics.

\begin{theorem} In a tight fine disk $(M,F)$, a path $\gamma$ in $|M^1|$ from boundary to boundary is a shortest path if and only if it does not form directed bigons with the Postnikov strands.
\end{theorem}


Using fine surfaces it is also possible to prove Pu's inequality for directed Finsler metrics (an original result), which follows from the following discrete version.

\begin{theorem} Let $(M,F)$ be a fine surface homeomorphic to the projective plane, such that the length of every non-contractible closed curve is $\geq n$. Then the number of triangles of $M$ is at least $2n(n-1)$.
\end{theorem}

The proof is based on the following lemma. A \term{bi-shortest path} is a shortest path from $x$ to $y$ whose reverse is a shortest path from $y$ to $x$.

\begin{lemma} Let $M$ be a tight fine disk, and let $x,y\in\partial M$ be two boundary vertices. Then there is a tight fine disk $M'$, that has the same boundary and boundary distances than $M$ (therefore it is equivalent to $M$ by Postnikov moves), and such that the points $x,y$ are joined by a bi-shortest path along $M'$.
\end{lemma}

\newpage
\section{Work to do next}\label{sec:todo}

We finish this thesis with a list of questions, ideas and plans.

\subsection{Minimum filling by integer linear programming} The problem of minimally filling a $0$-dimensional cycle with integer (resp. binary) coefficients is the optimal transportation problem (resp. the optimum matching problem), which can be solved efficiently by linear programming. The polytope of feasible solutions is integral by a theorem of Birkhoff-von Neumann (resp. Edmonds). The Birkhoff-von Neumann theorem says that the minimum transportation cost is not reduced if we allow splitting each load into pieces.

The problem of isometrically filling the 1-dimensional cycle $C_{2n}$ can also be approached by linear programming. The first step is to restrict the space of possible isometric fillings. For each fixed $n$, all even isometric square-celled fillings of $C_{2n}$ can be mapped into a large but finite graph $\overline{C_{2n}}\supseteq C_{2n}$, called the \emph{injective hull} of $C_{2n}$.\footnote{Metric injective hulls were discovered and constructed (for self-reverse continuous metrics) by Isbell \cite{isbell1964six}. Here we need the injective hull $\overline{C_{2n}}$ of $C_{2n}$ in the category of \emph{bipartite} discrete self-reverse metric spaces, that are metric spaces where $d(x,y)=d(y,x)\in\NN$, each vertex has one of two colors and $d(x,y)$ is even if and only if $x,y$ have the same color). Isbell also proved that every non-positively curved square-celled disk is injective (the proof is easier in the discrete category), and this implies that it is the unique minimum filling of its boundary distance.} The hull is an $n$-cube graph $(0-1)^n$ with many additional diagonal edges.\footnote{If we disallow using the diagonal edges, then no filling can have less than $\frac{n(n-1)}2$ squares. This is proved in \cite[5.2; k=2]{dotterrer2016filling}.} Then each oriented (resp. unoriented) filling of $C_{2n}$ can be expressed as a linear combination of directed 4-cycles of $\overline{C_{2n}}$ with positive integer (resp. binary) coefficients, and the total area is a linear function of these coefficients (resp. the Hamming weight). If we allow \emph{rational} linear combinations, that correspond to fillings by rational chains instead of surfaces, then the fillings form a polyhedron, whose integral points correspond to oriented square-celled fillings, so the filling area problem is a problem of integer linear programming (resp. some error-correcting code). If the cosmos favors us, these problems will have a special structure that will allow us to solve them efficiently, as happened with the problems of minimal filling of 0-dimensional cycles (optimal transportation and optimal matching).

Let us stick to the case of oriented fillings. If all the vertices of the polyhedron of fillings are integral,\footnote{I briefly tested this recently (only for $n=3$, with a quick script), and I only found vertices with $0-1$ coordinates.} then our integer linear programming problem will relax to linear programming with rational variables, that may be easier to solve (note however that the size of this problem is already exponential on $n$). For each $n$ we would then have \[f_\QQ(n):=\Fill_{\QQ}\Area_\uHT(C_{2n})=\Fill_{\ZZ}\Area_\uHT(C_{2n})=:f_{\ZZ}(n)\in\ZZ.\]

Burago and Ivanov have already proved \cite{burago2002asymptotic}, based on the work by Busemann-Ewald-Shephard \cite{busemann1963convex}, that there are Finsler chains with rational coefficients that can isometrically replace a Euclidean disk but have less area than the disk (I computed an approximate version of their example and got a saving of about 5 parts in 10.000). Since a hemisphere can approximately be made of Euclidean disks, this implies that the (self-reverse) rational Finsler filling area of the circle is a nontrivial constant strictly smaller than the area of the hemisphere (with the same saving as the disk or perhaps more). By discretization, this implies that the minimum number of squares $f(n)\leq 2n(n-1)$ is a nontrivial rational number for sufficiently large $n$. Task: Find the first (or some) of these nontrivial values. 
Note that if the polyhedron of fillings has integral vertices, then the function $f_\ZZ:\NN\to\NN$ will also be non-trivial (and the Finsler FAC will be false).

I suspect that fine surfaces are better behaved, so I would rather compute the fine filling area of the fine cycle $C=C_{a,b}$ that has $a=\Len(C)$ positive edges and $b=\Len(-C)$ negative edges (the order in which they are distributed is irrelevant). So we should calculate
\[f_\ZZ(a,b):=\Fill_\ZZ\Area_{fine}(C_{a,b})\in\ZZ,\] 
\[f_\QQ(a,b):=\Fill_\QQ\Area_{fine}(C_{a,b})\in\QQ,\]
(the fine area is just the number of triangles of a surface), and again the second function is guaranteed to be nontrivial ($<2ab-a-b$ for sufficiently large $a,b$). To handle the problem as a linear programming problem one can construct fine injective hull by copying the continuous directed version \cite{kemajou2012isbell} in the category of \emph{discrete} directed metric spaces. If the vertices of the feasible polyhedron are integral, then the two functions $f_\QQ$ and $f_\ZZ$ will be equal, and we will have a nontrivial function $\NN\times\NN\to\NN$. Another reasonable question is whether $f(n,n)=4\,f(n)$ (both for $f_\QQ$ and for $f_\ZZ$). 



\subsection{Poset of minlength functions and algorithms} Consider only even wallsystems $W$ on a fixed topological surface $M$, where $\partial W=W\cap \partial M=:T$ is also fixed. Consider the poset of tight wallsystems $[W]$, ordered by their minlength function: $[W']\leq[W]$ if and only if $\Len_{[W]}\leq\Len_{[W']}$. 
The following is known (see also \cite{schrijver2003combinatorial}):
\begin{itemize}
\item If $M$ is a projective plane, then a tight even wallsystem is characterized by the systole $n\in 2\NN$ so the poset of minlength functions is isomorphic to $\NN$.
\item If $M$ is a disk with $\partial M=C_{2n}$ (equivalently, $|T|=2n$), then the poset is graded by the number of crossings of each wallsystem, and there is an \emph{explicit} characterization\footnote{The definition of $[W']\leq[W]$ by minlength function is considered \emph{implicit} because it says how to \emph{verify} that $[W']\leq[W]$, but not how to \emph{generate} those $[W']$ that are smaller than a given $[W]$.} of the order relation: $[W']\leq[W]$ if and only if $[W']$ can be obtained from $[W]$ by uncrossings. We will say that $(M,T)$ has the \term{descent property} when this happens. The descent property in this case follows from the Okamura-Seymour theorem \cite{okamura1981multicommodity}.
\item A simpler variation of the last case is the following: if $M$ is a disk with $\partial M=C_{2n}$, and we fix two antipodal boundary vertices $a$,~$b$, and restrict to the tight wallsystems $[W]$ such that $d_W(a,b)=n$, then these wallsystems form an upwards-closed subset of the poset of all tight wallsystems. (Actually, it is an interval since it has bottom and top element.) This poset is isomorphic to the set of permutations $S_n$ (because each wall goes from one component of $\partial M\setminus\{a,b\}$ to the other, hence the wallsystem defines a permutation of $n$ objects) with the strong Bruhat order. One can also consider $M$ as a minimal bigon, whose sides are the two pieces of boundary that go from $a$ to $b$. In this case the proof of the descent property is simpler and can be found in \cite[Prop. 3.1]{zhao2007bruhat}. The poset is graded by area.
\item If $M$ is a torus, then the minlength function of each $W$ is an even-valued integral seminorm characterized by its dual unit ball, which is a symmetric integral polygon with even vertices \cite{schrijver1993graphs}. The tight classes of wallsystems then correspond to integral polygons ordered by inclusion. This poset is graded: the rank of a polygon is the number of pairs of opposite non-zero even integral points it contains. If $[W]$, $[W']$ are tight wallsystems with $\Len_{[W']}\leq\Len_{[W]}$, then again we can obtain $[W']$ from $[W]$ by uncrossings (this is proved in \cite{frank1992edge}, but there is a proof using discrete differential (to be written). Note that the grading is not by area but by number of integer points, that is roughly proportional to area (by Pick's formula).
\item If $M$ is a Möbius band and we consider even wallsystems made of closed walls (so the length of the boundary is zero), then each minlength function is characterized by certain Newton polygon. The poset is graded, but the proof of this fact is not yet written. I do not know if the descent property holds in this case. See related material in \cite{graaf1995characterizing,graaf1997making}.
\item If $M$ is a cylinder with the two boundary components of the same  length $n$, we can consider tight wallsystems where each nonclosed wall joins different components of the boundary (this is an upwards-closed subset of the poset), and further restrict to tight wallsystems where no walls are closed (this is a downwards-closed subset of the latter). The wallsystems of this kind are equivalent to affine permutations $u:\ZZ\to\ZZ$ (bijections that satisfy $u(x+n)=u(x)+n$; see \cite[Sect.~8.3]{bjorner2005combinatorics}), classified into mutually incomparable classes by their average advance $\sum_{0\leq i<n}u(i)-i$. Restrict to one class. By \cite[Thm. 8.3.7]{bjorner2005combinatorics}, it is order-isomorphic to the affine permutation group $\tilde A_{n-1}$ with the strong Bruhat order. By the subword property \cite[Thm.~2.2.2]{bjorner2005combinatorics}, we have $[W']\leq[W]$ if and only if $W'$ can be obtained from $[W]$ by uncrossings.
\item If $M$ is a closed orientable surface of genus 2, then Schrijver found \cite[Fig. (109)]{schrijver1991decomposition} that \emph{the descent property fails}: there are two tight wallsystems $W$, $W'$ such that $[W']\leq[W]$ but $[W']$ cannot be obtained from $[W]$ by uncrossings.\footnote{In Schrijver's counterexample the wallsystems are not even but have the same parity. One can add walls to make the wallsystems even and the counterexample still works.} However, Schrijver and de Graaf proved that for all surfaces a ``rational'' version of the descent property holds: if $[W']\leq[W]$, then the tight wallsystem $[W']$ (seen as a multiset of homotopy classes of curves) is a convex combination of tight wallsystems that can be obtained from $[W]$ by uncrossings \cite{schrijver1991decomposition,graaf1997decomposition}. 
\item It would be interesting to know if the situation is better for fine structures on surfaces. I checked that when Schrijver's counterexample (actually, an even version of it) is translated to a fine structure, then it ceases being a counterexample: the fine version of $W'$ can be obtained from the fine version of $W$ by edge contractions and reversible moves that preserve the minlength function (including exact reorientation of the edges). I did not do any further experiments. The possibility of grading the poset should also be studied in the fine case.
\item In the case of fine disks, Postnikov has proved a kind of descent property \cite[Corollary 17.7]{postnikov2006total}. His study of circular Bruhat order and some additional work imply that if two tight fine disks $D,D'$ have the same boundary and $D'$ has smaller boundary distances than $D$, then the smaller fine disk can be obtained from the larger one by a perfect reorientation and a sequence of edge contractions. The intermediate disks obtained in the process are all tight.
\end{itemize}

\subsection{Discretization of $d$-dimensional metrics}\label{sec:finetorus} To discretize a Finsler $d$-manifold, it would be useful to know how to discretize each integrally normed $d$-dimensional space. 

\begin{conjecture}[Discretization of $d$-dimensional integral seminorms, or fine torus conjecture]\label{conj:discretization_normed} If $K\subseteq\RR^d$ is a nondegenerate integral convex polytope that contains the origin, then there is a $\ZZ^d$-periodic fine structure $\widetilde F$ on $\RR^d$ such that \[d_{\widetilde F}(x,x+v)=\max_{\phi\in K}\phi(v)=:\|v\|_K\] for every $v\in\ZZ^d$ and every vertex $x$ of $\widetilde F$, and the fine structure contains exactly $d!^2|K|$ cells of dimension $d$, up to integer translations.
\end{conjecture}

The discretization conjecture can also be stated in terms of a fine structure $F=\widetilde F/\ZZ^d$ on the torus $M=\RR^d/\ZZ^d$.

A fine structure $F$ that has less than $d!^2|K|$ fine simplices would probably contradict Burago--Ivanov's conjecture \cite[Conj. A]{burago2002asymptotic} that a flat (constant) Finsler metric on the torus $\TT^d=\RR^d/\ZZ^d$ 
cannot be replaced with a Finsler metric of smaller volume without decreasing the minlength function. This is because there is a way to turn fine manifolds into Finsler manifolds with proportional area and the same or very similar minlength function. For example, in dimension 3, we replace each fine tetrahedron by a copy of the tetrahedron $\Conv\{(0,0,0),(0,0,1),(0,1,1),(1,1,1)\}$ of the normed space $\RR^3$ with dual unit ball $B^*=\Conv\{0,e_0,e_1,e_2\}$. 

If the discretization conjecture is true, we would get a discrete model for normed spaces that do not admit Crofton formulas. Crofton formulas are only avaiable for hypermetric normed spaces, that is, those that can be embedded isometrically in $L^1$. Most normed spaces of dimension 3 and above are not hypermetric. For example, the $d$-dimensional $\ell^\infty$ space is not hypermetric if $d\geq 3$. 
Although I have not managed to discretize normed spaces of dimension 3 and above, I comment that it does not seem useful to restrict to self-reverse norms, because this hypothesis does not imply that the dual unit ball is a zonotope. To discretize 3-dimensional metrics, \emph{it seems necessary to allow metrics that are directed at the small scale, even if they will look self-reverse at the large scale.}

The status of Conjecture~\ref{conj:discretization_normed} is the following:

In dimension 2 the conjecture is true, as stated in Theorem~\ref{thm:discretize_plane_directed}.

In dimension 3, the conjecture is very interesting because it is not at all clear what could play the role of the Postnikov strands and the quarter-turn. I have tried to produce the fine structure for two polytopes: the tetrahedron and the cube.

When $K$ is the tetrahedron $\Conv\{0,e_0,e_1,e_2\}$ the conjecture is true. The fine structure has 6 simplices and is easy to find. The fundamental domain $[0,1]^3$ is partitioned into six tetrahedra $A_\sigma=\{x\in[0,1]^3:x_{\sigma(0)}\leq x_{\sigma(1)}\leq x_{\sigma(2)}\}$, where $\sigma\in S_3$ is a permutation of $\{0,1,2\}$. All the edges are oriented in the direction of increasing coordinates.\footnote{This construction can also be obtained as a product of simplicial sets.}

When $K$ is the cube $[0,1]^3$, the fine structure should have 36 tetrahedra. The best I could get so far is 40, using a greedy computer program.

Additionally, there are three general ways to produce new proven cases of Conjecture~\ref{conj:discretization_normed} from old ones:

\begin{itemize}
\item Let $F$ be a fine structure on the $d$-torus that solves the conjecture for certain integral polytope $K\subseteq\RR^d$. Then for any integral linear transformation $T:\RR^d\to\RR^d$, the fine structure $F'=[T](F)$ is a solution for the polytope $K'=T^*(K)$. The symbol $[T]$ denotes the self-map of the $d$-torus defined by the formula $[x]\mapsto[Tx]$, where $x\in\RR^d$ and $[x]$ is its class modulo $\ZZ^d$.
\item\footnote{This construction may possibly be extended using product of simplicial sets.} Let $F$ be a fine structure on the d-torus $M=(\RR^d/\ZZ^d)$ that solves the conjecture for certain integral polytope $K\subseteq\RR^d$. Then we can solve the conjecture for the cone $K'=\Conv(\{0_{\RR^{d+1}}\}\cup(\{1_\RR\}\times K))$ by replacing each fine $d$-simplex $T$ of $F$ with a prism $T\times[0,1]$, that can be divided into $d+1$ fine simplices, where all additional edges are oriented upwards (from $z=0$ to $z=1$, here $z$ is the last coordinate). Then we get a fine structure on the cylinder $M\times[0,1]$ that can be turned into a $d+1$-torus by gluing the top with the bottom.
\item If the conjecture is true for some polytope $K$, then for any integral point $p\in K\cap\ZZ^d$, the conjecture is true for the polytope $K'=K-p$. The fine structure $\widetilde F'$ that corresponds to the polytope $K'$ can be obtained by subtracting an appropriate calibration 1-form from the fine structure $\widetilde F$ that corresponds to the polytope $K$.
\end{itemize}

\begin{comment} Proving Conjecture~\ref{conj:discretization_normed} in the simple case $K=[0,1]^d$ would be interesting because it would imply that the $d$-dimensional compact Riemannian manifolds can be discretized. Indeed, by the general constructions mentioned above, from the case $K=[0,1]^d$ we would obtain the case $K'=[-1,1]^d$. This means that we would be able to discretize the $d$-dimensional $\ell_1$ metric using fine $d$-simplices. But every compact Riemannian $d$-manifold $M$ can be discretized using $d$-cubes with $\ell_1$ metric. The proof of this fact involves embedding the manifold $M$ in $\RR^k$ preserving lengths (which is possible by the Nash-Kuiper theorem), performing a polyhedral approximation and intersecting the manifold $M$ with random planes.
\end{comment}


\subsection{Riemannian rigidity and random surfaces}\label{subsect:riemannian-rigidity} Can we use our discrete square-celled (or fine) surfaces to model Riemannian surfaces? Of course we can since Riemannian metrics are a particular case of self-reverse Finsler metrics. 
However, Riemannian metrics have the crucial property that a simple Riemannian disk $M$ is determined by its boundary distances \cite{pestov2005two}. Can we computationally determine the surface using discrete methods? If we somehow 
approximate the boundary of $M$ and boundary distances by a discrete cycle $C=C_{2n}$ with a discrete disk-like boundary distance $d$, then we can construct a tight square-celled disk that fills $C$ and has boundary distance $d$. But there are many possibilities; they look approximately Finslerian at the large scale. How can we choose a filling that is near the unique Riemannian possibility? 
A very simple solution would be obtained if most fillings were concentrated near the Riemannian one. Then we would be able to find an approximately Riemannian filling by randomizing any known tight filling.

This possibility can be tested with a computer experiment (that I have not done yet). Take a simple Riemannian disk $M$ where geodesics and their crossings and lengths can be computed, and use random geodesics to form a tight wallsystem $W$ that approximates the Riemannian metric. Randomize this wallsystem using $\Ste_3$ moves. After many moves, one should obtain a random wallsystem $W'$ with approximately uniform distribution among all wallsystems that have the same boundary distances as $W$. To see if $W'$ is still near $W$ 
one can test, for each pair of walls $w_0,w_1\in W$ that cross, whether the corresponding walls of $w_0',w_1'\in W'$ cross at approximately the same point (that is, at the same distance from the startpoint of $w_0$ and $w_0'$).\footnote{More precisely, a possible distance between the two homotopic wallsystems $W,W'$ could be the maximum over $i$, $j$ of the number of triangles of $W'$ that have sides on $w_i'$ and $w_j'$ and are flipped with respect to the corresponding triangles of $W$. For each fixed simple Riemannian disk $M$, when the number of random walls $n$ goes to infinity, the distances between $W$ and $W'$ divided by $n$ should converge to $0$ in probability.}





Note that lozenge tillings of a simply-connected finite plane region $M$ can also be considered minimum square-celled disk fillings of certain boundary distances. 
In the scaling limit, random lozenge tillings do concentrate near a certain surface in 3-dimensional $\ell_1$ space (\cite{cohn2001variational,cohn1998shape}; see also \cite[Thm. 9]{kenyon2009lectures}). It may happen that the method of maximum entropy used for proving this fact can be extended to show that discretized simple Riemannian disks do not change much when they are subjected to random $R_3$ moves. In more detail, it may be true that most discrete tight square-celled surfaces with given boundary distances concentrate near a Finsler surface $(M,F)$ that maximizes the entropy integral \[\operatorname{Ent}(M,F)=\int_M\operatorname{ent}(F_x)\,\diff\Area_\uHT(x),\] where $\operatorname{ent}(\|-\|)$ is an entropy density that is defined for any 2-dimensional self-reverse norm $\|-\|$ and attains its maximum value if and only if $\|-\|$ is an Euclidean norm.

\printbibliography
\end{document}